\pdfoutput=1
\documentclass{article}

\usepackage[authoryear]{natbib}
\usepackage{enumerate}
\usepackage{fullpage}
\usepackage{graphicx}

\graphicspath{ {fig/} }
\usepackage{subcaption}

\usepackage{pig}

\title{Optimality and Sub-optimality of PCA \\for Spiked Random Matrices and Synchronization}

\usepackage{authblk}
\author[1]{Amelia Perry%
\footnote{The first two authors contributed equally.}%
\thanks{Email: {\tt ameliaperry@mit.edu}. This work is supported in part by NSF CAREER Award CCF-1453261 and a grant from the MIT NEC Corporation.}%
}
\newcommand{\firstauthmark}{\footnotemark[1]}
\author[1]{Alexander S.\ Wein%
\protect\firstauthmark%
\thanks{Email: {\tt awein@mit.edu}. This research was conducted with Government support under and awarded by DoD, Air Force Office of Scientific Research, National Defense Science and Engineering Graduate (NDSEG) Fellowship, 32 CFR 168a.}%
}
\author[1,3]{Afonso S.\ Bandeira%
\thanks{Email: {\tt bandeira@cims.nyu.edu}. A.S.B.\ was supported by NSF Grant DMS-1317308.
 Part of this work was done while A.S.B.\ was with the Department of Mathematics at the Massachusetts Institute of Technology.}%
}
\author[1,2]{Ankur Moitra%
\thanks{Email: {\tt moitra@mit.edu}. This work is supported in part by NSF CAREER Award CCF-1453261, NSF Large CCF-1565235, a grant from the MIT NEC Corporation and a Google Faculty Research Award.}%
}
\affil[1]{Department of Mathematics, Massachusetts Institute of Technology}
\affil[2]{Computer Science and Artificial Intelligence Lab, Massachusetts Institute of Technology}
\affil[3]{Department of Mathematics and Center for Data Science, Courant Institute of Mathematical Sciences, New York University}

\begin{document}
\maketitle

\vspace{-1em}
\begin{abstract}

A central problem of random matrix theory is to understand the eigenvalues of `spiked' or `deformed' random matrix models, in which a prominent eigenvector (or `spike') is planted into a random matrix. These distributions form natural statistical models for principal component analysis (PCA) problems throughout the sciences. 
\citet*{bbp} showed that the spiked Wishart ensemble exhibits a sharp phase transition asymptotically: when the signal strength is above a critical threshold, it is possible to detect the presence of a spike based on the top eigenvalue, and below the threshold the top eigenvalue provides no information. Subsequently, sharp spectral phase transitions have been proven in many other random matrix models. Such results form the basis of our understanding of when PCA can detect a low-rank signal in the presence of noise, and how well it can estimate it.

However, not all the information about the spike is necessarily contained in the spectrum. We study the fundamental limitations of statistical methods, including non-spectral ones. Our results include:

\begin{itemize}

\item For the Gaussian Wigner ensemble, we show that PCA achieves the optimal detection threshold for a variety of benign priors for the spike. We extend previous work on the spherically symmetric and \iid Rademacher priors through an elementary, unified analysis.

\item For any non-Gaussian Wigner ensemble, we show that PCA is always suboptimal for detection. However, a variant of PCA achieves the optimal threshold (for benign priors) by pre-transforming the matrix entries according to a carefully designed function.
This approach has been stated before, based on a linearization of approximate message passing, and we give a rigorous and general analysis.

\item Finally, for both the Gaussian Wishart ensemble and various synchronization problems over groups, we show that computationally inefficient procedures can work below the threshold where PCA succeeds, whereas no known efficient algorithm achieves this. This conjectural gap between what is statistically possible and what can be done efficiently remains an interesting open question.

\end{itemize}

\noindent Our results are based on several new tools for establishing that two matrix distributions are contiguous. In some cases, we establish non-asymptotic bounds for hypothesis testing, and also transfer our results to the corresponding estimation problems.

\end{abstract}

\section{Introduction}
\label{sec:intro}

One of the most common ways of analyzing a collection of data is to extract top eigenvectors that represent directions of largest variance, often referred to as principal component analysis (PCA). Starting from the work of Karl Pearson, this technique has been a mainstay in statistics and throughout the sciences for more than a century. For instance, genome-wide association studies construct a correlation matrix of expression levels, whereby PCA is able to identify collections of genes that work together. PCA is also used in economics to extract macroeconomic trends and to predict yields and volatility \citep{macro1,macro2}, and in network science to find well-connected groups of people to identify communities \citep{mcsherry}. More broadly, it underlies much of exploratory data analysis, dimensionality reduction and visualization. 

Classical random matrix theory provides a suite of tools to characterize the behavior of the eigenvalues of various random matrix models in high-dimensional settings. Nevertheless, most of these works can be thought of as focusing on a pure noise-model \citep{AGZ-book,BS-book,Tao-book} where there is not necessarily any low-rank structure to extract. A direction initiated by \cite{J-spk} has brought this powerful theory closer to statistical questions by introducing \emph{spiked models} that are of the form ``signal + noise.'' Such models have yielded fundamental new insights on the behaviors of several methods such as principal component analysis (PCA) \citep{JL-sparse-pca,Paul,Nadler}, sparse PCA \citep{AW-sparse-pca,VL-sparse,BR-opt,Ma-sparse,SSM-sparse,CMW-sparse,Birnbaum-sparse,DM-sparse-pca,KNV-sparse-pca}, and synchronization algorithms
\citep{sin11,boumal2014cramer,angular-tightness,boumal}. More precisely, given a true signal in the form of an $n$-dimensional unit vector $x$ called the \emph{spike}, we can define three natural spiked random matrix ensembles as follows:
\begin{itemize}
\item Spiked Wigner: observe $Y = \lambda x x^\top + \frac{1}{\sqrt{n}} W$, where $W$ is an $n \times n$ random symmetric matrix with entries drawn \iid (up to symmetry) from a fixed distribution of mean $0$ and variance $1$.
\item Spiked (Gaussian) Wishart: observe $Y = X X^\top$, where $X$ is an $n \times N$ matrix with columns drawn independently from $\cN(0,I_n + \beta x x^\top)$, in the high-dimensional setting where the sample count $N$ and dimension $n$ scale proportionally as $N \gamma \approx n$.
\item $\ZZ/L$ synchronization: with $x$ instead valued entrywise in the complex $L$th roots of unity, we observe an $n \times n$ Hermitian matrix $Y$ with independent entries (up to conjugate symmetry) as follows: $Y_{uv} = x_u / x_v$ with some low probability $\tilde{p}/\sqrt{n}$, and otherwise $Y_{uv}$ is a random $L$th root of unity.
\end{itemize}
\noindent We allow the spike $x$ to be drawn from an arbitrary but known prior, to encompass structured problems such as sparse PCA. These models together capture a rich collection of settings where noisy pairwise measurements are available and we wish to detect or estimate $x$. We will also consider generalizations of $\ZZ/L$ synchronization to arbitrary compact groups, where even defining a meaningful noise model is challenging; we present a framework to do this founded on representation theory.

We will refer to the parameters $\beta$, $\lambda$, or $\tilde{p}$ as the signal-to-noise ratio. In each of these models, we study the following statistical questions:
\begin{itemize}
\item \emph{Detection}: For what values of the signal-to-noise ratio is it information-theoretically possible to reliably distinguish (with probability $1-o(1)$ as  $n \to \infty$) between a random matrix drawn from the spiked distribution and one drawn from the corresponding unspiked distribution?

\item \emph{Recovery}: Is there any estimator that achieves a correlation with the ground truth $x$ that remains bounded away from zero as $n \to \infty$?
\end{itemize}

\noindent We will primarily study the detection problem; this type of problem has previously been explored throughout various statistical models \citep{Donoho-Jin,CJL,Ingster-detection,ACCD,ACCP,aCBL,BI-detect,SN-size,SN-max}.

The random matrix models above all enjoy a sharp characterization of the performance of PCA through random matrix theory. 
In the complex Wishart case, the seminal work of \citet*{bbp} 
showed that when $\beta > \sqrt{\gamma}$ an isolated eigenvalue emerges from the Marchenko--Pastur-distributed bulk. Later \citet*{baik-silverstein} established this result in the real Wishart case. Many other such sharp phase transitions are known. In the Wigner case, the top eigenvalue separates from the semicircular bulk when $\lambda > 1$ \citep{peche,fp,cdf,wig-spk}, and in synchronization, the threshold is $\tilde p > 1$ \citep{sin11}. Each of these results establishes a sharp threshold at which PCA is able to solve the detection problem for the respective spiked random matrix model. Moreover, it is known that above this threshold, the top eigenvector correlates nontrivially with $x$, while the correlation concentrates about zero below the threshold. We will refer to these results collectively as the {\em spectral threshold}. Despite a great deal of research on the spectral properties of spiked random matrix models, much less is known about the more general statistical question: Can any statistical procedure detect the presence of a spike below the threshold where PCA succeeds? Our main goal in this paper is to address this question in each of the models above, and as we will see, the answer varies considerably across them. Our results shed new light on how much of the accessible information about $x$ is {\em not} captured by the spectrum. 

Several recent works have examined this question. \citet{sphericity} studied the spiked Wishart model where $x$ is a unit vector that is chosen uniformly at random from the unit sphere. Such a model is symmetric under rotations, which implies that without loss of generality any hypothesis test can also be taken to be symmetric, and depend only on the eigenvalues. \citet{sphericity} shows that there is no test to reliably detect the presence of a spike below the spectral threshold, and complements this by showing that there are some tests that can nevertheless distinguish better than random guessing. Even more recent work \citep{all-eigs,rare-weak} elaborates on this point in other spiked models. Similar results were established in the Gaussian Wigner case by \citet{mrz-sphere}, through techniques similar to those of the present paper, which are not fundamentally limited to spherically symmetric models; indeed, these techniques were applied to sparse PCA in \citet{bmvx}.

In another line of work, several papers have studied recovery in spiked random matrix models through approximate message passing \citep{amp-cs,bm,jm} and various other tools originating from statistical physics. These results span sparse PCA \citep{sparse-pca-amp,phase-sparse-pca}, nonnegative PCA \citep{nonneg-pca}, cone-constrained PCA \citep{cone-constrained}, and general structured PCA \citep{RF-amp,LKZ-amp}. Rigorous results are known, for instance, in Wigner models where the distribution of $x$ is \iid Rademacher\footnote{Uniformly distributed on $\pm 1$} or sparse Rachemacher\footnote{Distributed with some mass on $0$ and the rest equally on $\pm 1$} \citep{dam,sparse-pca-amp,mi,mi-proof}. Methods based on approximate message passing often exhibit the same threshold as PCA but above the threshold they obtain better (and sometimes even information-theoretically optimal) estimates of the spike. Such results are attractive, but often times the method of analysis addresses the recovery problem only, and is limited to models where the coordinates of $x$ are independent.

We will primarily study the detection problem (following \citet{Donoho-Jin,CJL,Ingster-detection,ACCD,ACCP,aCBL,BI-detect,SN-size,SN-max}). We develop a number of general purpose tools for proving both upper and lower bounds on detection. We defer the precise statement of our results in each model to their respective sections, but for now we highlight some of our main results:

\begin{itemize}

\item In the Gaussian Wigner model, we show that the spectral threshold $\lambda=1$ is optimal for priors such as the uniform prior on the unit sphere (Theorem~\ref{thm:sphere-prior}), the \iid Rademacher prior (Theorem~\ref{thm:pmone-wigner}), and any prior with a sufficient sub-Gaussian bound (Theorem~\ref{thm:subg}). Thus there is no statistical test that solves the detection problem beneath this threshold. We also study sparse Rademacher priors, where we show the spectral threshold is sometimes optimal and sometimes suboptimal depending on the sparsity level (Section~\ref{sec:sparse-rad}). Our results here are similar to those of \citet{mi} and \citet{bmvx}, but extend them by proving non-detection up to the spectral threshold for sufficiently high density.

\item In the general Wigner model where the entries of $W$ are non-Gaussian, we show that the spectral threshold is never optimal (subject to some mild conditions on the distribution of $x$). More precisely, we show that there is a way to exploit the non-Gaussian distribution of the noise, by performing an entrywise transformation on the observed matrix that strictly improves the performance of PCA (Theorem~\ref{thm:nong-upper}). Such a method was described by \citet{LKZ-amp}, and in our work we give a rigorous analysis. Moreover we provide a lower bound which often matches our upper bound and precisely characterizes the information theoretic limits of detection, as parametrized by the distribution of the noise (Theorem~\ref{thm:nongauss-lower}).

\item Recall that in the Wishart setting, PCA is known to be optimal when $x$ is spherically distributed. In contrast, we show that when $x$ is \iid Rademacher distributed the spectral threshold is only sometimes statistically optimal. When $\gamma \leq 1/3$ we prove that there is no statistical test that succeeds beneath the spectral threshold (Proposition~\ref{prop:wishart-z2-spectral-tight}). But when $\gamma \geq 0.698$ and $\beta < 0$ we give a computationally inefficient test that succeeds even when spectral methods fail (Theorem~\ref{thm:wishart-mle-z2}). This exposes a new statistical phase transition phenomenon in the Wishart setting that seems to be previously unexplored. We prove similar results in a wide range of synchronization problems, establishing cases when spectral methods are optimal and other cases where there are computationally inefficient tests that perform better. 

\end{itemize}

All our lower bounds follow a similar pattern and are based on the notion of \emph{contiguity} introduced by \citet{lecam}. On a technical level, we show that a particular second moment is bounded which (as is standard in contiguity arguments) implies that the spiked distribution cannot be reliably distinguished (with $o(1)$ error as $n \to \infty$) from the corresponding unspiked distribution. We develop general tools for controlling the second moment based on large deviations and on sub-Gaussian fluctuations that apply across a range of models and a range of choices for the distribution on $x$. 

While bounds on the second moment do not \emph{a priori} imply anything about the recovery problem, we appeal to a result from \citet{bmvx} to make this connection and show that many of our non-detection results translate to non-recovery results as well. In addition, the value of the second moment yields specific bounds on the tradeoff between type I and type II error (see Proposition~\ref{prop:hyptest}), as illustrated in Figures~\ref{fig:hyp1} and \ref{fig:hyp2}. Although our focus is mainly on the limit as $n \to \infty$, in some cases we compute the exact second moment for finite $n$, resulting in non-asymptotic bounds.

\begin{figure}[!ht]
    \centering
    \begin{minipage}{0.47\textwidth}
    \includegraphics{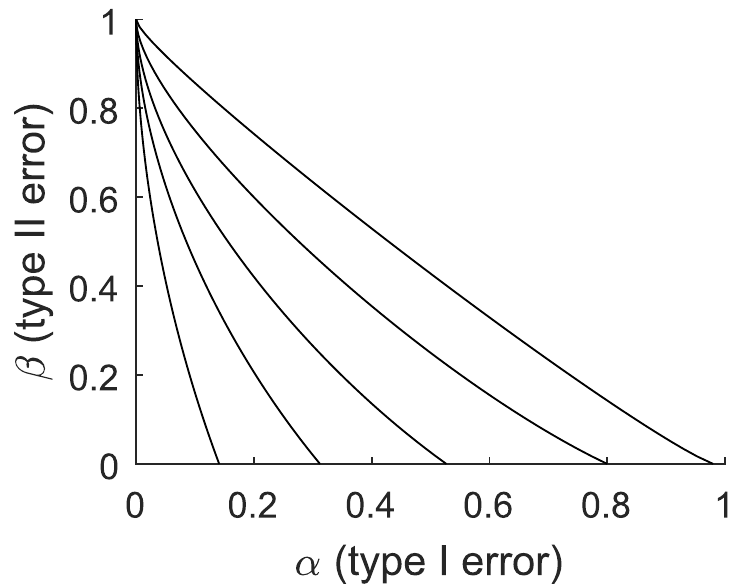}
    \captionof{figure}{Gaussian Wigner model spiked with a uniformly random unit vector. Below each curve, no hypothesis test with the specified type I and type II errors\protect\footnotemark  { }can exist. From left to right: $\lambda = 0.99, 0.95, 0.85, 0.6, 0.2$, in the limit as $n \to \infty$.}
    \label{fig:hyp1}
    \end{minipage}\hfill%
    \begin{minipage}{0.47\textwidth}
    \includegraphics{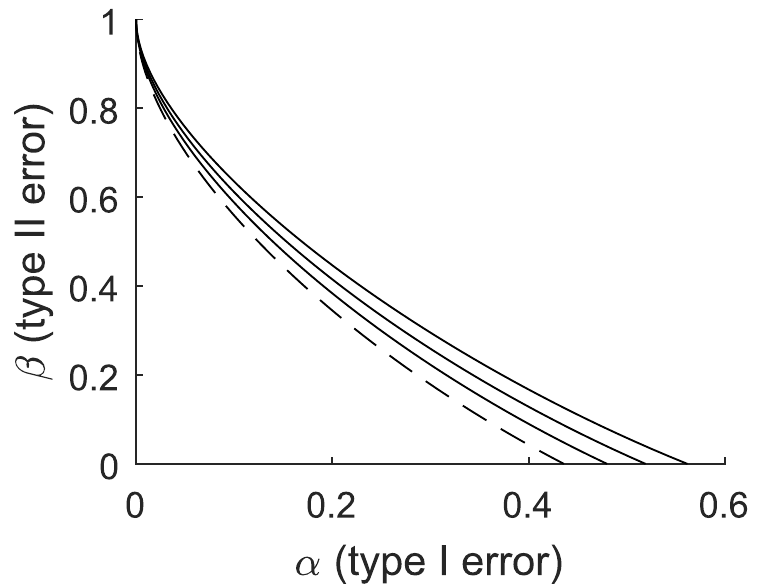}
    \captionof{figure}{Here $n$ varies while $\lambda = 0.9$ is fixed. From left to right: $n = \infty, 75, 25, 10$. This shows the non-asymptotic effectiveness of our methods. For smaller $\lambda$, the curves will appear closer.\\}
    \label{fig:hyp2}
    \end{minipage}
\end{figure}

\footnotetext{Recall that Type I error refers to the probability of reporting a spike when none exists (false positives), while Type II error is the probability of reporting no spike when one does exist (false negatives).}

\subsubsection*{Further related work} 
Above, we reviewed some of the extensive literature on structured PCA problems (e.g. sparse PCA, nonnegative PCA, etc). Synchronization problems are another important family of such problems, motivated by cryo-electron microscopy \citep{singer-shkolnisky}, time synchronization in networks \citep{time-synch}, signals processing \citep{mra}, and many other applications; see e.g.\ \citet{afonso-thesis} for a survey. These are a class of highly symmetric recovery problems, valued in a group such as the cyclic group $\mathbb{Z}/L$, the unit complex numbers $U(1)$, or 3D rotations $SO(3)$. The goal is to recover a collection of group elements from noisy pairwise measurements. The model for synchronization over $\ZZ/L$ described above was introduced by \citet{sin11}, where lower bounds on recovery are presented which we improve upon here. For continuous groups, Cram\'er--Rao bounds on the estimation error are established in \citet{boumal2014cramer}, but few lower bounds that entirely preclude recovery are known. The recovery threshold for $U(1)$ in a Gaussian synchronization model is predicted in \citet{sdp-phase} using techniques from statistical physics; such a Gaussian model was seen also in \citet{angular-tightness,sdp-phase,boumal}. We generalize this line of work by introducing a broad generalization of this Gaussian model, relying on representation theory, and we provide general lower bounds that hold over any compact group.

Finally, our work can be thought of as fitting into an emerging theme in statistics. We indicate several scenarios when PCA is suboptimal but the only known tests that beat it are computationally inefficient. Such computational vs.\ statistical gaps have received considerable recent attention (e.g.\ \citet{Ber-Rig,Ma-Wu}), often in connection with sparsity, but we explore settings here where the difficulty appears in connection with group structure or even the structure of the Rademacher prior. These examples may indicate new classes of problems that demonstrate a statistical price for computational efficiency.

\subsubsection*{Outline}
The rest of this paper is organized as follows. In Section~\ref{sec:contig} we give preliminaries on contiguity. In Section~\ref{sec:gwig} we study the spiked Gaussian Wigner model and in Section~\ref{sec:nong-wig} we study the spiked non-Gaussian Wigner model. In Section~\ref{sec:wish} we study the spiked Wishart model, and in Section~\ref{sec:synch} we study a variety of synchronization problems over compact groups.

\section{Contiguity and the second moment method}
\label{sec:contig}

Contiguity and related ideas will play a crucial role in this paper. To give some background, contiguity was first introduced by 
\citet{lecam} and since then has found many applications throughout probability and statistics. This notion and related tools such as the \emph{small subgraph conditioning method} have been used to establish many fundamental results about random graphs (e.g.\ \citet{rw-ham,janson,1-fact}; see \citet{wor-survey} for a survey). It has also been used to show the impossibility of detecting community structure in certain regimes of the stochastic block model \citep{mns, bmnn}. We will take inspiration from many of these works, in how we go about establishing contiguity. It is formally defined as follows:

\begin{definition}[\citet{lecam}]
Let distributions $P_n$, $Q_n$ be defined on the measurable space $(\Omega_n,\F_n)$. We say that the sequence $P_n$ is \defn{contiguous} to $Q_n$, and write $P_n \contig Q_n$, if for any sequence of events $A_n$, $Q_n(A_n) \to 0 \implies P_n(A_n) \to 0$ as $n \to \infty$.
\end{definition}

\noindent Contiguity implies that the distributions $P_n$ and $Q_n$ cannot be reliably distinguished in the following sense:
\begin{claim}\label{cor:contig}
If $P_n \contig Q_n$ then there is no a statistical test $\mathcal{D}$ that takes a sample from either $P_n$ or $Q_n$ (say each is chosen with probability $\frac{1}{2}$) and correctly outputs which of the two distributions it came from with error probability $o(1)$ as $n \to \infty$.
\end{claim}
\begin{proof}
Suppose that such a test $\mathcal{D}$ exists. Let $A_n$ be the event that $\cD$ outputs `$P_n$.' Since $\cD$ succeeds reliably when the sample comes from $Q_n$, we have $Q_n(A_n) \to 0$ (as $n \to \infty$). By contiguity this means $P_n(A_n) \to 0$. But this contradicts the fact that $\cD$ succeeds reliably when the sample comes from $P_n$.
\end{proof}
\noindent Note that $P_n \contig Q_n$ and $Q_n \contig P_n$ are not the same. Nevertheless either of them implies non-distinguishability. Also, showing that two distributions are contiguous does not rule out the existence of a test that distinguishes between then with constant probability. In fact, for many pairs of contiguous random graph models, such tests do exist. 

Our goal in this paper is to show thresholds below which spiked and unspiked random matrix models are contiguous. We will do this through computing a particular second moment, related to the $\chi^2$-divergence as $1+\chi^2(P_n || Q_n)$, through a form of the second moment method:
\begin{lemma}[see e.g.\ \citet{mrz-sphere,bmvx}]
\label{lem:sec}
Let $\{P_n\}$ and $\{Q_n\}$ be two sequences of probability measures on $(\Omega_n,\F_n)$. If the second moment
$$\Ex_{Q_n} \left[\left( \dd[P_n]{Q_n} \right)^2\right]$$
exists and remains bounded as $n \to \infty$, then $P_n \contig Q_n$.
\end{lemma}

All of the contiguity results in this paper will follow through Lemma~\ref{lem:sec}. The roles of $P_n$ and $Q_n$ are not symmetric, and we will always take $P_n$ to be the spiked distribution and take $Q_n$ to be the unspiked distribution, as the second moment is more tractable to compute in this direction. We include the proof of Lemma~\ref{lem:sec} here for completeness:

\begin{proof}
Let $\{A_n\}$ be a sequence of events. Using Cauchy--Schwarz,
\begin{align*}
P_n(A_n) &= \int_{A_n} \,\dee P_n = \int_{A_n} \dd[P_n]{Q_n} \,\dee Q_n \le \sqrt{\int_{A_n} \left( \dd[P_n]{Q_n} \right)^2 \,\dee Q_n} \;\cdot\; \sqrt{\int_{A_n} \,\dee Q_n} \\
&\le \sqrt{\Ex_{Q_n} \left( \dd[P_n]{Q_n} \right)^2} \;\cdot\; \sqrt{Q_n(A_n)}.
\end{align*}
The first factor on the right-hand side is bounded; so if $Q_n(A_n) \to 0$ as $n \to \infty$, we must also have $P_n(A_n) \to 0$, as desired.
\end{proof}

There will be times when the above second moment is infinite but we are still able to prove contiguity using a modified second moment that conditions on `good' events. This idea is based on \citet{bmnn}.
\begin{lemma}
\label{lem:cond}
Let $\omega_n$ be a `good' event that occurs with probability $1-o(1)$ under $P_n$.
Suppose $P_n$ and $\tilde P_n$ agree within $\omega_n$. If
$$\Ex_{Q_n} \left[\left( \dd[\tilde P_n]{Q_n} \right)^2\right]$$
remains bounded as $n \to \infty$, then $P_n \contig Q_n$.
\end{lemma}
\begin{proof}
By Lemma~\ref{lem:sec} we have $\tilde P_n \contig Q_n$. This implies $P_n \contig Q_n$ because $\tilde P_n(A_n) \to 0$ implies $P_n(A_n) \to 0$ (since $P_n(\omega_n) \to 1$).
\end{proof}

Moreover, given a value of the second moment, we are able to obtain bounds on the tradeoff between type I and type II error in hypothesis testing, which are valid non-asymptotically:
\begin{proposition}\label{prop:hyptest}
Let $\mathcal{D}$ be a distinguisher than takes a sample from either $P$ or $Q$ and outputs `$P$' or `$Q$.' Let $A$ be the event that $\mathcal{D}$ outputs `$P$.' Let $\alpha = Q(A)$ be the probability of type I error, and let $\beta = 1 - P(A)$ be the probability of type II error. Regardless of the distinguisher $\mathcal{D}$, we must have
$$\frac{(1-\beta)^2}{\alpha} + \frac{\beta^2}{(1-\alpha)} \le \Ex_Q \left(\dd[P]{Q}\right)^2,$$
assuming the right-hand side is defined and finite. Furthermore, this is tight in the sense that for any $\alpha,\beta \in (0,1)$ there exist $P,Q,\mathcal{D}$ for which equality holds.
\end{proposition}
\noindent This tradeoff is illustrated in Figures~\ref{fig:hyp1} and~\ref{fig:hyp2} in the introduction. 

\begin{proof}
Let $\bar A$ denote the complement of the event $A$ (defined above).
\begin{align*}
\Ex_Q\left(\dd[P]{Q}\right)^2 &= \int \dd[P]{Q} \,\dee P = \int_A \dd[P]{Q} \,\dee P + \int_{\bar A} \dd[P]{Q} \,\dee P \\
&\ge \frac{\left(\int_A \dee P\right)^2}{\int_A (\dee Q/ \dee P) \,\dee P} + \frac{\left(\int_{\bar A} \dee P\right)^2}{\int_{\bar A} (\dee Q/\dee P) \,\dee P} = \frac{(1-\beta)^2}{\alpha} + \frac{\beta^2}{(1-\alpha)	}
\end{align*}
where the inequality follows from Cauchy--Schwarz. The following example shows tightness: let $Q = \mathrm{Bernoulli}(\alpha)$ and let $P = \mathrm{Bernoulli}(1-\beta)$. On input $0$, $\mathcal{D}$ outputs `$Q$,' and on input $1$, it outputs `$P$.'
\end{proof}

Although contiguity is a statement about non-detection rather than non-recovery, our results also have implications for non-recovery. In general, the detection problem and recovery problem can have different thresholds, but this is due to unnatural counterexamples. In many settings we will be able to obtain non-recovery results by directly appealing to the following result of \citet{bmvx}:

\begin{theorem-nolabel}[\citet{bmvx}, Theorem 4]
Let $P_n$ and $Q_n$ be the spiked and unspiked models $Y = M + W$ and $Y = W$ respectively, where $M$ is a matrix (square or rectangular, with size depending on $n$) drawn from any prior satisfying $\EE[M] = 0$ and $\lim_{n \to \infty}\frac{1}{n}\EE\|M\|_F^2$ exists, and where $W$ is a Gaussian Wigner matrix or an \iid Gaussian matrix. If $\EE_{Q_n}\left(\dd[P_n]{Q_n}\right)^2$ is bounded as $n \to \infty$ then nontrivial recovery is impossible: for any estimator $\hat M = \hat M(Y)$ with $\EE_Y \|\hat M\|_F^2 = \mathcal{O}(n)$, we have that $\liminf_{n \to \infty} \frac{1}{n} \EE_{M,W} \langle M,\hat M \rangle = 0$.
\end{theorem-nolabel}

\noindent Both the spiked Gaussian Wigner model and the positively-spiked ($\beta > 0$) Wishart model\footnote{For the Wishart case, consider the asymmetric $n \times N$ matrix of samples.} fall into this framework. For other models that we consider, non-recovery results do not follow immediately from the above result, but we expect them to be true nonetheless.

\section{Gaussian Wigner models}
\label{sec:gwig}

In this section we establish non-detection results in spiked Gaussian Wigner models. In Section~\ref{sec:gwig-intro} we define the problem and state our main results. In Section~\ref{sec:2nd-comp} we compute the second moment (defined in Lemma~\ref{lem:sec}) of the Gaussian Wigner model. In Section~\ref{sec:spherical} we consider the spherical prior where $x$ is a random unit vector. In Section~\ref{sec:subg-method} we give a general ``sub-Gaussian method'' of analysis for priors with \iid entries, based on sub-Gaussian tail bounds. In Section~\ref{sec:cond-method} we give an improved ``conditioning method'' for \iid priors with finite support. The other subsections contain various examples and applications of these techniques.

\subsection{Main results}\label{sec:gwig-intro}

The spiked Gaussian Wigner model is defined as follows:
\begin{definition}
\label{def:prior}
A \emph{spike prior} is a family of distributions $\cX = \{\cX_n\}$, where $\cX_n$ is a distribution over $\RR^n$. We require our priors to be normalized so that $x^{(n)}$ drawn from $\cX_n$ has $\|x^{(n)}\| \to 1$ (in probability) as $n \to \infty$.
\end{definition}
\noindent We normalize the prior in this way so that the eigenvalue threshold is always $\lambda = 1$.

\begin{definition}
For $\lambda \ge 0$ and a spike prior $\cX$, we define the spiked Gaussian Wigner model $\GWig(\lambda,\cX)$ as follows. We first draw a spike $x \in \RR^n$ from the prior $\cX_n$. Then we reveal
$$Y = \lambda xx^\top + \frac{1}{\sqrt n} W$$
where $W$ is drawn from the $n \times n$ $\GOE$ (Gaussian orthogonal ensemble), i.e.\ $W$ is a random symmetric matrix with off-diagonal entries $\mathcal{N}(0,1)$, diagonal entries $\mathcal{N}(0,2)$, and all entries independent (except for symmetry $W_{ij} = W_{ji}$). We denote the unspiked model ($\lambda = 0$) by $\GWig(0)$.
\end{definition}

It is well known that this model admits the following spectral behavior.
\begin{theorem}[\cite{fp,nong-eigv1}]
Let $Y$ be drawn from $\GWig(\lambda,\cX)$ with any spike prior $\cX$.
\begin{itemize}
\item If $\lambda \le 1$, the top eigenvalue of $Y$ converges almost surely to $2$ as $n \to \infty$, and the top (unit-norm) eigenvector $v$ has trivial correlation with the spike: $\langle v,x \rangle^2 \to 0$ almost surely.
\item If $\lambda > 1$, the top eigenvalue converges almost surely to $\lambda + 1/\lambda > 2$ and $v$ has nontrivial correlation with the spike: $\langle v,x \rangle^2 \to 1 - 1/\lambda^2$ almost surely.
\end{itemize}
\end{theorem}
\noindent Therefore PCA solves the detection and recovery problems precisely when $\lambda > 1$. Our goal is now to investigate whether any method can beat this threshold.

As a starting point for all of our proofs, we compute the second moment of Lemma~\ref{lem:sec}:
\begin{proposition}
\label{prop:2nd-comp}
Let $\lambda \ge 0$ and let $\cX$ be a spike prior. Let $P_n = \GWig_n(\lambda,\cX)$ and $Q_n = \GWig_n(0)$. Let $x$ and $x'$ be independently drawn from $\cX_n$. Then
$$\Ex_{Q_n}\left(\dd[P_n]{Q_n}\right)^2 = \Ex_{x,x'} \exp\left(\frac{n \lambda^2}{2} \langle x,x' \rangle^2\right).$$
\end{proposition}
\noindent We defer the proof of this proposition until Section~\ref{sec:2nd-comp}. 
For specific choices of the prior $\cX$, our goal will be to show that if $\lambda$ is below some critical $\lambda^*_\cX$, this second moment is bounded as $n \to \infty$ (implying that detection is impossible). We will specifically consider the following types of priors.
\begin{definition}
Let $\cXs$ denote the spherical prior: $x$ is a uniformly random unit vector in $\RR^n$.
\end{definition}
\begin{definition}
If $\pi$ is a distribution on $\RR$ with $\EE[\pi] = 0$ and $\mathrm{Var}[\pi] = 1$, let $\mathrm{iid}(\pi)$ denote the spike prior that samples each coordinate of $x$ independently from $\frac{1}{\sqrt n} \pi$.
\end{definition}

We will give two general techniques for showing contiguity for \iid priors. We call the first method the \emph{sub-Gaussian method}, and it is presented in Section~\ref{sec:subg-method}. The idea is that if we can show that the random variable $\pi \pi'$ (product of two independent copies of $\pi$) obeys a particular \emph{sub-Gaussian} condition, then this implies strong tail bounds on $\langle x,x' \rangle$ which can be integrated to show that the second moment is bounded.

The second method we use to show contiguity is called the \emph{conditioning method} and is based on ideas from \citet{bmnn}. This method is presented in Section~\ref{sec:cond-method}. The method only applies to \iid priors for which $\pi$ has finite support, but when $\pi$ does have finite support, the conditioning method is at least as strong (and sometimes strictly stronger) than the sub-Gaussian method. The main idea behind the conditioning method is that in some cases the second moment $\EE_{Q_n}\left(\dd[P_n]{Q_n}\right)^2$ is infinite, but only because of contributions from extremely rare `bad' values for the spike $x$. To fix this, we define the distribution $\tilde P_n$ similarly to $P_n = \GWig_n(\lambda,\cX)$ except it disallows `bad' $x$ values whose empirical distributions of entries differ significantly from $\pi$. We then proceed by computing the modified second moment $\EE_{Q_n}\left(\dd[\tilde P_n]{Q_n}\right)^2$. Appealing to a result of \citet{bmnn}, we find that the behavior of this modified second moment is governed by the solution to a particular optimization problem over matrices.

For certain benign priors, we are able to show contiguity up to the spectral threshold:
\begin{theorem-nolabel}[see Theorems \ref{thm:sphere-prior}, \ref{thm:gauss-prior}, \ref{thm:pmone-wigner}]
Let $\cX$ be one of the following priors:
\begin{itemize}
\item the spherical prior $\cXs$
\item the \iid Gaussian prior $\IID(\cN(0,1))$
\item the \iid Rademacher prior $\IID(\pm 1)$.
\end{itemize}
If $\lambda < 1$ then $\GWig(\lambda,\cX)$ is contiguous to $\GWig(0,\cX)$.
\end{theorem-nolabel}
\noindent Note that each of these results is tight, matching the spectral threshold. (We do not consider the behavior exactly at the critical point $\lambda = 1$.) The result for the Rademacher prior was known previously for the related problem of estimating the spike \citep{dam}.

The proof for the spherical prior (Theorem~\ref{thm:sphere-prior}) involves direct computation of the second moment, yielding an expression in terms of a hypergeometric function for which asymptotics are known. The proof for the Gaussian prior (Theorem~\ref{thm:gauss-prior}) is by comparison to the spherical prior. The proof for the Rademacher prior (Theorem~\ref{thm:pmone-wigner}) uses the sub-Gaussian method.

Not all priors are as well behaved as those above. In Section~\ref{sec:sparse-rad} we apply our techniques to the sparse Rademacher prior (defined later) and numerically compute bounds on the threshold where contiguity occurs. Finally, we show that regardless of the prior, the distribution of \emph{eigenvalues} in the spiked model is contiguous to that of the unspiked model for all $\lambda < 1$. This means that no eigenvalue-based test can distinguish the models below the $\lambda = 1$ threshold, even though there are other tests that can in the sparse Rademacher model \citep{mi,mi-proof,bmvx}.

\subsection{Second moment computation}
\label{sec:2nd-comp}

We begin by computing the second moment $\Ex_{Q_n} \left[\left(\dd[P_n]{Q_n}\right)^2\right]$ where $P_n = \GWig_n(\lambda,\cX)$ and $Q_n = \GWig_n(0)$.
First we simplify the likelihood ratio:
\begin{align*}
\dd[P_n]{Q_n} &= \frac{\Ex_{x \sim \cX_n} \exp(-\frac{n}{4} \langle Y-\lambda x x^\top, Y-\lambda x x^\top \rangle)}{\exp(-\frac{n}{4} \langle Y,Y \rangle)} \\
&= \Ex_{x \sim \cX_n} \exp\left( \frac{\lambda n}{2} \langle Y, x x^\top \rangle - \frac{\lambda^2 n}{4} \langle x x^\top, x x^\top \rangle \right).
\end{align*}
Now passing to the second moment:
\begin{align*}
\Ex_{{Q}_n}\left(\dd[P_n]{Q_n}\right)^2
&= \Ex_{x,x' \sim \cX} \Ex_{Y \sim Q_n} \exp\left( \frac{\lambda n}{2} \langle Y, x x^\top + x' x'^\top \rangle - \frac{\lambda^2 n}{4} \left( \langle x x^\top, x x^\top \rangle + \langle x' x'^\top, x' x'^\top \rangle \right) \right),
\intertext{
where $x$ and $x'$ are drawn independently from $\mathcal{X}_n$. Using the Gaussian moment-generating function:}
&= \Ex_{x,x'} \exp\left( \frac{\lambda^2 n}{4} \langle x x^\top + x' x'^\top, x x^\top + x' x'^\top \rangle - \frac{\lambda^2 n}{4} \left( \langle x x^\top, x x^\top \rangle + \langle x' x'^\top, x' x'^\top \rangle \right) \right) \\
&= \Ex_{x,x'} \exp\left( \frac{\lambda^2 n}{2} \langle x,x' \rangle^2 \right).
\end{align*}
This is as far as we can take the computation without specializing to a particular choice of the prior $\mathcal{X}_n$.
\begin{repproposition}{prop:2nd-comp}
Let $\lambda \ge 0$ and let $\cX$ be a spike prior. Let $P_n = \GWig_n(\lambda,\cX)$ and $Q_n = \GWig_n(0)$. Let $x$ and $x'$ be independently drawn from $\cX_n$. Then
$$\Ex_{Q_n}\left(\dd[P_n]{Q_n}\right)^2 = \Ex_{x,x'} \exp\left(\frac{\lambda^2 n}{2} \langle x,x' \rangle^2\right).$$
\end{repproposition}
\noindent Recall that by Lemma~\ref{lem:sec} we have contiguity provided that this second moment is bounded. In the following subsections we will control this quantity for some specific choices of the prior $\cX_n$.

\subsection{Application: the spherical prior}
\label{sec:spherical}

We now begin specializing the above to various choices for the prior $\cX$. We begin with the simple spherical prior $\cXs$ where $x$ is a uniform random unit vector in $\mathbb{R}^n$. Our methods for this case will be specialized to the spherical prior, but we will give more general approaches in the following sections. Note that the non-detection result for this prior has previously appeared in \citet{mrz-sphere}; we give an alternative proof together with non-asymptotic hypothesis testing bounds.

\begin{theorem}\label{thm:sphere-prior} 
Consider the spherical prior $\cXs$. If $\lambda < 1$ then $\GWig(\lambda,\cXs)$ is contiguous to $\GWig(0)$.
\end{theorem}

\noindent Recall that this matches the spectral threshold $\lambda = 1$, above which the spiked and unspiked models can be reliably distinguished via the top eigenvalue.

\begin{proof}[Proof (sketch)]
Exploiting symmetry, the second moment is identified as ${}_1 F_1( 1/2 ; n/2 ; \lambda^2 n / 2)$, which tends to $(1-\lambda^2)^{-1/2}$ as $n \to \infty$ when $\lambda < 1$. The full proof is deferred to Appendix~\ref{app:confluent}.
\end{proof}
\noindent These non-asymptotic and asymptotic second moments yield the hypothesis testing lower bounds in Figures~\ref{fig:hyp1} and~\ref{fig:hyp2}, through Proposition~\ref{prop:hyptest}.

\subsection{The sub-Gaussian method}
\label{sec:subg-method}

In this section we give a method for controlling the quantity $\EE_{x,x'} \exp\left(\frac{n \lambda^2}{2} \langle x,x' \rangle^2\right)$ in the case where the prior $\mathcal{X} = \IID(\pi)$ draws each entry of $x$ independently from $\frac{1}{\sqrt n} \pi$ for some distribution $\pi$ satisfying $\EE[\pi] = 0$ and $\mathrm{Var}[\pi] = 1$. (Note that this ensures that $x$ will have approximately unit norm.) The method of this section is based on sub-Gaussian tail bounds. We will need the concept of a sub-Gaussian random variable.
\begin{definition}\label{def:subg}
We say that a real-valued random variable $X$ is \emph{sub-Gaussian with variance proxy} $\sigma^2$ if $\EE[X] = 0$ and
$$\EE \exp(t X) \le \exp\left(\frac{1}{2} \sigma^2 t^2\right)$$
for all $t \in \mathbb{R}$.
\end{definition}
\noindent This condition says that the moment-generating function of $X$ is bounded by that of $\cN(0,\sigma^2)$. In particular, $\cN(0,\sigma^2)$ is sub-Gaussian with variance proxy equal to its variance $\sigma^2$. One can think of the sub-Gaussian condition as requiring the tails of a distribution to be smaller than that of a Gaussian $\cN(0,\sigma^2)$.

The main result of this section is the following.
\begin{theorem}[Sub-Gaussian method]
\label{thm:subg}
Let $\cX = \mathrm{iid}(\pi)$ for some distribution $\pi$ on $\RR$. Let $P_n = \GWig_n(\lambda,\cX)$ and $Q_n = \GWig_n(0)$. Suppose $\pi\pi'$ (product of two independent copies) is sub-Gaussian with variance proxy $\sigma^2$. If $\lambda < \frac{1}{\sigma}$ then
$$\lim_{n \to \infty} \Ex_{Q_n}\left(\dd[P_n]{Q_n}\right)^2 = (1-\lambda^2)^{-1/2} < \infty$$
and so $P_n \contig Q_n$.
\end{theorem}
\noindent Note that since the variance proxy can never be smaller than the variance, we must have $\frac{1}{\sigma} \le 1$. If $\sigma = 1$ then Theorem~\ref{thm:subg} gives a tight result, matching the spectral threshold. The idea of the proof is that sub-Gaussianity implies tail bounds (via the standard Chernoff bound argument), which can be used to show that the second moment is bounded.
\begin{proof}
This proof follows a similar idea to the proof of Lemma~5.5 in \citet*{mns}. By the central limit theorem, $\sqrt{n} \langle x,x' \rangle$ converges in distribution to a Gaussian:
$$\sqrt{n} \langle x,x' \rangle = \sqrt{n} \sum_{i=1}^n x_i x_i' = \sqrt{n} \sum_{i=1}^n \frac{\pi_i}{\sqrt n} \cdot \frac{\pi_i'}{\sqrt n} = \frac{1}{\sqrt n} \sum_{i=1}^n \pi_i\pi_i' \xrightarrow{d} \mathcal{N}(0,1)$$
since $\mathrm{Var}[\pi\pi'] = 1$. By the continuous mapping theorem applied to $g(z) = \exp\left(\frac{\lambda^2}{2}z^2\right)$, we also get the convergence in distribution
$$\exp\left(\frac{n \lambda^2}{2} \langle x, x' \rangle^2\right) \xrightarrow{d} \exp\left( \frac{\lambda^2}{2} \chi_1^2 \right)$$
In order for this convergence in distribution to imply the convergence 
$$\mathbb{E}\exp\left(\frac{n \lambda^2}{2} \langle x, x' \rangle^2\right) \rightarrow \mathbb{E}\exp\left( \frac{\lambda^2}{2} \chi_1^2 \right)$$
that we want, we need to show that the sequence $\exp\left(\frac{n \lambda^2}{2} \langle x, x' \rangle^2\right)$ is uniformly integrable. We will show this provided $\lambda < \frac{1}{\sigma}$. The desired result then follows using the chi-squared moment-generating function:
$$\mathbb{E}\exp\left( \frac{\lambda^2}{2} \chi_1^2 \right) = (1-\lambda^2)^{-1/2}$$
which is finite for $\lambda < 1$.

To complete the proof we need to show uniform integrability of the sequence $\exp\left(\frac{n \lambda^2}{2} \langle x, x' \rangle^2\right)$. Since $\pi\pi'$ is sub-Gaussian with variance proxy $\sigma^2$, it follows that $\sum_{i=1}^n \pi\pi_i'$ is sub-Gaussian with variance proxy $n \sigma^2$ and so we have the sub-Gaussian tail bound
$$\mathbb{P}\left[\sum_{i=1}^n \pi_i\pi_i' > t\right] \le \exp\left(-\frac{t^2}{2n\sigma^2}\right).$$
To show uniform integrability,
\begin{align*}
\mathbb{P}\left[\exp\left(\frac{n \lambda^2}{2} \langle x, x' \rangle^2\right) \ge M\right] &= \mathbb{P}\left[\langle x,x' \rangle \ge \sqrt{\frac{2 \log M}{n\lambda^2}}\right] = \mathbb{P}\left[\sum_{i=1}^n \pi_i\pi_i' \ge \sqrt{\frac{2n \log M}{\lambda^2}}\right] \\
&\le \exp\left(-\frac{1}{2n\sigma^2} \frac{2n \log M}{\lambda^2}\right)
= M^{-1/(\lambda^2 \sigma^2)}
\end{align*}
which is integrable near $\infty$ (uniformly in $n$) provided $\lambda^2 \sigma^2 < 1$, i.e.\ $\lambda < \frac{1}{\sigma}$.
\end{proof}

We remark that if we only want to show that the second moment is bounded (and not find the limit value), we only need the uniform integrability step because we can control the expectation of $\exp\left(\frac{n \lambda^2}{2} \langle x, x' \rangle^2\right)$ by integrating a tail bound; see the proof of Theorem~\ref{thm:synch-subg} for a related example.
We also note that one could in principle strengthen the sub-Gaussian method by using the Chernoff bound (Cram\'er's theorem on large deviations) in place of the sub-Gaussian tail bound. However, this adds additional complication (for instance, one needs to compute the Legendre transform) and does not actually seem to improve any of our results.

\subsection{\texorpdfstring{Application: the Rademacher prior and $\mathbb{Z}/2$ synchronization}{Application: the Rademacher prior and Z/2 synchronization}}
\label{sec:z2}

In this subsection, we consider the special case when $\pi$ is a Rademacher random variable, i.e.\ uniform on $\{-1,+1\}$. We abbreviate this prior as $\IID(\pm 1)$. This case of the Gaussian Wigner model has been studied by \citet{sdp-phase} as a Gaussian model for $\mathbb{Z}/2$ synchronization \citep{z2-toh,z2}. It has also been studied as a Gaussian variant of the community detection problem in \citet{dam}, where it is shown that the spectral threshold $\lambda = 1$ is precisely the threshold above which nontrivial recovery of the signal is possible. We show contiguity below this $\lambda = 1$ threshold (which, recall, is not implied by non-recovery).

\begin{theorem}\label{thm:pmone-wigner}
If $\lambda < 1$ then $\GWig(\lambda,\IID(\pm 1)) \contig \GWig(0)$.
\end{theorem}
\begin{proof}
Recall Hoeffding's Lemma: if $X \in [a,b]$ is a bounded random variable then
$$\EE[\exp(tX)] \le \exp\left(\frac{1}{8}t^2(a-b)^2\right).$$
This implies that the Rademacher random variable is sub-Gaussian with variance proxy $1$. The result now follows from Theorem~\ref{thm:subg}.
\end{proof}

Note that the Rademacher prior and the spherical prior both have the same threshold: $\lambda = 1$. A matching upper bound in both cases is PCA (top eigenvalue). Perhaps it is surprising that PCA is optimal for the $\pm 1$ case because this suggests that there is no way to exploit the $\pm 1$ structure. However, PCA is only optimal in terms of the threshold and not in terms of error in recovering the spike once $\lambda > 1$. The optimal algorithm for minimizing mean squared error is the AMP (approximate message passing) algorithm of \citet{dam}.

\subsection{The conditioning method}
\label{sec:cond-method}

In this subsection, we give an alternative to the sub-Gaussian method that can give tighter results in some cases. We assume again that the prior $\mathcal{X} = \IID(\pi)$ draws each entry of $x$ independently from $\frac{1}{\sqrt n} \mathcal{\pi}$ where $\mathcal{\pi}$ is mean-zero and unit-variance, but now we also require that $\mathcal{\pi}$ has finite support.

The argument that we will use is based on \citet{bmnn}, in particular their Proposition~5. Suppose $\omega_n$ is a set of `good' $x$ values so that $x \in \omega_n$ with probability $1-o(1)$. Let $P_n = \GWig_n(\lambda,\cX)$ and let $Q_n = \GWig_n(0)$. Let $\tilde\cX$ be the prior that draws $x$ from $\cX$, outputs $x$ if $x \in \omega_n$, and outputs the zero vector otherwise. Let $\tilde P_n = \GWig_n(\lambda,\tilde \cX)$. Our goal is to show $\tilde{P}_n \contig {Q}_n$,  from which it follows that ${P}_n \contig {Q}_n$ (see Lemma~\ref{lem:cond}). 
In our case, the bad events are when the empirical distribution of $x$ differs significantly from $\pi$, i.e.\ $x$ has atypical proportions of entries.
If we let $\Omega_n$ be the event that $x$ and $x'$ are both in $\omega_n$, our second moment becomes
$$\Ex_{Q_n}\left(\dd[\tilde{{P}}_n]{{Q}_n}\right)^2 = \Ex_{\tilde x,\tilde x' \sim \tilde \cX}\left[\exp\left(\frac{n \lambda^2}{2} \langle \tilde x,\tilde x' \rangle^2\right)\right] = \Ex_{x,x' \sim \cX}\left[\one_{\Omega_n}\exp\left(\frac{n \lambda^2}{2} \langle x,x' \rangle^2\right)\right] + o(1).$$

Let $\Sigma \subseteq \mathbb{R}$ (a finite set) be the support of $\pi$, and let $s = |\Sigma|$. We will index $\Sigma$ by $[s] = \{1,2,\ldots,s\}$ and identify $\pi$ with the vector of probabilities $\pi \in \mathbb{R}^s$. For $a,b \in \Sigma$, let $N_{ab}$ denote the number of indices $i$ for which $x_i = \frac{a}{\sqrt n}$ and $x'_i = \frac{b}{\sqrt n}$ (recall $x_i$ is drawn from $\frac{1}{\sqrt n}\pi$). Note that $N$ follows a multinomial distribution with $n$ trials, $s^2$ outcomes, and with probabilities given by $\bar\alpha = \pi \pi^\top \in \mathbb{R}^{s \times s}$. We have
$$\frac{n\lambda^2}{2}\langle x,x' \rangle^2 = \frac{\lambda^2}{2n}\left(\sum_{a,b \in \Sigma} a b N_{ab}\right)^2 = \frac{\lambda^2}{2n}\sum_{a,b,a',b'} aba'b' N_{ab} N_{a'b'} = \frac{1}{n}N^\top A N$$
where $A$ is the $s^2 \times s^2$ matrix $A_{ab,a'b'} = \frac{\lambda^2}{2} aba'b'$, and the quadratic form $N^\top A N$ is computed by treating $N$ as a vector of length $s^2$.

We are now in a position to apply Proposition~5 from \citet{bmnn}. Define $Y = (N - n \bar\alpha)/\sqrt n$. Let $\Omega_n$ be the event defined in Appendix~A of \citet{bmnn}, which enforces that the empirical distributions of $x$ and $x'$ are close to $\pi$ (in a specific sense).

Note that $\bar\alpha$ (treated as a vector of length $s^2$) is in the kernel of $A$ because $\pi$ is mean-zero: the inner product between $\bar\alpha$ and the $(a,b)$ row of $A$ is
$$\sum_{a',b'} A_{ab,a'b'} \bar\alpha_{a'b'} = \frac{\lambda^2}{2} \sum_{a',b'} aba'b'\pi_{a'}\pi_{b'} = \frac{\lambda^2}{2} \,ab\left(\sum_{a'} a' \pi_{a'}\right)\left(\sum_{b'} b' \pi_{b'}\right) = 0.$$
Therefore we have $\frac{1}{n}N^\top A N = Y^\top A Y$ and so we can write our second moment as $\EE[\one_{\Omega_n} \exp(Y^\top A Y)] + o(1)$.

Let $\Delta_{s^2}(\pi)$ denote the set of nonnegative vectors $\alpha \in \mathbb{R}^{s^2}$ with row- and column-sums prescribed by $\pi$, i.e.\ treating $\alpha$ as an $s \times s$ matrix, we have (for all $i$) that row $i$ and column $i$ of $\alpha$ each sum to $\pi_i$. Let $D(u,v)$ denote the KL divergence between two vectors: $D(u,v) = \sum_i u_i \log(u_i/v_i)$. For convenience, we restate Proposition~5 in \citet{bmnn}.

\begin{proposition}[\citet{bmnn} Proposition~5]
\label{prop:nn}
Let $\pi \in \mathbb{R}^s$ be any vector of probabilities. Let $A$ be any $s^2 \times s^2$ matrix. Define $N$, $Y$, $\bar\alpha$, and $\Omega_n$ as above (depending on $\pi$). Let
$$m = \sup_{\alpha \in \Delta_{s^2}(\pi)} \frac{(\alpha-\bar\alpha)^\top A (\alpha-\bar\alpha)}{D(\alpha,\bar\alpha)}.$$
If $m < 1$ then
$$\lim_{n \to \infty} \EE[\one_{\Omega_n} \exp(Y^\top A Y)] = \EE[\exp(Z^\top A Z)] < \infty$$
where $Z \sim \cN(0,\diag(\bar\alpha) - \bar\alpha \bar\alpha^\top)$. Conversely, if $m > 1$ then
$$\lim_{n \to \infty} \EE[\one_{\Omega_n} \exp(Y^\top A Y)] = \infty.$$
\end{proposition}

The intuition behind this matrix optimization problem is the following. The matrix $\alpha$ represents the `type' of a pair of spikes $(x,x')$ in the sense that for any $a,b \in \Sigma$, $\alpha_{ab}$ is the fraction of entries $i$ for which $x_i = a$ and $x'_i = b$. A pair $(x,x')$ of type $\alpha$ yields $\exp(Y^\top A Y) = \exp(n (\alpha - \bar\alpha)^\top A (\alpha - \bar\alpha))$.
The probability (when $x,x' \sim \IID(\pi)$) that a particular type $\alpha$ occurs is asymptotically $\exp(-n D(\alpha,\bar\alpha))$. Due to the exponential scaling, the second moment is dominated by the worst $\alpha$ value: the second moment is unbounded if there is some $\alpha$ such that $(\alpha - \bar\alpha)^\top A (\alpha - \bar\alpha) > D(\alpha,\bar\alpha)$. Rearranging this yields the optimization problem in the theorem. The fact that we are conditioning on `good' values of $x$ (that have close-to-typical proportions of entries) allows us to add the constraint $\alpha \in \Delta_{s^2}(\pi)$. If we were not conditioning, we would have the same optimization problem over $\alpha \in \Delta_{s^2}$ (the simplex of dimension $s^2$), which in some cases gives a worse threshold.

Unfortunately we do not have a good general technique to understand the value of the matrix optimization problem. However, in certain special cases we do. Namely, in Section~\ref{sec:sparse-rad} we show how to use symmetry to reduce the problem to only two variables so that it can be easily solved numerically. We are also able to find a closed form solution for a similar optimization problem when we consider synchronization problems (see Theorems~\ref{thm:toh} and \ref{thm:synch-finite}). In other applications, closed form solutions to related optimization problems have been found \citep{an-chrom,bmnn}.

Applying Proposition~\ref{prop:nn} to our specific choice of $\pi$ and $A$ gives the following.
\begin{theorem}[conditioning method]
\label{thm:cond-method}
Let $\mathcal{X} = \mathrm{iid}(\pi)$ where $\pi$ has finite support $\Sigma \subseteq \mathbb{R}$ with $|\Sigma| = s$. Let $P_n = \GWig_n(\lambda,\cX)$, $\tilde P_n = \GWig_n(\lambda,\tilde \cX)$, and $Q_n = \GWig_n(0)$. Define the $s \times s$ matrix $\beta_{ab} = ab$ for $a,b \in \Sigma$. Let $D(u,v)$ denote the KL divergence between two vectors: $D(u,v) = \sum_i u_i \log(u_i/v_i)$. Identify $\pi$ with the vector of probabilities $\pi \in \RR^\Sigma$, and define $\bar\alpha = \pi \pi^\top$. Let $\Delta_{s^2}(\pi)$ denote the set of $s \times s$ matrices with row- and column-sums prescribed by $\pi$, i.e.\ row $i$ and column $i$ of $\alpha$ each sum to $\pi_i$. Let
$$\lambda^*_\cX = \left[\sup_{\alpha \in \Delta_{s^2}(\pi)} \frac{\langle \alpha,\beta \rangle^2}{2D(\alpha,\bar\alpha)}\right]^{-1/2}.$$
If $\lambda < \lambda^*_\cX$ then
$$\lim_{n \to \infty} \Ex_{Q_n}\left(\dd[\tilde P_n]{Q_n}\right)^2 = \frac{1}{\sqrt{1-\lambda^2}} < \infty$$
and so $P_n \contig Q_n$. Conversely, if $\lambda > \lambda^*_\cX$ then
$$\lim_{n \to \infty} \Ex_{Q_n}\left(\dd[\tilde P_n]{Q_n}\right)^2 = \infty.$$
\end{theorem}

\noindent Note that this is a tight characterization of when the second moment is bounded, but not necessarily a tight characterization of when contiguity holds.

Above we have computed the limit value of the second moment in the case $\lambda < \lambda^*_\cX$ as follows. Defining $Z$ as in Proposition~\ref{prop:nn} we have $\langle Z, \beta \rangle \sim \cN(0,\sigma^2)$ where
\begin{align*}
\sigma^2 &= \beta^\top (\diag(\bar\alpha) - \bar\alpha \bar\alpha^\top)\beta \\
&= \sum_{ab} \beta_{ab}^2 \bar\alpha_{ab} + \left(\sum_{ab} \beta_{ab} \bar\alpha_{ab} \right)^2 \\
&= \left(\sum_a a^2 \pi_a\right)\left(\sum_b b^2 \pi_b\right) + \left(\sum_a a \pi_a \sum_b b \pi_b \right)^2 \\
&= 1
\end{align*}
since $\pi$ is mean-zero and unit-variance, and so
$$\EE[\exp(Z^\top A Z)] = \EE\left[\exp\left(\frac{\lambda^2}{2}\langle Z, \beta \rangle^2\right)\right] = \EE\left[\exp\left(\frac{\lambda^2}{2}\chi_1^2\right)\right] = \frac{1}{\sqrt{1-\lambda^2}}.$$

\subsection{Application: the sparse Rademacher prior}
\label{sec:sparse-rad}

As an example, consider the case where $\pi = \sqrt{1/\rho}\,\mathcal{R}(\rho)$ where $\mathcal{R}(\rho)$ is the sparse Rademacher distribution with sparsity $\rho \in [0,1]$:
$$\mathcal{R}(\rho) = \left\{\begin{array}{ccc} 0 & \text{w.p.} & 1-\rho \\ +1 & \text{w.p.} & \rho/2 \\ -1 & \text{w.p.} & \rho/2 \end{array}\right..$$
First we try the sub-Gaussian method of Section~\ref{sec:subg-method}. Note that $\pi\pi' = \frac{1}{\rho}\mathcal{R}(\rho^2)$. The variance proxy $\sigma^2$ for $\pi\pi'$ needs to satisfy
\begin{equation} \exp\left(\frac{1}{2} \sigma^2 t^2\right) \ge \EE \exp(t \pi\pi') = (1-\rho^2) + \frac{\rho^2}{2}\exp(t/\rho) + \frac{\rho^2}{2}\exp(-t/\rho) \label{eq:sprad-mgf-compare}\end{equation}
for all $t \in \mathbb{R}$ so the best (smallest) choice for $\sigma^2$ is
$$(\sigma^*)^2 = \sup_{t \in \mathbb{R}} \frac{2}{t^2}\log\left[(1-\rho^2) + \frac{\rho^2}{2}\exp(t/\rho) + \frac{\rho^2}{2}\exp(-t/\rho)\right].$$

Recall that Theorem~\ref{thm:subg} (sub-Gaussian method) gives contiguity for all $\lambda < 1/\sigma^*$. We now resolve a conjecture stated in \citet{bmvx}. For sufficiently large $\rho$, this optimum is in fact $\sigma^* = 1$, implying that PCA is tight:
\begin{theorem}\label{thm:sparse-subg}
When $\rho \geq 1/\sqrt{3} \approx 0.577$, we have $\sigma^* = 1$, yielding contiguity for all $\lambda < 1$. On the other hand, if $\rho < 1/\sqrt{3}$, then $\sigma^* > 1$.
\end{theorem}
\begin{proof}
We are equivalently interested in the following reformulation of (\ref{eq:sprad-mgf-compare}):
\begin{equation} \frac12 \sigma^2 t^2 \stackrel{?}{\geq} \log\left( (1-\rho^2) + \frac{\rho^2}{2}\exp(t/\rho) + \frac{\rho^2}{2}\exp(-t/\rho) \right) = \log\left( 1-\rho^2 + \rho^2 \cosh(t/\rho) \right) \defeq k_\rho(t). \label{eq:sprad-cgf-compare} \end{equation}
Both sides of the inequality are even functions, agreeing in value at $t=0$. When $\sigma^2 < 1$, the inequality fails, by comparing their second-order behavior about $t=0$. When $\sigma^2 = 1$ but $\rho < 1/\sqrt{3}$, the inequality fails, as the two sides have matching behavior up to third order, but $k^{(4)}_\rho(0) = 3 - \rho^{-2} < 0$.

It remains to show that the inequality (\ref{eq:sprad-cgf-compare}) does hold for $\rho > 1/\sqrt{3}$ and $\sigma^2 = 1$. As the left and right sides agree to first order at $t=0$, and are both even functions, it suffices to show that for all $t \geq 0$,
$$ 1 \stackrel{?}{\geq} k_\rho''(t) = \frac{\rho^2 + (1-\rho^2) \cosh(t/\rho)}{(1-\rho^2+\rho^2 \cosh(t/\rho))^2}. $$
Completing the square for $\cosh$, we have the equivalent inequality:
$$ 0 \stackrel{?}{\leq} 1 - 3\rho^2 + \rho^4 + \Big( \underbrace{\rho^2 \cosh(t/\rho) + \frac{(2\rho^2-1)(1-\rho^2)}{2\rho^2}}_{(*)} \Big)^2 - \frac{(2\rho^2 - 1)^2(1-\rho^2)^2}{4\rho^4}. $$
Note that $\cosh$ is bounded below by $1$; thus for $\rho > 1/\sqrt{3}$, the underbraced term ($*$) is nonnegative, and hence minimized in absolute value when $t=0$. It then suffices to establish the above inequality in the case $t=0$, so that $\cosh(t/\rho) = 1$; but here the inequality is in fact an equality, by simple algebra.
\end{proof}

Note that Theorem~\ref{thm:sparse-subg} above implies that the sub-Gaussian method cannot show that PCA is optimal when $\rho < 1/\sqrt{3}$. Using the conditioning method of Section~\ref{sec:cond-method}, we will now improve the range of $\rho$ for which PCA is optimal, although our argument here relies on numerical optimization. Thus, this is an example where conditioning away from `bad' events improves the behavior of the second moment.

\begin{example}\label{ex:sparse-rad}
Let $\cX$ be the sparse Rademacher prior $\IID(\sqrt{1/\rho} \,\mathcal{R}(\rho))$. There exists a critical value $\rho^* \approx 0.184$ (numerically computed) such that if $\rho \ge \rho^*$ and $\lambda < 1$ then $\GWig(\lambda,\cX)$ is contiguous to $\GWig(0,\cX)$. When $\rho < \rho^*$ we are only able to show contiguity when $\lambda < \lambda^*_\rho$ for some $\lambda^*_\rho < 1$.
\end{example}

\noindent We use ``example'' rather than ``theorem'' to indicate results that rely on numerical computations.

\begin{proof}[Details]
Consider the optimization problem of Theorem~\ref{thm:cond-method} (conditioning method). We will first use symmetry to argue that the optimal $\alpha$ must take a simple form. Abbreviate the support of $\pi$ as $\{0,+,-\}$. For a given $\alpha$ matrix, define its complement by swapping $+$ and $-$, e.g.\ swap $\alpha_{0+}$ with $\alpha_{0-}$ and swap $\alpha_{-+}$ with $\alpha_{+-}$. Note that if we average $\alpha$ with its complement, the numerator $\langle \alpha,\beta \rangle^2$ remains unchanged, the denominator $D(\alpha,\bar\alpha)$ can only decrease, and the row- and column-sum constraints remain satisfied; this means the new solution is at least as good as the original $\alpha$. Therefore we only need to consider $\alpha$ values satisfying $\alpha_{++} = \alpha_{--}$ and $\alpha_{+-} = \alpha_{-+}$. Note that the remaining entries of $\alpha$ are uniquely determined by the row- and column-sum constraints, and so we have reduced the problem to only two variables. It is now easy to solve the optimization problem numerically, say by grid search. The result is that we have contiguity for all $\lambda < 1$ provided $\rho$ exceeds a new critical value $\rho^* \approx 0.184$, an improvement over the sub-Gaussian method.
\end{proof}

We do not expect that $0.184$ is the true critical $\rho$ value (the value at which it becomes possible to detect below the spectral threshold) because we expect non-detection and non-recovery to behave the same, and the critical $\rho$ for recovery is known to be approximately $0.09$. This was first conjectured based on heuristics from statistical physics and later proven rigorously \citep{mi,mi-proof}. They also conjecture a computational gap: no efficient algorithm can go below the spectral threshold (regardless of $\rho$).

Although our result for the sparse Rademacher prior does not seem to be tight in terms of the critical $\rho$ value, it is worth noting that it is tight in the sense that once $\rho$ exceeds $0.184$, the modified second moment used in the conditioning method (conditioned on `good' spikes) is infinite. This follows from Theorem~\ref{thm:cond-method} (conditioning method), which is based on Proposition~5 in \citet{bmnn}. Note that this yields an example where the modified second moment is unbounded yet contiguity is expected to hold.

Contiguity for the sparse Rademacher model via the second moment method was also recently studied in \citet{bmvx}. They obtain analytic upper and lower bounds for the threshold $\lambda^*_\rho$, focusing on small values of $\rho$ for which it is possible (via inefficient algorithms) to go below the spectral threshold. They do not, however, give contiguity results that match PCA for any $\rho$; our results resolve a conjecture that they state in this direction, that the PCA threshold is tight for sufficiently large $\rho$.

We remark that our results for the sparse Rademacher prior can be improved by a more involved conditioning method where the `bad' events depend on both and signal and noise; see \cite{noise-cond}.

\subsection{Application: the Gaussian prior}

We now highlight some important issues by discussing the \iid Gaussian prior where $\pi$ is $\mathcal{N}(0,1)$. Although this appears to be a well-behaved prior, we actually cannot apply the sub-Gaussian method because the product of two independent Gaussians is not sub-Gaussian (with any variance proxy). In fact, we expect that the second moment $\EE_{x,x'} \exp\left(\frac{n \lambda^2}{2} \langle x,x' \rangle^2\right)$ is unbounded for all $\lambda > 0$. However, we are still able to prove contiguity for all $\lambda < 1$ by using a variant of the conditioning idea. Note that the \iid Gaussian prior is very similar to the spherical prior; the spherical prior is obtained by drawing $x$ from the Gaussian prior and then normalizing it (due to Gaussian spherical symmetry). The reason that the spherical prior has finite second moment while the Gaussian one does not is because the Gaussian prior allows for extremely rare `bad' events where $x$ has large norm; due to the exponential scaling on the second moment, these rare events dominate. In order to fix this issue, we condition on the `good' events and show that the resulting second moment is finite by comparison to the spherical prior.

\begin{theorem}
\label{thm:gauss-prior}
If $\lambda < 1$ then $\GWig(\lambda,\IID(\cN(0,1))) \contig \GWig(0)$.
\end{theorem}

\begin{proof}
Let $\eps > 0$. Let $x$ be drawn from $\cX = \IID(\cN(0,1))$. Let $\omega_n$ be the `good' event $\|x\|^2 \le 1+\eps$, which occurs with probability $1-o(1)$. Let $y = \frac{x}{\|x\|}$ and note that $y$ is drawn from the spherical prior of Section~\ref{sec:spherical} (by Gaussian spherical symmetry). Let $\tilde P_n$ be the distribution that samples $x$ from $\cX$, outputs $x$ if $\omega_n$ occurs, and outputs the zero vector otherwise. We will use the second moment argument to show $\tilde P_n \contig Q_n$, which implies $P_n \contig Q_n$ (see Lemma~\ref{lem:cond}). Let $\Omega_n$ be the event that $x$ and $x'$ both satisfy $\omega_n$.
\begin{align*}
\Ex_{Q_n}\left(\dd[\tilde{{P}}_n]{{Q}_n}\right)^2
&= \Ex_{x,x'}\left[\one_{\Omega_n}\exp\left(\frac{n \lambda^2}{2} \langle x,x' \rangle^2\right)\right] + o(1) \\
&= \Ex_{x,x'}\left[\one_{\Omega_n}\exp\left(\frac{n \lambda^2}{2} \|x\|^2 \|x'\|^2 \langle y,y' \rangle^2\right)\right] + o(1) \\
&\le \Ex_{y,y'}\exp\left(\frac{n \lambda^2}{2} (1+\eps)^2 \langle y,y' \rangle^2\right) + o(1). \\
\end{align*}
From  Theorem~\ref{thm:sphere-prior} on the spherical prior, we know that this is bounded as $n \to \infty$ provided that $\lambda (1+\eps) < 1$. For any $\lambda < 1$ we can choose $\eps > 0$ small enough to ensure this.

\end{proof}

\subsection{Contiguity of eigenvalues}

In this subsection we show that for \emph{any} prior $\cX$ (with $\|x\| \to 1$ in probability) and for any $\lambda < 1$, the \emph{eigenvalues} of the spiked model are contiguous to the eigenvalues of the unspiked model. Thus, we have that regardless of the prior, no eigenvalue-based test can detect the spike when $\lambda < 1$. This does not follow from any of the known spectral results on spiked matrices because although we know that when $\lambda < 1$ the two models asymptotically agree on many statistics such as top eigenvalue and empirical eigenvalue distribution, this does not rule out all possible eigenvalue-based tests (e.g.\ gaps between eigenvalues, etc.). Therefore, for the sparse Rademacher prior with sufficiently low sparsity (for instance) we know that for $\lambda < 1$, no eigenvalue-based test succeeds, yet there exist other tests that do \citep{bmvx}. However, the only known tests that can go below $\lambda = 1$ are computationally inefficient.

\begin{theorem}
Let $\cX$ be any spike prior (with $\|x\| \to 1$ in probability). Let $P_n$ be the joint distribution of eigenvalues of $\GWig_n(\lambda,\cX)$ and let $Q_n$ be the joint distribution of eigenvalues of $\GWig_n(0)$. If $\lambda < 1$ then $P_n$ is contiguous to $Q_n$.
\end{theorem}
\begin{proof}
Fix $\eps > 0$. Since $\|x\|^2 \le 1+\eps$ with probability $1-o(1)$, it is sufficient to show contiguity for the modified prior $\tilde X$ that changes $x$ to the zero vector whenever the bad event $\|x\|^2 > 1+\eps$ occurs; let $\tilde P_n$ be the eigenvalue distribution of $\GWig_n(\lambda,\tilde X)$. Due to Gaussian spherical symmetry, the distribution of eigenvalues of the spiked matrix depend only on the norm of the spike and not its direction. Therefore, the following prior $\cX'$ also yields the eigenvalue distribution $\tilde P_n$: draw $x$ from $\tilde X$ and output $\|x\|\, y$ where $y$ is drawn from the spherical prior $\cXs$. But (for sufficiently small $\eps$) we have that $\GWig(\lambda, \cX')$ is contiguous to $\GWig(0)$ because the second moment is
$$\Ex_{x,x' \sim \cX'} \exp\left(\frac{\lambda^2 n}{2} \langle x,x' \rangle^2\right) = \Ex_{\substack{x,x' \sim \tilde X\\y,y' \sim \cXs}} \exp\left(\frac{\lambda^2 n}{2} \|x\|^2 \|x'\|^2 \langle y,y' \rangle^2\right) \le \Ex_{y,y' \sim \cXs} \exp\left(\frac{\lambda^2 n}{2} (1+\eps)^2 \langle y,y' \rangle^2\right)$$
which, by Theorem~\ref{thm:sphere-prior} on the spherical prior, is bounded as $n \to \infty$ provided $\lambda (1+\eps) < 1$. But if two matrix distributions are contiguous then so are their eigenvalues, completing the proof.
\end{proof}

\section{Non-Gaussian Wigner models}
\label{sec:nong-wig}

In this section we consider the spiked Wigner model with non-Gaussian noise distributions. In Section~\ref{sec:mainngw} we define and state our main results. In Section~\ref{sec:gnih} we show that of all noise distributions, Gaussian noise makes the detection problem the hardest. In Section~\ref{sec:nong-symm} we establish contiguity results for non-Gaussian Wigner models. In Section~\ref{sec:nong-wig-upper} we show a modified PCA procedure that can solve the detection problem strictly below the threshold where standard PCA works.

\subsection{Main results}\label{sec:mainngw}

The spiked non-Gaussian Wigner model is defined as follows:

\begin{definition}\label{def:spiked-nongaussian-wigner}
Given $\lambda \geq 0$, a spike prior $\cX$, and a noise distribution $\cP$ on $\RR$ with mean $0$ and variance $1$, we define the general spiked Wigner model $\Wig(\lambda,\cP,\cX)$ as follows: a spike $x \in \RR^n$ is drawn from $\cX$, and we observe the matrix
$$ Y = \lambda x x^\top + \frac{1}{\sqrt{n}}W, $$
where the symmetric matrix $W$ is drawn entrywise from $\cP$, with entries independent except for symmetry. For simplicity we take the diagonal entries of $Y$ to be $0$.
\end{definition}

\noindent Recall that the prior $\cX$ is required to obey the normalization $\|x\| \to 1$ in probability (see Definition~\ref{def:prior}). The spectral behavior of this model is well understood (see e.g.\ \citet{fp,cdf,wig-spk,nong-eigv1}). In fact it exhibits \emph{universality} (see e.g.\ \cite{TV-univ}): regardless of the choice of the noise distribution $\cP$ (as long as it has mean zero, variance one, and sufficiently many finite moments), many properties of the spectrum behave the same as if $\cP$ were a standard Gaussian distribution. In particular, for $\lambda \le 1$, the spectrum bulk has a semicircular distribution and the maximum eigenvalue converges almost surely to $2$. For $\lambda > 1$, an isolated eigenvalue emerges from the spectrum with value converging to $\lambda + 1/\lambda$, and (under suitable assumptions) the top eigenvector has squared correlation $1 - 1/\lambda^2$ with the truth. 

In contrast we will show that from a statistical standpoint, universality breaks down entirely. We will see that the difficulty of the problem depends on $\cP$ via the parameters $\lambda^*_\cX$ and $F_\cP$ defined below, with Gaussian noise being the hardest (for a fixed variance). Let $\cX$ be a spike prior, and suppose that through the second moment method, we can establish contiguity between the Gaussian spiked and unspiked models whenever $\lambda$ lies below some critical value
\begin{equation}\label{eq:lambdastar} \lambda^*_{\cX} = \sup \left\{ \lambda \mid \EE_{x,x' \sim \cX} \exp\left( \frac{\lambda^2 n}{2} \langle x,x' \rangle^2\right) \text{ is bounded as $n \to \infty$} \right\}. \end{equation}
For instance, as discussed in the previous section, we have $\lambda^*_\cX = 1$ for the uniform prior on the unit sphere, as well as for the \iid Rademacher prior. Following universality, we might imagine that contiguity holds in the non-Gaussian setting as well \--- but this is far from the case.  Instead, we find that the choice of noise shifts the threshold:

\begin{theorem-nolabel}[informal; see Theorems~\ref{thm:nongauss-lower} and~\ref{thm:nong-upper}]
Under suitable conditions (see Assumptions~\ref{as:nong-lower} and~\ref{as:nong-upper}), the spiked model is contiguous to the unspiked model for all $\lambda < \lambda^*_\cX/\sqrt{F_\cP}$; but when $\lambda > 1/\sqrt{F_\cP}$, there exists an entrywise transformation $f$ such that the spiked and unspiked models can be distinguished via the top eigenvalue of $f(\sqrt{n} Y)$,and furthermore the top eigenvector of $f(\sqrt{n} Y)$ has nontrivial correlation with the spike. 
\end{theorem-nolabel}
\noindent The function $f$ is explicitly defined below. We require $\cP$ to be a continuous distribution with density $p(w)$. The parameter $F_\cP$, which quantifies its difficulty, is the Fisher information of $\cP$ under translation:
$$ F_\cP = \Ex_{w \sim \cP}\left[\left(\frac{p'(w)}{p(w)}\right)^2\right] = \int_{-\infty}^\infty\frac{p'(w)^2}{p(w)} \,\dee w. $$
Provided $\cP$ has unit variance, this quantity is always at least $1$, with equality only in the case of a standard Gaussian (among all noise distributions satisfying certain reasonable properties); see Proposition~\ref{prop:F}. Note that when $\lambda^*_\cX = 1$ (i.e.\ PCA is optimal for the Gaussian noise setting), our upper and lower bounds match, and so our modified PCA procedure is optimal for the non-Gaussian setting.

Our upper bound proceeds by a modified PCA procedure. Define $f(w) = -p'(w)/p(w)$, where $p$ is the probability density function of the noise $\cP$. Given the observed matrix $Y$, we apply $f$ entrywise to $\sqrt{n} Y$, and examine the largest eigenvalue. This entrywise transformation approximately yields another spiked Wigner model, but with improved signal-to-noise ratio. One can derive the transformation $-p'(w)/p(w)$ by using calculus of variations to optimize a spike-to-noise ratio of this new spiked Wigner model. This phenomenon is illustrated in Figures~\ref{fig:spectrum} and~\ref{fig:transform}:

\begin{figure}[!ht]
    \centering
    \begin{minipage}{0.4\textwidth}
    \hspace*{-0.15in}\includegraphics{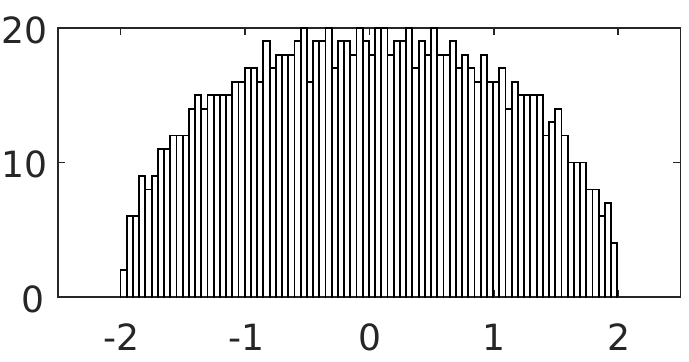}\\
    \hspace*{-0.15in}\includegraphics{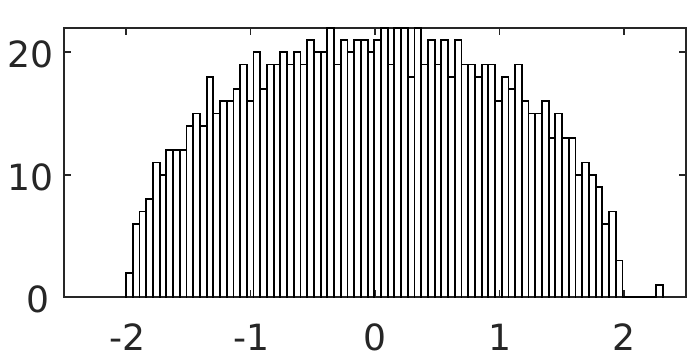}
    \captionof{figure}{The spectrum of a spiked Wigner matrix ($\lambda = 0.9$, $n=1200$) with bimodal noise, before (above) and after (below) the entrywise transformation. An isolated eigenvalue is evident only in the latter. Both are normalized by $1/\sqrt{n}$.}
    \label{fig:spectrum}
    \end{minipage}\hspace{3em}%
    \begin{minipage}{0.4\textwidth}
    \vspace{.38in}\hspace{-0.2in}
    \includegraphics{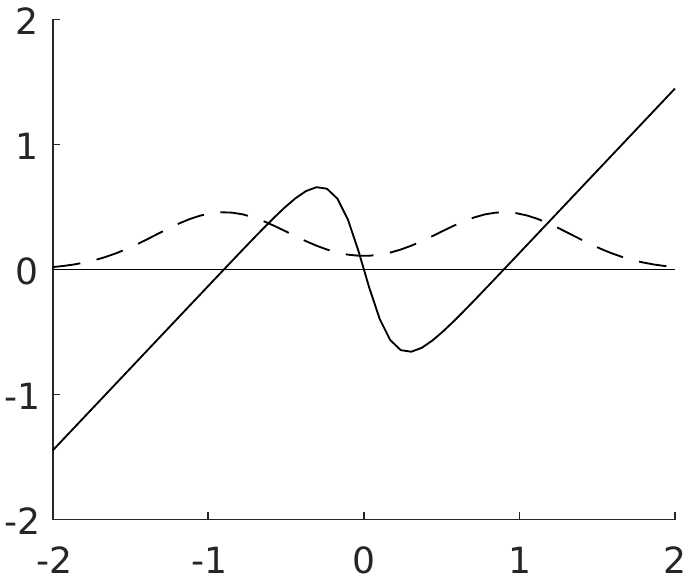}
    \captionof{figure}{The bimodal density $p$ (dashed) and entrywise transformation $-p'/p$ (solid). The noise is a convolution of Rademacher and Gaussian random variables.}
    \label{fig:transform}
    \end{minipage}
\end{figure}

These results on non-Gaussian noise parallel a \emph{channel universality} phenomenon for mutual information, due to \citet{mi} (shown for finitely-supported \iid priors); in particular, channel universality implies (via the I-MMSE relation of \citet{i-mmse}) the analogue of our results for the recovery threshold. The modified PCA procedure we use for our upper bound was previously suggested in \citet{LKZ-amp} based on linearizing an AMP algorithm, but to our knowledge, no rigorous results have been previously established about its performance in general. Other entrywise pre-transformations have been shown to improve spectral approaches to various structured PCA problems \citep{DM-sparse-pca,kv}.

\subsection{Gaussian noise is the hardest}\label{sec:gnih}

In this subsection, we prove that Gaussian noise is the hardest in the following sense:

\begin{proposition}\label{prop:F}
Let $\cP$ be a continuous distribution with a continuously differentiable density function $p(w)$ with $p(w) > 0$ everywhere. Suppose $\mathrm{Var}[\cP] = 1$. Then $F_\cP \ge 1$ with equality if and only if $\cP$ is a standard Gaussian.
\end{proposition}

To intuitively understand why non-Gaussian noise makes the detection problem easier, consider the extreme case where the noise distribution is uniform on $\{\pm 1\}$. Since the signal $\lambda xx^\top$ is entrywise $\Theta(1/n)$ and the noise $\frac{1}{\sqrt n} W$ is entrywise $\pm 1/\sqrt n$, it is actually quite easy to detect the spike in this case. If there is no spike, all the entries will be $\pm \frac{1}{\sqrt n}$. If there is a spike, each entry will be $\pm \frac{1}{\sqrt n}$ plus a much smaller offset. One can therefore subtract off the noise and recover the signal exactly. In fact, if we let the noise be a smoothed version of $\{\pm 1\}$ (so that the derivative $p'$ exists), the entrywise transformation $-p'(w)/p(w)$ is precisely implementing this noise-subtraction procedure. Note that this justifies the restriction to continuous noise distributions because any distribution with a point mass will admit a similar trivial recovery procedure and we will not have contiguity for \emph{any} $\lambda > 0$.

\begin{proof}[Proof of Proposition~\ref{prop:F}]
Since $F_\cP$ is translation-invariant, assume $\EE[\cP] = 0$ without loss of generality. We have
\begin{align*}
0 &\le \int_{-\infty}^\infty \frac{1}{p(w)}\left(p'(w) + w p(w)\right)^2 \,\dee w \\
&= \int_{-\infty}^\infty \left[\frac{p'(w)^2}{p(w)} + 2wp'(w) + w^2 p(w)\right] \,\dee w \\
&= F_\cP + \int_{-\infty}^\infty 2wp'(w)\,\dee w + 1
\end{align*}
since $\EE[\cP] = 0$ and $\mathrm{Var}[\cP] = 1$. (The integral in the first line is finite, provided that $F_\cP$ and $\mathrm{Var}[\cP]$ are finite.) Using integration by parts,
$$\int_{-\infty}^\infty 2wp'(w)\,\dee w = 2wp(w)\Big|_{-\infty}^\infty - \int_{-\infty}^\infty 2p(w) \,\dee w = -2$$
since $wp(w) \to 0$ as $w \to \pm\infty$ or else $p(w)$ would not be integrable. (Here we have used the fact that the limits $\lim_{w \to \pm \infty} wp(w)$ must exists, since the left-hand side is defined.)
We now have $F_\cP \ge 1$. Equality holds only if $p'(w) = -wp(w)$ for all $w$. We can solve this differential equation for $p(w)$:
$$\dd{p}\log p(w) = -w$$
$$p(w) = C \exp\left(-\frac{w^2}{2}\right)$$
which is a standard Gaussian.
\end{proof}

\subsection{Symmetric noise and general priors}
\label{sec:nong-symm}

In this subsection, we set up and state our main statistical lower bound that establishes contiguity in the non-Gaussian Wigner setting. Given a noise distribution, define the \emph{translation function}
$$ \tau(a,b) = \log \Ex_{\cP}\left[ \dd[T_a \cP]{\cP} \dd[T_b \cP]{\cP} \right] = \log \Ex_{z \sim \cP}\left[ \frac{p(z-a)}{p(z)} \frac{p(z-b)}{p(z)} \right], $$
where $T_a \cP$ denotes the translation of distribution $\cP$ by $a$. For instance, the translation function of standard Gaussian noise is computed to be $\tau(a,b) = ab$.

\begin{assumption}\label{as:nong-lower}
We assume the following of the prior $\cX$:
\begin{enumerate}[(i)]
\item With probability $1-o(1)$, for all $i \in [n]$, $|x_i| < n^{-1/3}$,
\item for each $q \in \{2,4,6,8\}$, there exists a constant $\alpha_q$ with $\Pr[ \|x\|_q > \alpha_q n^{\frac1q - \frac12}] = o(1)$.\\
\quad\\
We assume the following of the noise $\cP$:
\item $\cP$ is a continuous distribution with a density function $p(w)$,
\item $p(w) > 0$ everywhere,
\item The translation function $\tau$ is $C^4$ in a neighborhood of $(0,0)$,
\item $\cP$ is symmetric about $0$.
\end{enumerate}
\end{assumption}

The above assumptions on $\cP$ are satisfied by any symmetric mixture of Gaussians of positive variance, for example, but will rule out some extremely sparse priors whose entries are large when nonzero.
We expect that the symmetry of $\cP$ is not crucial, but relaxing this condition adds considerable complication to the next theorem. In Appendix~\ref{app:nong-prior-conditions} we show that the assumptions on $\cX$ are satisfied for the spherical prior and certain \iid priors; see Propositions~\ref{prop:nong-lower-spherical} and \ref{prop:nong-lower-spherical} below.

\begin{theorem}\label{thm:nongauss-lower}
Under Assumption~\ref{as:nong-lower}, $\Wig(\lambda,\cP,\cX)$ is contiguous to $\Wig(0,\cP)$ for all $\lambda < \lambda^*_\cX/\sqrt{F_\cP}$.
\end{theorem}

\begin{proof}
We begin by defining a modification $\tilde \cX$ of the prior $\cX$, by returning the spike $0$ whenever one of the tail events described in Assumption~\ref{as:nong-lower} occur---namely, when some entry $x_i$ exceeds $n^{-1/3}$ in magnitude, or when $\|x\|_q > \alpha_q$ for some $q \in \{2,4,6,8\}$. By hypothesis, with probability $1 - o(1)$, no such tail event occurs; hence if $\Wig(\lambda,\cP,\tilde\cX)$ is contiguous to $\Wig(0,\cP)$ then so is $\Wig(\lambda,\cP,\cX)$. Let $P_n = \Wig_n(\lambda,\cP,\cX)$, $\tilde P_n = \Wig_n(\lambda,\cP,\tilde\cX)$, and $Q_n = \Wig_n(0,\cP)$.

We proceed from the second moment:
\begin{align*}
\Ex_{Q_n} \left(\dd[\tilde P_n]{Q_n}\right)^2
&= \Ex_{Y \sim Q_n}\left[ \Ex_{x,x' \sim \tilde\cX} \prod_{i < j} \frac{p(\sqrt{n} Y_{ij} - \lambda \sqrt{n} x_i x_j)}{p(\sqrt{n} Y_{ij})} \frac{p(\sqrt{n} Y_{ij} - \lambda \sqrt{n}  x_i' x_j')}{p(\sqrt{n} Y_{ij})} \right] \\
&= \Ex_{x,x' \sim \tilde\cX}\left[ \prod_{i < j} \Ex_{\sqrt{n} Y_{ij} \sim \cP} \frac{p(\sqrt{n} Y_{ij} - \lambda\sqrt{n} x_i x_j)}{p(\sqrt{n} Y_{ij})} \frac{p(\sqrt{n} Y_{ij} - \lambda \sqrt{n} x_i' x_j')}{p(\sqrt{n} Y_{ij})} \right] \\
&= \Ex_{x,x' \sim \tilde\cX}\left[ \exp\left( \sum_{i < j} \tau(\lambda \sqrt{n} x_i x_j, \lambda \sqrt{n} x_i' x_j') \right) \right].
\end{align*}

We will expand $\tau$ using Taylor's theorem, using the $C^4$ assumption:
$$ \tau(a,b) = \sum_{0 \leq k+\ell \leq 3} \frac{\partial^{k+\ell} \tau}{\partial a^k \partial b^\ell}(0,0)\; a^k b^\ell + \sum_{k+\ell=4} \left( \frac{\partial^4 \tau}{\partial a^k \partial b^\ell}(0,0) + h_{k,\ell}(a,b) \right) a^k b^\ell $$
for some remainder function $h_{k,\ell}(a,b)$ tending to $0$ as $(a,b) \to (0,0)$. Given the bounds assumed on the entries of $x$ and $x'$, these remainder terms $h_{k,\ell}(\lambda\sqrt{n} x_i x_j, \lambda\sqrt{n} x_i' x_j')$ are $o(1)$ as $n \to \infty$. Note that $\tau(a,0) = 0 = \tau(0,b)$, so that the non-mixed partials of $\tau$ vanish. Further, by the hypothesis of noise symmetry, we have $\tau(-a,-b) = \tau(a,b)$, so that all partials of odd total degree vanish; in particular the mixed third partials vanish. We note also that $\frac{\partial^2 \tau}{\partial a \partial b}(0,0) = F_\cP$, the Fisher information defined above. Thus,
\begin{align*}
\Ex_{Q_n} \left(\dd[\tilde P_n]{Q_n}\right)^2 &= \Ex_{x,x' \sim \tilde \cX}\left[ \exp\left( F_\cP \lambda^2 n \sum_{i < j} x_i x_j x_i' x_j' + \sum_{k+\ell=4} \left( \frac{\partial^4 \tau}{\partial a^k \partial b^\ell}(0,0) + o(1) \right) \lambda^4 n^2 \sum_{i < j} x_i^k x_j^k (x_i')^\ell (x_j')^\ell \right) \right] \\
&\leq \Ex_{x,x' \sim \tilde\cX}\left[ \exp\left( \frac{F_\cP \lambda^2 n}{2} \langle x,x' \rangle^2 \right) \prod_{k+\ell=4} \exp\left( \left( \frac{\partial^4 \tau}{\partial a^k \partial b^\ell}(0,0) + o(1) \right) \frac{\lambda^4 n^2}{2} \langle x^k, (x')^\ell \rangle^2 \right) \right],
\end{align*}
where $x^k$ denotes entrywise $k$th power. For all $\eps > 0$, we can apply the weighted AM--GM inequality:
\begin{align*}
&\leq \Ex_{x,x' \sim \tilde\cX}\left[ (1-\eps) \exp\left( \frac{F_\cP \lambda^2 n}{2} \langle x,x' \rangle^2 \right)^{(1-\eps)^{-1}} + \sum_{k+\ell=4} \frac{\eps}{5} \exp\left( \left( \frac{\partial^4 \tau}{\partial a^k \partial b^\ell}(0,0) + o(1) \right) \frac{\lambda^4 n^2}{2} \langle x^k, (x')^\ell \rangle^2 \right)^{5/\eps} \right] \\
&= \Ex_{x,x' \sim \tilde\cX}\left[ (1-\eps) \exp\left( \frac{(1-\eps)^{-1} F_\cP \lambda^2 n}{2} \langle x,x' \rangle^2 \right) \right] \\
&\qquad + \sum_{k+\ell=4} \frac{\eps}{5} \Ex_{x,x' \sim \tilde\cX}\left[ \exp\left( \left( \frac{\partial^4 \tau}{\partial a^k \partial b^\ell}(0,0) + o(1) \right) \frac{5 \lambda^4 n^2}{2 \eps} \langle x^k, (x')^\ell \rangle^2 \right) \right],\numberthis\label{eq:momentparts}
\end{align*}
so it suffices to bound each of these expectations.

By hypothesis, $\lambda < \lambda^*_\cX / \sqrt{F_\cP}$, implying that we can choose $\eps > 0$ such that $(1-\eps)^{-1} F_\cP \lambda^2 < (\lambda^*_\cX)^2$. But $\tilde\cX$ is dominated as a measure by the sum of $\cX$ and an $o(1)$ mass at $0$; it follows that $\lambda_\cX \leq \lambda_{\tilde \cX}$, and the first expectation in (\ref{eq:momentparts}) is bounded.

We bound each of the other expectations using Cauchy--Schwarz:
\begin{align*}
& \Ex_{x,x' \sim \tilde\cX}\left[ \exp\left( \left( \frac{\partial^4 \tau}{\partial a^k \partial b^\ell}(0,0) + o(1) \right) \frac{5 \lambda^4 n^2}{2\eps} \langle x^k, (x')^\ell \rangle^2 \right) \right] \\
&\leq \Ex_{x,x' \sim \tilde\cX}\left[ \exp\left( \left( \frac{\partial^4 \tau}{\partial a^k \partial b^\ell}(0,0) + o(1) \right) \frac{5 \lambda^4 n^2}{2\eps} \|x^k\|_2^2\; \|(x')^\ell\|_2^2 \right) \right] \\
&= \Ex_{x,x' \sim \tilde\cX}\left[ \exp\left( \left( \frac{\partial^4 \tau}{\partial a^k \partial b^\ell}(0,0) + o(1) \right) \frac{5\lambda^4 n^2}{2\eps} \|x\|_{2k}^{2k}\; \|x'\|_{2\ell}^{2\ell} \right) \right] \\
&\leq \Ex_{x,x' \sim \tilde\cX}\left[ \exp\left( \left( \frac{\partial^4 \tau}{\partial a^k \partial b^\ell}(0,0) + o(1) \right) \frac{5\lambda^4 n^2}{2\eps} \alpha_{2k}^{2k} n^{1-k} \alpha_{2\ell}^{2\ell} n^{1-\ell} \right) \right],
\intertext{due to the norm restrictions on prior $\tilde\cX$,}
&= \exp\left( \left( \frac{\partial^4 \tau}{\partial a^k \partial b^\ell}(0,0) + o(1) \right) \frac{5\lambda^4}{2\eps} \alpha_{2k}^{2k} \alpha_{2\ell}^{2\ell} \right),
\end{align*}
which remains bounded as $n \to \infty$.

With the overall second moment $\Ex_{Q_n} \left(\dd[\tilde P_n]{Q_n}\right)^2$ bounded as $n \to \infty$, the result follows from Lemma~\ref{lem:cond}.
\end{proof}

In Appendix~\ref{app:nong-prior-conditions}, we verify the hypotheses of this theorem for spherical and \iid priors:

\begin{proposition}\label{prop:nong-lower-spherical}
Consider the spherical prior $\cXs$. Then conditions (i) and (ii) in Assumption~\ref{as:nong-lower} are satisfied.
\end{proposition}

\begin{proposition}\label{prop:nong-lower-iid}
Consider an \iid prior $\cX = \IID(\pi)$ where $\pi$ is zero-mean and unit-variance with $\EE[\pi^{16}] < \infty$. Then conditions (i) and (ii) in Assumption~\ref{as:nong-lower} are satisfied.
\end{proposition}
\noindent An immediate implication of this is that conditions (i) and (ii) are also satisfied for a `conditioned' prior which draws $x$ from $\IID(\pi)$ but then outputs zero if a `bad' event occurred.

\subsection{Pre-transformed PCA}
\label{sec:nong-wig-upper}

In this subsection we analyze a modified PCA procedure for the non-Gaussian spiked Wigner model, which in certain cases matches the lower bound of the previous subsections. Recall that the non-Gaussian Wigner model $\Wig(\lambda,\cX,\cP)$ is given by
$$Y = \lambda xx^\top + \frac{1}{\sqrt n} W.$$
In this subsection, however, we will prefer the normalization
$$\hat Y = \sqrt{n}\, Y = \lambda \sqrt{n}\, xx^\top + W$$
so that the noise is entrywise constant size. Recall that $x$ is drawn from a prior $\cX$, normalized so that $\|x\| \to 1$ in probability. The noise $W$ is a symmetric matrix with off-diagonal entries drawn from some distribution $\cP$. For consistency with the previous section we take the diagonal entries of $Y$ to be zero. The entries of $W$ are independent except for symmetry: $W_{ij} = W_{ji}$. We do not allow $\lambda$ and $\cP$ to depend on $n$. We make the following regularity assumptions.

\begin{assumption}
\label{as:nong-upper}
Assumption on $\cX$:
\begin{enumerate}[(i)]
\item With probability $1-o(1)$, all entries of $x$ are small: $|x_i| \le n^{-1/2 + \alpha}$ for some fixed $\alpha < \frac{1}{32}$.\\
\quad\\
Assumptions on $\cP$:
\item $\cP$ is a continuous distribution with a density function $p(w)$ that is three times differentiable.
\item $p(w) > 0$ everywhere.
\item Letting $f(w) = -p'(w)/p(w)$, we have that $f$ and its first two derivatives are polynomially-bounded: there exists $C > 0$ and an even integer $m \ge 2$ such that $|f^{(\ell)}(w)| \le C + w^m$ for all $0 \le \ell \le 2$.
\item With $m$ as in (iv), $\cP$ has finite moments up to $5m$: $\EE|\cP|^k < \infty$ for all $1 \le k \le 5m$.
\end{enumerate}
\end{assumption}

\noindent An important consequence of assumptions (iv) and (v) is the following.
\begin{lemma}
\label{lem:mom}
$\EE|f^{(\ell)}(\cP)|^q < \infty$ for all $0 \le \ell \le 2$ and $1 \le q \le 5$.
\end{lemma}
\begin{proof}
Using $|a+b|^q \le |2a|^q + |2b|^q = 2^q(|a|^q+|b|^q)$ we have
\begin{equation*} \EE|f^{(\ell)}(\cP)|^q \le \EE|C + \cP^m|^q \le 2^q(C^q + \EE|\cP|^{mq}) < \infty. \qedhere\end{equation*}
\end{proof}

The main theorem of this section is the following.
\begin{theorem}\label{thm:nong-upper}
Let $\lambda \ge 0$ and let $\cX,\cP$ satisfy Assumption~\ref{as:nong-upper}. Let $\hat Y = \sqrt{n}\, Y$ where $Y$ is drawn from $\Wig(\lambda,\cX,\cP)$. Let $f(\hat Y)$ denote entrywise application of the function $f(w) = -p'(w)/p(w)$ to $\hat Y$, except the diagonal entries remain zero. Let
$$F_\cP = \EE[f(\cP)^2] = \int_{-\infty}^\infty\frac{p'(w)^2}{p(w)}dw.$$
\begin{itemize}
\item If $\lambda \le 1/\sqrt{F_\cP}$ then $\frac{1}{\sqrt n}\lambda_{\max}(f(\hat Y)) \to 2\sqrt{F_\cP}$ as $n \to \infty$.
\item If $\lambda > 1/\sqrt{F_\cP}$ then $\frac{1}{\sqrt n}\lambda_{\max}(f(\hat Y)) \to \lambda F_\cP + \frac{1}{\lambda} > 2 \sqrt{F_\cP}$ as $n \to \infty$ and furthermore the top (unit-norm) eigenvector $v$ of $f(\hat Y)$ correlates with the spike:
$$\langle v,x \rangle^2 \ge \frac{(\lambda - 1/\sqrt{F_\cP})^2}{\lambda^2} - o(1) \quad\text{with probability } 1-o(1).$$
\end{itemize}
Convergence is in probability. Here $\lambda_{\max}(\cdot)$ denotes the largest eigenvalue of a matrix.
\end{theorem}
\noindent Note that Lemma~\ref{lem:mom} implies that the expectation defining $F_\cP$ is finite. The following corollary is immediate.

\begin{corollary}\label{cor:nong-noncontig}
Suppose $\cX$ and $\cP$ satisfy Assumption~\ref{as:nong-upper}. If $\lambda > 1/\sqrt{F_\cP}$ then $\Wig(\lambda,\cP,\cX)$ is {\bf not} contiguous to $\Wig(0,\cP)$.
\end{corollary}
\noindent Note that this matches the lower bound of Theorem~\ref{thm:nongauss-lower} provided the prior $\cX$ has $\lambda^*_\cX = 1$. In other words, if for some prior we are able to show that PCA is optimal for Gaussian noise, then modified PCA is optimal for any type of non-Gaussian noise.

\begin{proof}[Proof of Theorem~\ref{thm:nong-upper}]
First we justify a local linear approximation of $f(\hat Y_{ij})$. For $i \ne j$, define the error term $\mathcal{E}_{ij}$ by $$f(\hat Y_{ij}) = f(W_{ij}) + \lambda \sqrt{n} x_i x_j f'(W_{ij}) + \mathcal{E}_{ij}.$$
(Define $\mathcal{E}_{ii} = 0$.) We will show that the operator norm of $\mathcal{E}$ is small: $\|\mathcal{E}\| = o(\sqrt n)$ with probability $1-o(1)$.
Apply the mean-value form of the Taylor approximation remainder: $\mathcal{E}_{ij} = \frac{1}{2} f''(W_{ij} + e_{ij}) \lambda^2 n x_i^2 x_j^2$ for some $|e_{ij}| \le |\lambda \sqrt n x_i x_j|$. Bound the operator norm by the Frobenius norm:
$$\|\mathcal{E}\|^2 \le \|\mathcal{E}\|_F^2 = \frac{\lambda^4 n^2}{4} \sum_{i \ne j} x_i^4x_j^4 f''(W_{ij} + e_{ij})^2 \le \frac{\lambda^4}{4} n^{8\alpha - 2} \sum_{i \ne j} f''(W_{ij} + e_{ij})^2.$$
Using the polynomial bound on $f''$ and the fact $|a+b|^{k} \le 2^{k}(|a|^{k} + |b|^{k})$, we have
\begin{align*}
f''(W_{ij} + e_{ij})^2 &\le (C + (W_{ij} + e_{ij})^{m})^2 \\
&\le 4C^2 + 4(W_{ij} + e_{ij})^{2m} \\
&\le 4C^2 + 4 \cdot 2^{2m} (W_{ij}^{2m} + e_{ij}^{2m}) \\
&\le 4C^2 + 2^{2m+2} (W_{ij}^{2m} + \lambda^{2m}n^{(4\alpha-1)m}) \\
&= 4C^2 + 2^{2m+2} W_{ij}^{2m} + o(1).
\end{align*}
Using finite moments of $W_{ij} \sim \cP$, it follows that
$$\EE\left[\sum_{i \ne j} f''(W_{ij} + e_{ij})^2\right] = \mathcal{O}(n^2)$$
and so $\EE \|\mathcal{E}\|^2 = \mathcal{O}(n^{8\alpha}).$
Since $\alpha < \frac{1}{32}$, Markov's inequality now gives the desired result: with probability $1-o(1)$, $\|\mathcal{E}\|^2 = o(n^{1/4})$ and so $\|\mathcal{E}\| = o(\sqrt n)$.

Our goal will be to show that $f(\hat Y)$ is, up to small error terms, another spiked Wigner matrix. Toward this goal we define another error term: for $i \ne j$, let
$$\Delta_{ij} = \lambda \sqrt{n} x_i x_j \left(f'(W_{ij}) - \mathbb{E}[f'(W_{ij})]\right)$$
so that
\begin{equation}
\label{eq:off-diag}
f(\hat Y_{ij}) = f(W_{ij}) + \lambda \sqrt{n} x_i x_j \mathbb{E}[f'(W_{ij})] + \mathcal{E}_{ij} + \Delta_{ij}.
\end{equation}
(Define $\Delta_{ii} = 0$.) We will show that the operator norm of $\Delta$ is small: $\|\Delta\| = o(\sqrt n)$ with probability $1-o(1)$. Let $A_{ij} = f'(W_{ij}) - \mathbb{E}[f'(W_{ij})]$ so that $\Delta_{ij} = \lambda \sqrt{n} x_i x_j A_{ij}$. (Define $A_{ii} = 0$.) We have $\|\Delta\| \le \lambda n^{-1/2+2\alpha} \|A\|$ because for any unit vector $y$,
\begin{align*}
y^\top \Delta y &= \sum_{i,j} \lambda \sqrt{n} x_i x_j A_{ij} y_i y_j \le \sum_{i,j} \lambda \sqrt{n} z_i A_{ij} z_j \qquad\text{where } z_i = x_i y_i \\
&\le \lambda \sqrt{n}\, \|A\| \cdot \|z\|^2 \le \lambda n^{-1/2+2\alpha} \|A\| \cdot \|y\| = \lambda n^{-1/2+2\alpha} \|A\|.
\end{align*}
Note that $A$ is a Wigner matrix (i.e.\ a symmetric matrix with off-diagonal entries i.i.d.) and so $\|A\| = \mathcal{O}(\sqrt n)$ with probability $1-o(1)$. This follows from \citet{wig-spk} Theorem~1.1, provided we can check that each entry of $A$ has finite fifth moment. But this follows from Lemma~\ref{lem:mom}:
$$\EE|A_{ij}|^5 \le 2^5 \left(\EE|f'(W_{ij})|^5 + |\EE[f'(W_{ij})]|^5\right) < \infty.$$
Now we have $\|\Delta\| = \mathcal{O}(n^{2\alpha}) = o(\sqrt n)$ with probability $1-o(1)$ as desired.

From (\ref{eq:off-diag}) we now have that, up to small error terms, $f(\hat Y)$ is another spiked Wigner matrix:
$$f(\hat Y) = f(W) + \lambda \sqrt{n}\, \EE[f'(\cP)]\, xx^\top + \mathcal{E} + \Delta - \delta$$
where (to take care of the diagonal) we define $f(W)_{ii} = 0$, $\delta_{ij} = 0$, and $\delta_{ii} = \lambda \sqrt{n}\, \EE[f'(\cP)] x_i^2$. Note that the final error term $\delta$ is also small: $\|\delta\| \le \|\delta\|_F = \mathcal{O}(n^{2\alpha}) = o(\sqrt n)$. We now have
$$\frac{1}{\sqrt n} \lambda_{\max}(f(\hat Y)) = \lambda_{\max}\left(\frac{1}{\sqrt n} f(W) + \lambda\, \mathbb{E}[f'(\cP)] \,xx^\top \right) + o(1)$$
and so the theorem follows from known results on the spectrum of spiked Wigner matrices, namely Theorem~1.1 from \citet{wig-spk}. We need to check the following details. First note that the Wigner matrix $f(W)$ has off-diagonal \iid entries that are centered:
$$\mathbb{E}[f(W_{ij})] = \int_{-\infty}^\infty f(w)p(w)dw = \int_{-\infty}^\infty \frac{-p'(w)}{p(w)}p(w)dw = -\int_{-\infty}^\infty p'(w)dw = p(-\infty) - p(\infty) = 0.$$
Each off-diagonal entry of $f(W)$ has variance
$$\mathbb{E}[f(W_{ij})^2] = F_\cP.$$
The rank-1 deformation $\lambda\, \mathbb{E}[f'(\cP)] \,xx^\top$ has top eigenvalue $\lambda\, \mathbb{E}[f'(\cP)] \cdot \|x\|^2$. Recall that $\|x\|^2 \to 1$ in probability. Also,
$$f'(w) = \frac{d}{dw} \frac{-p'(w)}{p(w)} = -\frac{p''(w)p(w) - p'(w)^2}{p(w)^2}$$
and so
$$\mathbb{E}[f'(\cP)] = \int_{-\infty}^\infty f'(w)p(w)dw = \int_{-\infty}^\infty \left[-p''(w) + \frac{p'(w)^2}{p(w)}\right]dw = \int_{-\infty}^\infty \frac{p'(w)^2}{p(w)}dw = F_\cP.$$
Therefore the top eigenvalue of the rank-1 deformation converges in probability to $\lambda F_\cP$. Finally, by Lemma~\ref{lem:mom}, the entries of $f(W)$ have finite fifth moment.

The desired convergence of the top eigenvalue now follows. It remains to show that when $\lambda > 1/\sqrt{F_\cP}$, the top eigenvalue of $f(\hat Y)$ correlates with the planted vector $x$. Let $v$ be the top eigenvector of $f(\hat Y)$ with $\|v\| = 1$. From above we have
$$v^\top \left(\frac{1}{\sqrt n} f(\hat Y)\right)v = v^\top \left(\frac{1}{\sqrt n} f(W)\right)v + \lambda F_\cP \langle v,x \rangle^2 + o(1).$$
We know $\frac{1}{\sqrt n} f(\hat Y)$ has top eigenvalue $\lambda F_\cP + 1/\lambda + o(1)$ and $\frac{1}{\sqrt n} f(W)$ has top eigenvalue $2\sqrt{F_\cP} + o(1)$, which yields
\begin{align*}
\langle v,x \rangle^2 &\ge \frac{1}{\lambda F_\cP}(\lambda F_\cP + 1/\lambda - 2\sqrt{F_\cP})-o(1)\\
&= \frac{(\lambda - 1/\sqrt{F_\cP})^2}{\lambda^2} - o(1).\qedhere
\end{align*}

\end{proof}


\section{Spiked Wishart models}
\label{sec:wish}

In this section we consider the spiked Wishart model, which apart from the signal-to-noise ratio $\beta$ has a scaling parameter $\gamma$. In Section~\ref{sec:wishmain} we define the model and state our main results. In Section~\ref{sec:wishsecond} and Section~\ref{sec:rate} we develop an understanding of the relevant second moment through large deviations behavior together with comparison to the Wigner model. In Section~\ref{sec:wishcontig} we show contiguity results which among other things establish that for any $\gamma \leq \frac{1}{3}$, PCA is optimal for the Rademacher prior. On the other hand, in Section~\ref{sec:wishup} we show that for any $\gamma \geq 0.698$, PCA is suboptimal for the Rademacher prior with a negative spike ($\beta < 0$), and that there is a computationally inefficient procedure that solves this detection problem even when PCA fails. This pair of results suggests that a phase transition occurs at some $\gamma$.

\subsection{Main results}\label{sec:wishmain}

 We first formally define the spiked Wishart model:
\begin{definition}\label{def:spiked-wishart}
The spiked (Gaussian) Wishart model $\Wish(\gamma,\beta,\cX)$ on $n \times n$ matrices is defined thus: we first draw a hidden spike $x \sim \cX$, and then reveal $Y = X X^\top$, where $X$ is an $n \times N$ matrix whose columns are sampled independently from $\cN(0,I+\beta x x^\top)$; the parameters $N$ and $n$ scale proportionally with $n/N \to \gamma$ (the high-dimensional regime).
\end{definition}

\noindent Recall that the prior $\cX$ is required to obey the normalization $\|x\| \to 1$ in probability (see Definition~\ref{def:prior}).

In this high-dimensional setting, the spectrum bulk converges in the unspiked case to the Marchenko--Pastur distribution with shape parameter $\gamma$. By results of \citet{bbp} and \citet{baik-silverstein}, it is known that the top eigenvalue distinguishes the spiked and unspiked models when $\beta > \sqrt{\gamma}$. In fact, matching lower bounds are known for the spherical prior, due to \citet{sphericity}: for $0 \leq \beta < \sqrt{\gamma}$, no hypothesis test distinguishes this spiked model from the unspiked model with $o(1)$ error. In the case of $-1 \leq \beta < 0$, a corresponding PCA threshold exists: the minimum eigenvalue exits the bulk when $\beta < -\sqrt{\gamma}$ \citep{baik-silverstein}, but we are not aware of lower bounds in the literature. The case of $\beta < -1$ is of course invalid, as the covariance matrix must be positive semidefinite.

We approach contiguity for the spiked Wishart model through the second moment method outlined in Section~\ref{sec:contig}. Note that detection and recovery can only become easier given the original sample matrix $X$ (instead of $XX^\top$), so we establish the stronger statement that the spiked distribution on $X$ is contiguous to the unspiked distribution. We first simplify the second moment in high generality:
\begin{repproposition}{prop:wishart-2mom}
With $P_n$ and $Q_n$ the spiked and unspiked models, respectively, we have
$$ \Ex_{Q_n}\left[ \left(\dd[P_n]{Q_n} \right)^2 \right] = \EE_{x,x' \sim \cX}\left[ \left( 1 - \beta^2 \langle x,x' \rangle^2 \right)^{-N/2} \right].$$
\end{repproposition}
\noindent It is worth noting that this expression has the curious property of symmetry under replacing $\beta$ with $-\beta$. In contrast, the original Wishart model does not. This is a limitation of our methods since they cannot distinguish between positive and negative $\beta$, even though the resulting Wishart distributions are quite different. (In Appendix~\ref{app:wish-nc}, however, we show how a more sophisticated conditioning method can break this symmetry and obtain stronger results.) For the \iid Rademacher prior with either positive or negative $\beta$, we show:
\begin{theorem-nolabel}[see Proposition~\ref{prop:wishart-z2-spectral-tight} and Theorem~\ref{thm:wishart-mle-z2}]
When $\cX$ is the \iid Rademacher prior, and $\gamma \leq \frac13$, then the spectral threshold is tight: the spiked and unspiked models are contiguous when $|\beta| < \sqrt{\gamma}$. However, for large enough $\gamma < 1$, spectral is not tight in the $\beta < 0$ case: in fact for any prior $\cX_n$ supported on at most $c^n$ points, there is a computationally inefficient procedure that distinguishes between the spiked Wishart model $\Wish(\gamma,\beta,\cX)$ and the unspiked model $\Wish(\gamma)$, with $o(1)$ probability of error, whenever
$$ (-\beta) + \log(1-(-\beta)) < -2\gamma \log c. $$
\end{theorem-nolabel}
\noindent In particular, for all $\gamma > 0.698$, this second statement shows that the Rademacher-spiked and unspiked Wishart models can be distinguished for some $\beta \in (-\sqrt{\gamma},0)$; thus there is some threshold $\gamma^* \in [\frac13, 0.698]$ at which the spectral method ceases to be tight in the case of $\beta < 0$. To our knowledge, this phase transition has not been previously observed; it is somewhat surprising that PCA becomes suboptimal for a prior as natural as \iid Rademacher.

In the case of positive $\beta$ with the Rademacher prior, the PCA threshold is always optimal. We show this in Appendix~\ref{app:wish-nc}.

\begin{figure}[!ht]\centering
\begin{minipage}{0.8\textwidth}\centering
\includegraphics[scale=1.0]{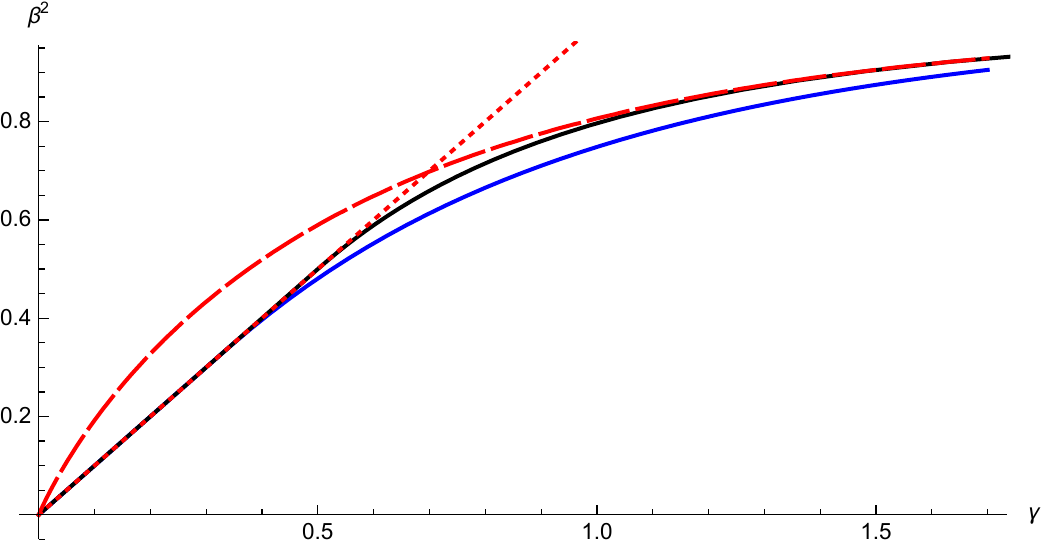}
\caption{Wishart model spiked with Rademacher prior, for $\beta < 0$. Above the dotted red line, the bottom eigenvalue distinguishes spiked from unspiked; below, the bottom eigenvalue fails. Above the dashed red curve, MLE distinguishes spiked from unspiked. Below the solid black curve, the spiked model is contiguous to the unspiked, and no hypothesis test achieves full power; this is the stronger noise-conditioned argument of Appendix~\ref{app:wish-nc}. The blue, lowest curve is the weaker lower bound of Theorem~\ref{thm:wishart-ld}.}
\end{minipage}
\end{figure}

We obtain our lower bounds in the Rademacher setting through a combination of comparison to the Wigner model as well as large deviations theory. In general, the best second moment lower bound for a prior involves an optimization problem involving the large deviations rate function of $\langle x,x' \rangle$, the correlation of two independent spikes drawn from $\cX$. A precise result to this end is given in Theorem~\ref{thm:wishart-ld}.

A simpler, looser bound can be achieved using only comparison to the Wigner model:
\begin{claim-nolabel}[see Remark~\ref{rmk:wigner-wishart-bound}]
The $\cX$-spiked model is contiguous to the unspiked model whenever $\beta^2 < 1-e^{-\gamma (\lambda^*_\cX)^2}$, where $\lambda^*_\cX$ is the threshold for $\lambda$ beyond which the second moment becomes unbounded in the Wigner setting.
\end{claim-nolabel}
\noindent In particular, whenever $\lambda^*_\cX = 1$, so that the spectral method is optimal in the Wigner setting, it follows that the ratio between the above lower bound (Remark~\ref{rmk:wigner-wishart-bound}) and the spectral upper bound tends to $1$ as $\gamma \to 0$.

\subsection{Second moment computation}\label{sec:wishsecond}

\begin{proposition}\label{prop:wishart-2mom}
Let $\cX$ be a spike prior. In distribution $P_n$, let a hidden spike $x$ be drawn from $\cX$, and let $N$ independent samples $y_i$, $1 \leq i \leq N$, be revealed from the normal distribution $\cN(0, I_{n \times n} + \beta x x^\top)$. In distribution $Q_n$, let $N$ independent samples $y_i$, $1 \leq i \leq N$, be revealed from $\cN(0,I_{n \times n})$. Then we have
$$ \EE_{Q_n}\left[ \left(\dd[P_n]{Q_n} \right)^2 \right] = \EE_{x,x' \sim \cX}\left[ \left( 1 - \beta^2 \langle x,x' \rangle^2 \right)^{-N/2} \right].$$
\end{proposition}
\noindent In particular, if this second moment remains bounded as $N,n \to \infty$ (with $n = \gamma N$), then we have $P_n \contig Q_n$, and thus $\Wish(\gamma,\beta,\cX) \contig \Wish(\gamma)$.

\begin{proof}
We first compute:
\begin{align*}
\dd[P_n]{Q_n}(y_1,\ldots,y_N) &= \EE_{x' \sim \cX}\left[ \prod_{i=1}^n \frac{\exp(-\frac12 y_i^\top (I + \beta x' (x')^\top)^{-1} y_i)}{\sqrt{\det(I + \beta x' (x')^\top)} \exp(-\frac12 y_i^\top y_i)} \right]\\
&= \EE_{x'}\left[ \det(I + \beta x' (x')^\top)^{-N/2} \prod_{i=1}^N \exp\left( -\frac12 y_i^\top ((I + \beta x' (x')^\top)^{-1} - I) y_i \right) \right].
\intertext{Note that $(I + \beta x' (x')^\top)^{-1}$ has eigenvalue $(1+\beta |x'|^2)^{-1}$ on $x'$ and eigenvalue $1$ on the orthogonal complement of $x'$. Thus $(I + \beta x' (x')^\top)^{-1} - I = \frac{-\beta}{1+\beta |x'|^2} x' (x')^\top$, and we have:}
&= \EE_{x'} \left[ (1+\beta |x'|^2)^{-N/2} \prod_{i=1}^N \exp\left( \frac12 \frac{\beta}{1+\beta |x'|^2} y_i^\top x' (x')^\top y_i \right) \right]\\
&= \EE_{x'} \left[ (1+\beta |x'|^2)^{-N/2} \prod_{i=1}^N \exp\left( \frac12 \frac{\beta}{1+\beta |x'|^2} \langle y_i, x' \rangle^2 \right) \right]. \numberthis\label{eq:wishart-dpdq}
\end{align*}

Passing to the second moment, we compute:
\begin{align*}
\EE_{Q_n}\left[ \left(\dd[P_n]{Q_n} \right)^2 \right] &= \EE_{P_n}\left[ \dd[P_n]{Q_n} \right] \\
&= \EE_{x,x'}\left[ (1+\beta |x'|^2)^{-N/2} \prod_{i=1}^N \EE_{y_i \sim \N(0,I+\beta x x^\top)} \exp\left( \frac12 \frac{\beta}{1+\beta |x'|^2} \langle y_i, x' \rangle^2 \right) \right].
\intertext{Over the randomness of $y_i$, we have $\langle y_i, x' \rangle \sim \cN(0, |x'|^2 + \beta \langle x,x' \rangle^2)$, so that the inner expectation is a moment generating function of a $\chi_1^2$ random variable:}
&= \EE_{x,x'}\left[ (1+\beta |x'|^2)^{-N/2} \prod_{i=1}^N \left( 1 - \frac{\beta}{1+\beta |x'|^2} (|x'|^2 + \beta \langle x,x' \rangle^2) \right)^{-1/2} \right] \\
&= \EE_{x,x'}\left[ \left( 1 - \beta^2 \langle x,x' \rangle^2 \right)^{-N/2} \right], \numberthis\label{eq:wishart-moment}
\end{align*}
as desired.
\end{proof}

\subsection{Rate functions}\label{sec:rate}
In this subsection, we develop an understanding of the second moment (\ref{eq:wishart-moment}), through a combination of large deviations theory and comparison to the Wigner model. Recall from equation (\ref{eq:lambdastar}):
$$ \lambda^*_{\cX} = \sup \left\{ \lambda \mid \EE_{x,x' \sim \cX} \exp\left( \frac{\lambda^2 n}{2} \langle x,x' \rangle^2\right) \text{ is bounded as $n \to \infty$} \right\}. $$
and suppose that the following deviations
$$ f_{n,\cX}(t) = -\frac1n \log \Pr_{x,x' \sim \cX}[\langle x,x' \rangle^2 \geq t] $$
converge pointwise to a limit $f_\cX(t) \in [0,\infty]$. We require the Chernoff-style bound
\begin{equation}\label{eq:chernoff-rate-assumption} f_{n,\cX}(t) \geq f_\cX(t).\end{equation}
We furthermore assume that $f_\cX$ is lower semi-continuous on $(0,1+\delta_f]$ for some $\delta_f > 0$. We call such $f_{\cX}$ the \emph{rate function} of the prior $\cX$.

\begin{theorem}\label{thm:wishart-ld}
Let the spike prior $\cX$ have rate function $f_\cX$ which is finite on $(0,1)$.
\begin{enumerate}[(i)]
\item Suppose that $\beta^2 < 1$, that $\beta^2 / \gamma < (\lambda_{\cX}^*)^2$, and that $f_{\cX}(t) > \frac{-1}{2\gamma} \log(1-\beta^2 t)$ for all $t \in (0,1)$. Then $\Wish(\gamma,\beta,\cX) \contig \Wish(\gamma)$.
\item If $\beta^2 > 1$, or if $\beta^2/\gamma > (\lambda^*_\cX)^2$, or if $f_{\cX}(t) < \frac{-1}{2\gamma} \log(1-\beta^2 t)$ for some $t \in (0,1)$, then the second moment (\ref{eq:wishart-moment}) is unbounded.
\end{enumerate}
\end{theorem}
Part (ii) is a partial converse to part (i), though unboundedness of the second moment does not imply that the distributions are not contiguous. In fact, in Appendix~\ref{app:wish-nc} we show how to obtain stronger lower bounds by conditioning away from certain bad events in the second moment.
\begin{proof}
{\bf Part (i).} Let $\delta, \eps > 0$, to be chosen later. With $x \sim \cX$, we have that $\|x\|_2 \to 1$ in probability (by Definition~\ref{def:prior}),
so that $\Pr[\|x\|_2 > 1+\delta] = o(1)$. Let $\cX'$ be the spike prior given by sampling from $\cX$ but instead returning the zero vector if $\|x\| > 1+\delta$. It suffices to bound the second moment (\ref{eq:wishart-moment}) applied to $\cX'$.

We can split the second moment as follows:
\begin{align*} \Ex_{x,x' \sim \cX'}\left[ (1-\beta^2\langle x,x' \rangle^2)^{-n/2\gamma} \right] &= \EE_{x,x'}\left[ (1-\beta^2\langle x,x' \rangle^2)^{-n/2\gamma} \;\one[\langle x,x' \rangle^2 \leq \eps] \right] \\
&+ \EE_{x,x'}\left[ (1-\beta^2\langle x,x' \rangle^2)^{-n/2\gamma} \;\one[\langle x,x' \rangle^2 > \eps] \right]. \numberthis\label{eq:deviations-split}
\end{align*}
We control these two parts separately. For the first part, we compute:
\begin{align*}
\EE_{x,x'} (1-\beta^2\langle x,x' \rangle^2)^{-n/2\gamma} \;\one[\langle x,x' \rangle^2 \leq \eps] &= \EE_{x,x'} \exp\left(\frac{-n}{2\gamma} \log(1-\beta^2\langle x,x' \rangle) \right) \;\one[\langle x,x' \rangle^2 \leq \eps] \\
&= \EE_{x,x'} \exp\left( \frac{-n}{2\gamma \eps^2} \log(1-\eps^2\beta^2) \langle x,x' \rangle^2 \right) \;\one[\langle x,x' \rangle^2 \leq \eps] \\
&\leq \EE_{x,x' \sim \cX'} \exp\left( \frac{-n}{2\gamma \eps^2} \log(1-\eps^2\beta^2) \langle x,x' \rangle^2 \right) \\
&\leq \EE_{x,x' \sim \cX} \exp\left( \frac{-n}{2\gamma \eps^2} \log(1-\eps^2\beta^2) \langle x,x' \rangle^2 \right) \numberthis\label{eq:parti-spectral-log} \\
&\leq \EE_{x,x' \sim \cX} \exp\left( \frac{n \beta^2}{2 \gamma(1-\beta^2 \eps^2)} \langle x,x' \rangle^2 \right), \numberthis\label{eq:parti-spectral-frac}
\end{align*}
using the bound $\log t \geq 1 - 1/t$. This contribution is bounded so long as
$$ \frac{\beta^2}{\gamma(1-\beta^2\eps^2)} < (\lambda^*_\cX)^2,
\text{ or equivalently, }
\frac{1}{\beta^2} - \frac{1}{\gamma(\lambda^*_\cX)^2} > \eps^2. $$
We can choose such $\eps > 0$ so long as $\beta^2 / \gamma < (\lambda^*_\cX)^2$, as assumed.

We control the second expectation in (\ref{eq:deviations-split}) as follows:
\begin{align*}
&\EE (1-\beta^2\langle x,x' \rangle^2)^{-n/2\gamma} \;\one[\langle x,x' \rangle^2 > \eps] \\
&= \int_0^\infty \Pr_{\cX'}\left[ (1-\beta^2\langle x,x' \rangle^2)^{-n/2\gamma} \;\one[\langle x,x' \rangle^2 > \eps] \geq u \right] \,\dee u.
\intertext{Substitute $u = (1-\beta^2 t)^{-n/2\gamma}$, so that $\dee u = \frac{\beta^2 n}{2\gamma} (1-\beta^2 t)^{-1-n/2\gamma} \,\dee t$:}
&= \int_0^\infty \frac{\beta^2 n}{2\gamma} (1-\beta^2 t)^{-1-p/2\gamma} \Pr_{\cX'}\left[ \langle x,x' \rangle^2 \;\one[\langle x,x' \rangle^2 > \eps] \geq t \right] \,\dee t \\
&\leq \int_0^\infty \frac{\beta^2 n}{2\gamma} (1-\beta^2 t)^{-1-n/2\gamma} \Pr_{\cX'}\left[ \langle x,x' \rangle^2 \geq \max(\eps,t) \right] \,\dee t \\
&= \int_0^{1+\delta} \frac{\beta^2 n}{2\gamma} (1-\beta^2 t)^{-1-n/2\gamma} \Pr_{\cX'}\left[ \langle x,x' \rangle^2 \geq \max(\eps,t) \right] \,\dee t \\
&\leq \int_0^{1+\delta} \frac{\beta^2 n}{2\gamma} (1-\beta^2 t)^{-1-n/2\gamma} \Pr_\cX\left[ \langle x,x' \rangle^2 \geq \max(\eps,t) \right] \,\dee t \\
&= \frac{\beta^2}{2\gamma} \int_0^{1+\delta} \exp\left(n\left( \frac{\log n}{n} - \left( \frac{1}{2\gamma} + \frac{1}{n} \right) \log(1-\beta^2 t) - \left( -\frac{1}{n} \log \Pr_\cX[\langle x,x' \rangle^2 \geq \max(\eps,t)] \right)  \right)\right) \,\dee t.
\end{align*}
As $\beta^2 < 1$, we can choose $\delta \leq \delta_f$ sufficiently small so that $\beta^2 (1+\delta) < 1$, so that $\log(1-\beta^2 t)$ is bounded on $[0,1+\delta]$. Then $-\log(1-\beta^2 t)/2\gamma - f_\cX(\max(\eps,t))$ is negative and upper semi-continuous on $[0,1+\delta]$, so that the supremum
$$ m = \sup_{t \in [\eps,1+\delta]} -\log(1-\beta^2 t)/2\gamma - f_\cX(\max(\eps,t)) $$
is attained. As $f_\cX(t) > \frac{-1}{2\gamma} \log(1-\beta^2 t)$ on $(0,1)$, by semi-continuity we can assume the same on $[\eps,1+\delta]$ (perhaps decreasing $\delta$), so that $m < 0$. Using the bound (\ref{eq:chernoff-rate-assumption}), we have for all $n$,
$$-\log(1-\beta^2 t)/2\gamma - f_{n,\cX}(\max(\eps,t)) \leq m < 0 \quad\text{ on } [\eps,1+\delta],$$
so that
$$ \EE\left[ (1-\beta^2\langle x,x' \rangle^2)^{-n/2\gamma} \;\one[\langle x,x' \rangle^2 > \eps] \right] \leq \frac{\beta^2}{2\gamma} \int_0^{1+\delta} \exp\left(n\left( \frac{\log n}{n} - \frac{1}{n} \log(1-\beta^2 t) + m  \right)\right) \,\dee t. $$
As $\frac{\log n}{n} - \frac{1}{n} \log(1-\beta^2 t) - m$ is bounded below $0$ for all sufficiently large $n$, the integral above tends to $0$ as $n \to \infty$, as desired.

{\bf Part (ii).} Suppose first that $\beta^2/\gamma > (\lambda^*_\cX)^2$. We bound the second moment (\ref{eq:wishart-moment}) as follows:
\begin{align*}
\EE_{x,x' \sim \cX}\left[ (1-\beta^2 \langle x,x' \rangle^2 )^{-n/2\gamma} \right] &= \EE_{x,x'}\left[ \exp\left( \frac{-n}{2\gamma} \log(1-\beta^2\langle x,x' \rangle^2) \right) \right] \\
&\geq \EE_{x,x'}\left[ \exp\left( \frac{n \beta^2}{2\gamma} \langle x,x' \rangle^2 \right) \right],
\end{align*}
which is unbounded as $n \to \infty$ as $\beta^2/\gamma > (\lambda^*_\cX)^2$.

Next, suppose that $f_\cX(t) < \frac{-1}{2\gamma} \log(1-\beta^2 t)$ for some $t \in (0,1]$. Then there exists $\eps > 0$ so that, for all sufficiently large $n$,
\begin{equation}
\frac{-1}{n} \log \Pr_{x,x' \sim \cX}[\langle x,x \rangle^2 \geq t] \leq -\eps - \frac{-1}{2\gamma} \log(1-\beta^2 t). \label{eq:finp-ldbound}
\end{equation}
we bound the second moment as follows:
\begin{align*}
\EE_{x,x' \sim \cX}\left[ (1-\beta^2 \langle x,x' \rangle^2 )^{-n/2\gamma} \right] &\geq \Pr[\langle x,x' \rangle^2 \geq t] (1-\beta^2 t)^{-n/2\gamma} \\
&= \exp\left( n \left( \frac1n \log \Pr[\langle x,x' \rangle^2 \geq t] - \frac{1}{2\gamma} \log(1-\beta^2 t) \right) \right) \\
&\geq \exp\left( n \eps \right)
\end{align*}
which is unbounded as $n \to \infty$.

Finally, suppose that $\beta^2 > 1$. Note that $\frac{-1}{2\gamma} \log(1-\beta^2 t)$ tends to infinity as $t \to \beta^{-2}$ from below, whereas $f_{\cX}(t)$ can only become infinite for $t \geq 1$. Hence we must have $f_\cX(t) < \frac{-1}{2\gamma} \log(1-\beta^2 t)$ for some $t < \beta^{-2}$, so that the second moment is unbounded by the argument above.
\end{proof}

\begin{remark}\label{rmk:wigner-wishart-bound}
Certain parts of the argument above do not rely on the existence of a rate function $f_\cX(t)$. Particularly, in part (i), equation (\ref{eq:parti-spectral-log}) demonstrates that we can bound the first expectation with some $\eps > 1$ so long as $-\log(1-\beta^2)/\gamma < (\lambda^*_\cX)^2$, or equivalently $\beta^2 < 1-e^{-\gamma (\lambda^*_\cX)^2}$. In this case, we can choose $\delta < \eps-1$, so that the integral in the second part vanishes entirely, as $\Pr[\langle x,x' \rangle^2 \geq \eps] = 0$. Thus we obtain $\Wish(\gamma,\beta,\cX) \contig \Wish(\gamma)$ whenever $\beta^2 < 1-e^{-\gamma (\lambda^*_\cX)^2}$, for every spike prior $\cX$, without assuming existence of a rate function.
\end{remark}

\begin{remark}
Suppose we have a bound on the Wigner second moment, that
$$ \limsup_{n \to \infty} \EE_{x,x'\sim \cX}\exp\left( \frac{\lambda^2 n}{2} \langle x,x' \rangle^2 \right) \leq B_\cX(\lambda^2). $$
Then from the argument of Theorem~\ref{thm:wishart-ld}, we obtain an explicit bound on the asymptotic second moment as $n \to \infty$: so long as the conditions of part (i) are satisfied, the second integral in (\ref{eq:deviations-split}) is $o(1)$, and the first integral is bounded (\ref{eq:parti-spectral-frac}) by
$$ B_\cX\left( \frac{\beta^2}{\gamma(1-\beta^2\eps^2)} \right). $$
Assuming continuity of $B_\cX$, we can now take $\eps \to 0$ and bound the Wishart second moment by
$ B_\cX(\beta^2/\gamma)$, in order to obtain explicit hypothesis testing error bounds through Proposition~\ref{prop:hyptest}. Non-asymptotic bounds might also be obtained from this line of argument, but we do not pursue this.
\end{remark}

\subsection{\texorpdfstring{Application: \iid priors}{Application: i.i.d. priors}}\label{sec:wishcontig}
Here we discuss the existence of rate functions for \iid priors, and explore contiguity for the $\ZZ/2$ prior in particular.

\begin{claim}Let $\pi$ be a distribution on $\RR$. Suppose that $x_1 x_2$ has a moment generating function $m(\theta) = \EE[e^{\theta x_1 x_2}]$ defined on all\footnote{There are acceptable weaker conditions for this statement.} of $\RR$, where $x_1, x_2 \sim \pi$ independently. Then $\cX = \IID(\pi)$ has rate function $f_\cX(t) = I(\sqrt{t})$, where
$$ I(u) = \sup_{v} u\theta - \log m(\theta). $$
\end{claim}
\begin{proof}
This pointwise convergence of deviations follows from Cram\'er's theorem. The bound (\ref{eq:chernoff-rate-assumption}) follows from the Chernoff bound. $I$ is lower semi-continuous by virtue of being defined as a convex dual, and so the same is true of $f_\cX$.
\end{proof}

\begin{lemma}
The prior $\cX = \IID(\{\pm 1\})$ has rate function 
$$ f_{\cX}(t) = \frac12 (1+\sqrt{t}) \log(1+\sqrt{t}) + \frac12 (1-\sqrt{t}) \log(1-\sqrt{t}) = \log 2 - H\left(\frac{1+\sqrt{t}}{2}\right), $$
where $H$ is the binary entropy. For $t > 1$ we have $f_{\cX}(t) = \infty$.
\end{lemma}
\begin{proof}
We take the convex dual of the cumulant generating function $\log m(\theta) = \log(\cosh(\theta))$, solving the optimization problem via first-order optimality.
\end{proof}

\begin{proposition}\label{prop:wishart-z2-spectral-tight}
Let $\cX = \IID(\{\pm 1\})$. If $\gamma \leq \frac13$ and $\beta^2 < \gamma$, then $f_\cX(t) > \frac{-1}{2\gamma} \log(1-\beta^2 t)$, so that $\Wish(\gamma,\beta,\IID(\{\pm 1\})) \contig \Wish(\gamma)$.
\end{proposition}
\noindent In other words, the spectral upper bound, which achieves recovery when $\beta^2 > \gamma$, is tight when $\gamma \leq \frac13$. 
\begin{proof}
We wish to see that $f_\cX(t) + \frac{1}{2\gamma} \log(1-\beta^2 t) > 0$ for all $t \in (0,1)$. Indeed, its Taylor series takes the form
\begin{equation}\label{eq:wishart-z2-series} \sum_{i \geq 1} \frac{1}{2i} \left( \frac{1}{2i-1} - \frac{\beta^{2i}}{\gamma} \right) t^i, \end{equation}
convergent on $[0,1)$ since $\beta^2 < \gamma < 1$. We will argue that each term is positive. Since $\beta^2 < \gamma$, we have
$$ \frac{1}{2i} \left( \frac{1}{2i-1} - \frac{\beta^{2i}}{\gamma} \right) > \frac{1}{2i} \left( \frac{1}{2i-1} - \gamma^{i-1} \right), $$
so it suffices that $\gamma^{i-1} \leq 1/(2i-1)$ for all $i \geq 1$. This is clearly the case for $i=1$, and for $i \geq 2$, the strictest of the constraints $\gamma \leq (2i-1)^{-1/(i-1)}$ occurs when $i=2$. It then suffices that $\gamma \leq \frac13$, as assumed. The contiguity result now follows from Theorem~\ref{thm:wishart-ld}(i), noting that $f_\cX$ is finite on $(0,1)$, and that $\lambda^*_\cX = 1$ as computed in Theorem~\ref{thm:pmone-wigner}.
\end{proof}
We have a partial converse:
\begin{proposition}
Let $\cX = \IID(\{\pm 1\})$. For $\gamma > \frac13$, there exists $\beta^2 < \gamma$ for which the second moment (\ref{eq:wishart-moment}) diverges. Further, whenever $\beta^2 > 1-e^{-(2\log 2)\gamma}$, the second moment diverges.
\end{proposition}
\begin{proof}
For the first assertion, note that if we take $\beta^2 = \gamma$, then from the series expansion (\ref{eq:wishart-z2-series}), $f_\cX(t) + \frac{1}{2\gamma} \log(1-\beta^2 t)$ has vanishing $t^0$ and $t^1$ coefficients and negative $t^2$ coefficient for $\gamma \geq \frac13$. It follows that there exists some $t > 0$ for which this quantity is negative. By continuity, this statement remains true if we fix $\gamma$ and decrease $\beta$ a sufficiently small amount. The assertion now follows from Theorem~\ref{thm:wishart-ld}(ii).

The condition on the second assertion is precisely that $f_\cX(t) + \frac{1}{2\gamma} \log(1-\beta^2 t)$ is negative at $t=1$, as $f_\cX(1) = \log 2$. Hence this assertion follows also from Theorem~\ref{thm:wishart-ld}(ii).
\end{proof}

The previous two propositions suggest a threshold behavior, where the spectral method is tight for $\gamma \leq \frac13$ but might not be for larger $\gamma$. In the next section, we verify that the spectral method becomes suboptimal for $\gamma > 0.698$ in the negatively-spiked case where $\beta < 0$. When $\beta > 0$ the spectral threshold is always optimal; we show this in Appendix~\ref{app:wish-nc}.

\subsection{Detecting a negative spike}\label{sec:wishup}

In this subsection, we analyze the performance of a certain computationally inefficient test for the detection problem. Note the following well-known large deviations behavior for $\chi^2$ distributions, which follows from Cram\'er's theorem:
\begin{lemma}\label{lemma:chi2rate}
For all $z < 1$ and $c > 0$,
$$ \lim_{p \to \infty} \frac1p \Pr\left[ \chi_{p}^2 < zp \right] = \frac12(1-z+\log z). $$
\end{lemma}

\begin{theorem}\label{thm:wishart-mle-z2} Let $\beta < 0$. Let $\cX_n$ be any prior supported on at most $c^n$ points, for some fixed $c$. Then there is a computationally inefficient procedure that distinguishes between the spiked Wishart model $\Wish(\gamma,\beta,\cX))$ and the unspiked model $\Wish(\gamma)$, with $o(1)$ probability of error, whenever
$$ (-\beta) + \log(1-(-\beta)) < -2\gamma \log c, $$
\end{theorem}

\begin{proof}
Given a matrix $Y$, consider the test statistic $T = \min_{x \in \supp \cX_n} \frac1n x^\top Y x$. Under $Y \sim \Wish(\gamma,\beta,\cX)$ with true spike $x^*$, we have that $\frac1n (x^*)^\top Y x^* \sim (1+\beta) \chi_{n/\gamma}^2$, which converges in probability to $(1+\beta)/\gamma$. Hence, for all $\eps > 0$, we have that $T < (1+\beta+\eps)/\gamma$ with probability $1-o(1)$ under the spiked model $\Wish(\gamma,\beta,\cX)$.

Under the unspiked model, we have
\begin{align*}
\Pr[T \leq (1+\beta+\eps)/\gamma] &\leq \sum_{x \in \supp \cX} \Pr[x^\top Y x > (1+\beta-\eps)n/\gamma] \\
&\leq c^n \Pr\left[\chi_{n/2\gamma}^2 \leq (1+\beta+\eps)n/\gamma \right] \\
&= \exp\left(n\left( \log c + \frac1n \Pr\left[\chi_{n/2\gamma}^2 \leq (1+\beta+\eps)n/\gamma\right] \right)\right).
\end{align*}
This is $o(1)$ so long as
\begin{align*}
0 &> \log c + \lim_{n \to \infty} \frac1n \Pr\left[\chi_{n/2\gamma}^2 \leq (1+\beta+\eps)n/\gamma \right] \\
&= \log c + \frac{1-(1+\beta+\eps)+\log(1+\beta+\eps)}{2\gamma} \quad\text{by Lemma~\ref{lemma:chi2rate};} \\
-2\gamma\log c &> -\beta-\eps + \log(1+\beta+\eps).
\end{align*}
We can choose such $\eps > 0$ precisely under the hypothesis of this theorem.

Hence, by thresholding the statistic $T$ at $(1+\beta+\eps)/\gamma$, we obtain a hypothesis test that distinguishes $Y \sim \Wish(\gamma,\beta,\cX)$ from $Y \sim \Wish(\gamma)$, with probability $o(1)$ of error of either type.
\end{proof}

\begin{remark}
The upper bound of Theorem~\ref{thm:wishart-mle-z2} is satisfied for some $-\beta < 1$ when $\gamma = 1$; by continuity, it is satisfied for some $-\beta < \sqrt{\gamma}$ for all $\gamma < 1$ sufficiently large. Hence this theorem demonstrates that for each prior $\cX$ of exponential-sized support, the spectral upper bound ceases to be optimal in the negative $\beta$ case for some critical $\gamma \in [0,1)$. To the best of our knowledge, this phenomenon has not appeared previously in the literature.
\end{remark}

The scenario above applies in particular to the Rademacher prior, where the critical $\gamma$ lies between $0.697$ and $0.698$. Other examples of exponential-sized priors include the sparse Rademacher prior.


\section{Synchronization over finite and infinite groups}
\label{sec:synch}

In this section we study synchronization problems over compact groups. In Section~\ref{sec:synchmain} we state our main results. In Section~\ref{sec:tohlower} we study the truth-or-Haar model which was previously introduced, and establish contiguity results. In Section~\ref{sec:gsynch} we introduce a Gaussian sychronization model. We consider this one of the main contributions of this section, since it allows us to define interesting detection and recovery problems over infinite groups which is made possible through incorporating the appropriate notions from representation theory. In Section~\ref{sec:gsynchsecond}, Section~\ref{sec:gsynchsub} and Section~\ref{sec:gsynchcond} we establish methods for proving contiguity in the Gaussian model; similarly to the Gaussian Wigner model we have a sub-Gaussian method and a conditioning method. In Section~\ref{sec:gsynchup1} and Section~\ref{sec:gsynchup2} we give computationally inefficient procedures that succeed below the spectral threshold, for both the truth-or-Haar and Gaussian synchronization models.

\subsection{Main results}\label{sec:synchmain}

First, let us motivate and then define the \emph{truth-or-Haar} model. In earlier sections, we studied the problem of recovering the spike $x$ given
$$\frac{\lambda}{n} xx^\top + \frac{1}{\sqrt n} W$$
where the entries of $x$ are \iid from the Rademacher distribution. We can instead view these coordinates as elements of the group $\mathbb{Z}/2$. Then each entry in the observed matrix is a noisy measurement of the relative group element $x_i x_j^{-1}$. This suggests an entire family of recovery questions. What if we are given noisy measurements of $x_i x_j^{-1}$ where $x_i$ and $x_j$ belong to some group $G$? One simple noise model for such a problem is the following.

\begin{definition}
Let $G$ be a finite group and let $\tilde p \ge 0$. In the \emph{truth-or-Haar} model $\ToH(\tilde p, G)$ we first draw a vector $g \in G^n$ where each coordinate $g_u$ is chosen independently from uniform (Haar) measure on $G$. For each unordered pair $\{u,v\}$ (with $u \ne v$), with probability $p = \frac{\tilde p}{\sqrt n}$ let $Y_{uv} = g_u g_v^{-1}$, and otherwise let $Y_{uv}$ be drawn uniformly from $G$. Define $Y_{uv} = (Y_{vu})^{-1}$ and $Y_{uu} = 1$ (the identity element of $G$). We reveal the matrix $Y \in G^{n \times n}$.
\end{definition}

This problem has been studied previously by \citet{sin11} for the case where the group $G$ is the cyclic group $\mathbb{Z}/L$. It is important to note that since we only have pairwise measurements, we can only hope to recover the group elements up to a global right-multiplication by some group element. 

\citet{sin11} shows that for $G = \mathbb{Z}/L$ there is a spectral approach that succeeds at detection and recovery above the threshold $\tilde p > 1$. Specifically, the spectral method identifies each group element with a complex $L$th root of unity and takes the top eigenvalue (and eigenvector) of the complex-valued observed matrix $Y$. We expect that an efficient algorithm for detection exists for any finite group above this $\tilde p = 1$ threshold: for instance, if the group has a $\mathbb{Z}/L$ quotient (for any $L$) we can apply the $\ZZ/L$ spectral algorithm. Our main results in the truth-or-Haar model are:

\begin{theorem-nolabel}[see Theorems \ref{thm:toh} and \ref{thm:toh-upper}]\label{thm:combined}
Let $G$ be a finite group of order $L$ and let $\tilde p \ge 0$. If
$$\tilde p < \tilde p^*_L \equiv \sqrt{\frac{2(L-1)\log(L-1)}{L(L-2)}}$$
then $\ToH(G,\tilde p)$ is contiguous to $\ToH(G,0)$. For $L = 2$, $\tilde p^*_2 = 1$ (the limit value of the $0/0$ expression). Moreover if
$$\tilde p > \sqrt{\frac{4 \log L}{L-1}}$$
then a non-efficient algorithm can distinguish the spiked and unspiked models with error probability $o(1)$.
\end{theorem-nolabel}

\noindent Our lower bound matches the spectral threshold $\tilde p = 1$ when $L = 2$, but does not match it for larger values of $L$ (see the table in Section~\ref{sec:toh}). In fact the second part of the theorem shows that there are computationally inefficient procedures to solve detection that work below the spectral threshold $\tilde p = 1$ when $L$ is a large enough constant. 

An important observation is that the truth-or-Haar model is not interesting for infinite groups $G$. In particular if $G$ is infinite, if the measurements agree on a triangle they they must be correct (with probability $1$). Thus in order to meaningfully study synchronization over infinite groups such as $U(1)$ (unit-norm complex numbers) we need the noise to be continuous in nature. This motivates the Gaussian synchronization model in which we add Gaussian noise to the true relative group elements $g_u g_v^{-1}$. In Section~\ref{sec:gsynch} we show how to define this model over any compact group using representation theory. Our model allows for measurements on different `frequencies' (irreducible representations of the group).
Special cases of this model have been studied previously for synchronization over $\ZZ/2$ or $U(1)$ with a single frequency \citep{angular-tightness,dam,sdp-phase,boumal} (previously implicit in \citet{sin11}). To the best of our knowledge, we are the first to introduce this model for multiple frequencies and for all compact groups. The idea of optimizing objective functions that have information on multiple frequencies comes from \citet{nug}.

Similarly to the Gaussian Wigner model, we give a sub-Gaussian method and a conditioning method for contiguity in the Gaussian synchronization model. We give both lower and upper bounds for finite groups in Theorem~\ref{thm:synch-finite} and Theorem~\ref{thm:synch-upper} in the Gaussian synchronization model. In Section~\ref{sec:subg-ex} we use the sub-Gaussian method to show contiguity results for a number of examples such as $U(1)$ and variants with multiple frequencies, which show statistical limitations to how much information can be synthesized across different frequencies. Finally, we are once again able to see that there are computationally inefficient procedures that can solve the detection problem even when PCA fails.

\subsection{The truth-or-Haar model}\label{sec:tohlower}
\label{sec:toh}

In this subsection, we establish contiguity results in the truth-or-Haar model. Let $p = \frac{\tilde p}{\sqrt n}$. Let $P_n$ be the `spiked' model $\ToH_n(\tilde p, G)$ and let $Q_n = \ToH_n(0, G)$ be the `unspiked' model in which the observations are completely random. We give an upper bound on the second moment:
\begin{align*}
\dd[P_n]{Q_n} &= \Ex_g \prod_{u < v} \frac{p\, \one[Y_{uv} = g_u g_v^{-1}] + (1-p)/L}{1/L}, \\
\Ex_{Q_n}\left[\left(\dd[P_n]{Q_n}\right)^2\right] &= \Ex_{g,g'} \prod_{u < v} \Ex_{Y_{uv}\sim Q_n} (p L\, \one[Y_{uv} = g_u g_v^{-1}] + 1-p)(p L\, \one[Y_{uv} = g_u' (g_v')^{-1}] + 1-p) \\
&= \Ex_{g,g'} \prod_{u<v} \Ex_{Y_{uv}\sim Q_n} ( p^2 L^2\, \one[g_u g_v^{-1} = Y_{uv} = g_u g_v^{-1}] + p (1-p) L\, \one[Y_{uv} = g_u g_v^{-1}] \\
&\qquad\qquad + p (1-p) L\, \one[Y_{uv} = g_u' (g_v')^{-1}] + (1-p)^2 )
\end{align*}
\begin{align*}
&= \Ex_{g,g'} \prod_{u<v} (1 + p^2(L \one[g_u g_v^{-1} = g_u' (g_v')^{-1}] - 1)) \\
&= \Ex_{g,g'} \prod_{u<v} (1 + p^2(L \one[g_u^{-1} g_u' = g_v^{-1} g_v'] - 1)) \\
&\leq \Ex_{g,g'} \prod_{u<v} \exp \left[p^2(L \one[g_u^{-1} g_u' = g_v^{-1} g_v'] - 1)\right] \\
&\le \Ex_{g,g'} \prod_{u,v} \exp \left[\frac{p^2}{2}(L \one[g_u^{-1} g_u' = g_v^{-1} g_v'] - 1)\right] \\
&= e^{-n^2 p^2 /2} \,\Ex_{g,g'} \exp \left[\frac{p^2 L}{2} \sum_{u,v} \one[g_u^{-1} g_u' = g_v^{-1} g_v']\right].
\end{align*}

As in Section~\ref{sec:cond-method} we apply the conditioning method of \citet{bmnn}. We can write the above as $\EE\exp(Y^\top A Y)$ where $Y = \frac{N-n \bar{\alpha}}{\sqrt n} \in \mathbb{R}^{L^2}$, $N_{ab} = |\{u \,|\, g_u = a, g'_u = b\}|$, $\bar{\alpha} = \frac{1}{L^2}\one_{L^2}$, $p = \frac{\tilde p}{\sqrt n}$, and $A$ is the $L^2 \times L^2$ matrix $A_{ab,a'b'} = \frac{\tilde p^2 L}{2} \one\{a^{-1}b = a'^{-1}b'\}$. By Proposition~5 in \citet{bmnn} (Proposition~\ref{prop:nn} in this paper) we get contiguity provided
$$\sup_\alpha \frac{(\alpha-\bar\alpha)^\top A (\alpha-\bar\alpha)}{D(\alpha,\bar\alpha)} < 1$$
where $D$ is the KL divergence and $\alpha$ ranges over (vectorized) nonnegative $L \times L$ matrices with row- and column-sums equal to $\frac{1}{L}$. Rewrite the numerator:
\begin{align*}
(\alpha - \bar\alpha)^\top A (\alpha - \bar\alpha) &= \alpha^\top A \alpha - 2 \alpha^\top A \bar\alpha + \bar\alpha^\top A \bar\alpha \\
&= \frac{\tilde p^2 L}{2} \left(\sum_{aba'b'} \alpha_{ab}\alpha_{a'b'} \one\{a^{-1}b = a'^{-1}b'\} - \frac{2}{L} + \frac{1}{L}\right) \\
&= \frac{\tilde p^2 L}{2} \left(\sum_{h \in G} \alpha_h^2 - \frac{1}{L}\right)
\end{align*}
where $\alpha_h = \sum_{(a,b) \in S_h} \alpha_{ab}$ and $S_h = \{(a,b)\,|\, a^{-1}b = h\}$.

In Appendix~\ref{app:matrix-opt} we prove the following result which provides the solution to the optimization problem above.

\begin{proposition}
\label{prop:matrix-opt}
For $L \ge 2$,
$$\sup_\alpha \frac{L}{2} \frac{\left(\sum_{h \in G} \alpha_h^2 - \frac{1}{L}\right)}{D(\alpha,\bar\alpha)} = \frac{L(L-2)}{2(L-1)\log(L-1)}$$
where $\alpha$ ranges over (vectorized) nonnegative $L \times L$ matrices with row- and column-sums equal to $\frac{1}{L}$. When $L = 2$, the right-hand side is taken to equal 1 (the limit value of the $0/0$ expression).
\end{proposition}

This immediately implies the following.
\begin{theorem}\label{thm:toh}
Let $G$ be a finite group of order $L$ and let $\tilde p \ge 0$. If
$$\tilde p < \tilde p^*_L \equiv \sqrt{\frac{2(L-1)\log(L-1)}{L(L-2)}}$$
then $\ToH(G,\tilde p)$ is contiguous to $\ToH(G,0)$. For $L = 2$, $\tilde p^*_2 = 1$ (the limit value of the $0/0$ expression).
\end{theorem}

We provide some numerical values for the critical value $\tilde p^*$.
\begin{center}
  \begin{tabular}{ | c || c | c | c | c | c | c | c | }
    \hline
    $L$ & 2 & 3 & 4 & 5 & 6 & 10 & 100\\ \hline
    $\tilde p^*$ & 1 & 0.961 & 0.908 & 0.860 & 0.819 & 0.703 & 0.305 \\ \hline
  \end{tabular}
\end{center}
Note that this lower bound matches the spectral threshold $\tilde p = 1$ when $L = 2$, but does not match it for $L \ge 3$. In Section~\ref{sec:gsynchup1} we give an upper bound for the truth-or-Haar model using a non-efficient algorithm; in particular we show that for sufficiently large $L$ it is possible to beat the spectral threshold.

\subsection{The Gaussian synchronization model}
\label{sec:gsynch}

We now turn to our Gaussian noise model for synchronization. In order to have a sensible notion of adding Gaussian noise to a group element, we need to introduce some representation theory. We will assume the reader is familiar with the basics of representation theory. See e.g.\ \citet{rep-theory-compact} for an introduction.

Since we will be discussing representations of quaternionic type, we need to recall basic facts about quaternions. (Quaternions and quaternionic-type representations can be skipped on a first reading.) Quaternions take the form $q = a + bi + cj + dk$ where $a,b,c,d \in \mathbb{R}$ and (non-commutative) multiplication follows the rules $i^2 = j^2 = k^2 = ijk = -1$. Like complex numbers, quaternions support the operations norm $|q| = \sqrt{a^2 + b^2 + c^2 + d^2}$, real part $\mathfrak{Re}(q) = a$, and conjugate $\bar q = a - bi - cj - dk$ satisfying $\bar{q_1 q_2} = \bar{q_2}\,\bar{q_1}$ and $q \bar{q} = \bar{q} q = |q|^2$. These allow for the natural notions of unitarity and conjugate transpose $A^*$ for quaternion-valued matrices $A$. The algebra of quaternions is denoted by $\mathbb{H}$.

Let $G$ be a compact group. The irreducible representations of $G$ over $\mathbb{C}$ are finite dimensional. Every irreducible representation of $G$ over $\mathbb{C}$ has one of three types: real, complex, or quaternionic. Representations of real type can be defined over the reals (i.e.\ each group element is assigned a matrix with real-valued entries). Representations of complex type are (unlike the other types) not isomorphic to their complex conjugate representation $\bar\rho$. Representations of quaternionic type can be assumed to take the following form: each $2 \times 2$ block of complex numbers encodes a quaternion value using the correspondence
$$a + bi + cj + dk \quad\leftrightarrow\quad \left(\begin{array}{cc} a+bi & c+di \\ -c+di & a-bi \end{array} \right).$$
Alternatively, we can think of quaternionic-type representations as being defined over the quaternions (i.e.\ each group element is assigned a quaternion-valued matrix) with dimension half as large. We will assume that our irreducible representations (over $\mathbb{C}$) are defined over $\mathbb{R}$, $\mathbb{C}$, or $\mathbb{H}$, depending on whether their type is real, complex, or quaternionic (respectively). Representations of complex type come in conjugate pairs. Without loss of generality, all these representations can be taken to be unitary.

Let $d_\rho$ be the dimension of representation $\rho$. For quaternionic-type representations we let $d_\rho$ be the quaternionic dimension, which is half the complex dimension. (For real-type representations, the real and complex dimensions are the same.) In defining our Gaussian model we need to fix a finite list of representations (`frequencies') to work with.

\begin{definition}
Let $G$ be a compact group. A \emph{list of frequencies} $\Psi$ is a finite set of non-isomorphic irreducible (over $\mathbb{C}$) representations of $G$. We do not allow the trivial representation to be included in this list. For representations of complex type, we do not allow $\rho$ and its conjugate $\bar\rho$ to both appear in the list.
\end{definition}

We need to introduce Gaussian noise of various types. The type of noise used will correspond to the type of the representation in question.

\begin{definition}
A \emph{standard Gaussian} of real, complex, or quaternionic type is defined to be
\begin{itemize}
\item for real type, $\cN(0,1)$
\item for complex type: $\cN(0,1/2) + \cN(0,1/2) \,i$
\item for quaternionic type: $\cN(0,1/4) + \cN(0,1/4)\,i + \cN(0,1/4)\,j + \cN(0,1/4)\,k$
\end{itemize}
where each component is independent.
\end{definition}
\noindent Note that the normalization ensures that the expected squared norm is 1.

\begin{definition}
Let a \emph{GOE}, \emph{GUE}, or \emph{GSE} (respectively) matrix be a random Hermitian matrix where the off-diagonals are standard Gaussians of real, complex, or quaternionic type (respectively), and the diagonal entries are real Gaussians $\cN(0,2/\beta)$ where $\beta = 1,2,4$ (respectively) depending on the type. All entries are independent except for the Hermitian constraint.
\end{definition}
\noindent These matrices are the well-known Gaussian orthogonal (resp.\ unitary, symplectic)\ ensembles from random matrix theory.

We can now formally state the Gaussian synchronization model over any compact group.

\begin{definition}
Let $G$ be a compact group and let $\Psi$ be a list of frequencies. For each $\rho \in \Psi$, let $\lambda_\rho \ge 0$. The \emph{Gaussian synchronization model} $\GSynch(\{\lambda_\rho\}, G, \Psi)$ is defined as follows. To sample from the $n$th distribution, draw a vector $g \in G^n$ by sampling each coordinate independently from Haar (uniform) measure on $G$. Let $X_\rho$ be the $n d_\rho \times d_\rho$ matrix formed by stacking the matrices $\rho(g_u)$ for all $u$. For each frequency $\rho \in \Psi$, reveal the $n d_\rho \times n d_\rho$ matrix
$$Y_\rho = \frac{\lambda_\rho}{n} X_\rho X_\rho^* + \frac{1}{\sqrt{nd_\rho}} W_\rho$$
where $W_\rho$ is an $n d_\rho \times n d_\rho$ Hermitian Gaussian matrix (GOE, GUE, or GSE depending on whether $\rho$ has real, complex, or quaternionic type, respectively). If we write a scalar $\lambda$ in place of $\{\lambda_\rho\}$ we mean that $\lambda_\rho = \lambda$ for all $\rho$.
\end{definition}

Special cases of this model have been studied previously for synchronization over $\ZZ/2$ or $U(1)$ with a single frequency \citep{angular-tightness,dam,sdp-phase,boumal} (previously implicit in \citet{sin11}). (Note that the $\ZZ/2$ case is simply the spiked Gaussian Wigner model with the Rademcacher prior.) To the best of our knowledge, we are the first to introduce this model for multiple frequencies and for all compact groups. The idea of optimizing objective functions that have information on multiple frequencies comes from \citet{nug}.

When $\lambda_\rho > 1$ for at least one $\rho$, we can use PCA (top eigenvalue) to reliably distinguish between $P_n = \GSynch_n(\{\lambda_\rho\},G,\Psi)$ and $Q_n = \GSynch_n(0,G,\Psi)$. If given $K$ frequencies, all with the same $\lambda$, it may appear that one should be able to combine the frequencies in order to achieve the threshold $\lambda > 1/\sqrt{K}$; after all, this would be possible if given $K$ independent observations of a single frequency. However, our contiguity results will show that $\lambda > 1/\sqrt{K}$ is not sufficient. In fact, we conjecture that $\lambda > 1$ is required for any efficient algorithm to succeed at detection, although there are inefficient algorithms that succeed below this.

\subsection{Second moment computation}\label{sec:gsynchsecond}

Let $P_n$ be $\GSynch(\{\lambda_\rho\},G,\Psi)$ and let $Q_n$ be $\GSynch(0,G,\Psi)$. Let $\beta_\rho = 1,2,4$ for real-, complex-, or quaternionic-type (respectively). We will use the standard Hermitian inner product for matrices: $\langle A,B \rangle = \mathrm{Tr}(AB^*)$ where $B^*$ denotes the conjugate transpose of $B$.

\begin{align*}
\dd[P_n]{Q_n} &= \Ex_X \prod_{\rho \in \Psi} \frac{\exp\left(-\frac{\beta_\rho n d_\rho}{4} \left\|Y_\rho - \frac{\lambda_\rho}{n}X_\rho X_\rho^*\right\|_F^2\right)}{\exp\left(-\frac{\beta_\rho n d_\rho}{4} \left\|Y\right\|_F^2\right)} \\
&= \Ex_X \prod_\rho \exp\left(\frac{\beta_\rho \lambda_\rho d_\rho}{2} \,\mathfrak{Re}\left\langle Y_\rho, X_\rho X_\rho^* \right\rangle - \frac{\beta_\rho \lambda_\rho^2 d_\rho}{4n} \left\|X_\rho X_\rho^*\right\|_F^2 \right).
\end{align*}
\begin{align*}
\Ex_{Q_n} & \left(\dd[P_n]{Q_n} \right)^2 \\
&= \Ex_{Y \sim Q_n} \Ex_{X,X'} \prod_\rho \exp\left(\frac{\beta_\rho \lambda_\rho d_\rho}{2} \,\mathfrak{Re}\left\langle Y_\rho, X_\rho X_\rho^* + X'_\rho {X'}_\rho^* \right\rangle - \frac{\beta_\rho \lambda_\rho^2 d_\rho}{4n} \left\|X_\rho X_\rho^*\right\|_F^2 - \frac{\beta_\rho \lambda_\rho^2 d_\rho}{4n} \left\|{X'}_\rho {X'}_\rho^*\right\|_F^2 \right) \\
&= \Ex_{X,X'} \prod_\rho \Ex_{Y_\rho} \exp\left(\frac{\beta_\rho \lambda_\rho d_\rho}{2} \,\mathfrak{Re}\left\langle Y_\rho, X_\rho X_\rho^* + X'_\rho {X'}_\rho^* \right\rangle - \frac{\beta_\rho \lambda_\rho^2 d_\rho}{4n} \left\|X_\rho X_\rho^*\right\|_F^2 - \frac{\beta_\rho \lambda_\rho^2 d_\rho}{4n} \left\|{X'}_\rho {X'}_\rho^*\right\|_F^2 \right). \\
\intertext{Use the Gaussian moment-generating function to eliminate $Y$: if $z$ is a scalar (from $\RR$, $\CC$, or $\mathbb{H}$) and $y$ is a standard Gaussian of the same type, then $\EE \exp(\mathfrak{Re}(y z)) = \exp(\frac{1}{2\beta} |z|^2)$. Recall that $Y_\rho$ (drawn from $Q_n$) is Hermitian with each off-diagonal entry $\frac{1}{\sqrt{n d_\rho}}$ times a standard Gaussian (of the appropriate type), and each diagonal entry real Gaussian $\cN(0,\beta/2)$. Continuing from above,}
&= \Ex_{X,X'} \prod_\rho \exp\left(\frac{1}{2\beta_\rho} \frac{1}{n d_\rho} \beta_\rho^2 \lambda_\rho^2 d_\rho^2 \,\frac{1}{2} \left\|X_\rho X_\rho^* + {X'}_\rho {X'}_\rho^* \right\|_F^2 - \frac{\beta_\rho \lambda_\rho^2 d_\rho}{4n} \left\|X_\rho X_\rho^*\right\|_F^2 - \frac{\beta_\rho \lambda_\rho^2 d_\rho}{4n} \left\|{X'}_\rho {X'}_\rho^*\right\|_F^2\right) \\
&= \Ex_{X,X'} \prod_\rho \exp\left(\frac{\beta_\rho \lambda_\rho^2 d_\rho}{4n} \left\|X_\rho X_\rho^* + {X'}_\rho {X'}_\rho^* \right\|_F^2 - \frac{\beta_\rho \lambda_\rho^2 d_\rho}{4n} \left\|X_\rho X_\rho^*\right\|_F^2 - \frac{\beta_\rho \lambda_\rho^2 d_\rho}{4n} \left\|{X'}_\rho {X'}_\rho^*\right\|_F^2 \right) \\
&= \Ex_{X,X'} \prod_\rho \exp\left(\frac{\beta_\rho \lambda_\rho^2 d_\rho}{4n} \,2\, \mathfrak{Re} \left\langle X_\rho X_\rho^*, {X'}_\rho {X'}_\rho^* \right\rangle \right) \\
&= \Ex_{X,X'} \prod_\rho \exp\left(\frac{\beta_\rho \lambda_\rho^2 d_\rho}{2n} \left\|X_\rho^*{X'}_\rho\right\|_F^2 \right). \\
\end{align*}

\subsection{The sub-Gaussian method}\label{sec:gsynchsub}

We will aim to show contiguity at a point where all $\lambda$'s are equal: $\lambda_\rho = \lambda$ for all $\rho$. (Note however that if we show contiguity at some $\lambda$ and we then decrease some of the individual $\lambda_\rho$'s, we still have contiguity because the second moment above will only decrease.) Ideally we want contiguity for all $\lambda < 1$, matching the spectral threshold.

For each $u \in [n]$ let $Z_u$ be a vector in $\RR^D$ where $D = \sum_{\rho \in \Psi} \beta_\rho d_\rho^2$, formed as follows. First draw $h_u$ independently from Haar measure on $G$. For each $\rho$, vectorize the matrix $\sqrt{\beta_\rho d_\rho}\,\rho(h_u)$ into a real-valued vector of length $\beta_\rho d_\rho^2$ by separating the $\beta_\rho$ components of each of the $d_\rho^2$ entries. Finally, concatenate all these vectors together to form $Z_u$. Let $Z^{(G,\Psi)}$ denote the distribution that each $Z_u$ follows.

We can rewrite the second moment as
$$\Ex_{Q_n} \left(\dd[P_n]{Q_n}\right)^2 = \Ex_Z \exp\left(\frac{\lambda^2}{2n} \left\|\sum_u Z_u\right\|^2 \right).$$
(Here $h_u = g_u^{-1} g'_u$.)

We will use the following definition of sub-Gaussian for vector-valued random variables.
\begin{definition}\label{def:subg-mult}
We say $z \in \mathbb{R}^m$ is \emph{sub-Gaussian with covariance proxy} $\sigma^2 I$ if $\EE[z] = 0$ and for all vectors $v \in \mathbb{R}^m$,
$$\EE\exp\left(\left\langle z,v \right\rangle\right) \le \exp\left(\frac{1}{2} \,\sigma^2 \|v\|^2 \right).$$
\end{definition}

More generally we can allow for a covariance proxy $\Sigma$ that is not a multiple of the identity by replacing $\sigma^2 \|v\|^2$ by $v^\top \Sigma v$, but we will not need this here. Standard methods in the theory of large deviations give the following multivariate sub-Gaussian tail bound.
\begin{lemma}\label{lem:subg-mult-tail}
Suppose $z \in \mathbb{R}^m$ is sub-Gaussian with covariance proxy $\sigma^2 I$. Let $\eps > 0$. For all $a \ge 0$,
$$\prob{\|z\|^2 \ge a} \le C \exp\left(-\frac{a(1-\eps)}{2 \sigma^2}\right)$$
where $C = C(\eps, m)$ is a constant depending only on $\eps$ and the dimension $m$.
\end{lemma}
\begin{proof}
Let $v_1, \ldots, v_C \in \mathbb{R}^m$ be a collection of unit vectors such that for every unit vector $\hat z \in \mathbb{R}^m$, there exists $i$ satisfying $\langle \hat z, v_i \rangle \ge \sqrt{1-\eps}$. If $\|z\|^2 \ge a$ then there must exist $i$ such that $\langle z, v_i \rangle \ge \sqrt{a(1-\eps)}$. For a fixed $i$ and for any $t > 0$ we have
\begin{align*}
\prob{\langle z,v_i \rangle \ge \sqrt{a (1-\eps)}} &= \prob{\exp(t \langle z, v_i \rangle) \ge \exp(t\sqrt{a(1-\eps)})} \\
&\le \EE[\exp(\langle z,t v_i \rangle)] \exp(-t \sqrt{a(1-\eps)}) \\
&\le \exp\left(\frac{1}{2} \sigma^2 t^2\right) \exp(-t \sqrt{a(1-\eps)})
\intertext{setting $t = \sqrt{a(1-\eps)}/\sigma^2$,}
&= \exp\left(-\frac{a(1-\eps)}{2\sigma^2}\right).
\end{align*}
The result now follows by a union bound over all $i$.
\end{proof}

The following theorem gives a sufficient condition for contiguity in terms of the sub-Gaussian property.
\begin{theorem}[sub-Gaussian method]
\label{thm:synch-subg}
Let $G$ be a compact group and let $\Psi$ be a list of frequencies. Suppose $Z^{(G,\Psi)}$ (defined above) is sub-Gaussian with covariance proxy $\sigma^2 I$. If $\lambda < 1/\sigma$ then $\GSynch(\lambda,G,\Psi)$ is contiguous to $\GSynch(0,G,\Psi)$.
\end{theorem}
\begin{proof}
Note that $\sum_u Z_u$ is sub-Gaussian with covariance proxy $n \sigma^2 I$. From above we have
\begin{align*}
\Ex_{Q_n} \left(\dd[P_n]{Q_n}\right)^2 &= \EE \exp\left(\frac{\lambda^2}{2n} \left\|\sum_u Z_u\right\|^2 \right) \\
&= \int_0^\infty \problr{\exp\left(\frac{\lambda^2}{2n} \left\|\sum_u Z_u\right\|^2 \right) \ge M}\, dM \\
&= \int_0^\infty \problr{\|Z\|^2 \ge \frac{2n \log M}{\lambda^2}}\, dM \\
&\le 1 + \int_1^\infty \problr{\|Z\|^2 \ge \frac{2n \log M}{\lambda^2}}\, dM \\
&\le 1 + \int_1^\infty C \exp\left(-\frac{(1-\eps)}{2n\sigma^2} \frac{2n \log M}{\lambda^2}\right) \, dM \\
&= 1 + \int_1^\infty C \exp\left(-\frac{(1-\eps)\log M}{\sigma^2 \lambda^2} \right) \, dM \\
&= 1 + \int_1^\infty C M^{-(1-\eps)/(\sigma^2 \lambda^2)} \, dM \\
\end{align*}
which is finite provided that $(1-\eps)/(\sigma^2 \lambda^2) > 1$. The second inequality uses Lemma~\ref{lem:subg-mult-tail}. Since $\eps$ was arbitrary, this completes the proof.
\end{proof}

We remark that if we want to find the limit value of the second moment (e.g.\ for the hypothesis testing bounds of Proposition~\ref{prop:hyptest}), we can apply the argument of Theorem~\ref{thm:subg} based on the central limit theorem (except now using the multivariate central limit theorem).

Note that $\EE[Z^{(G,\Psi)}] = 0$ (which is a requirement for sub-Gaussianity) is automatically satisfied; this follows from the Peter--Weyl theorem on orthogonality of matrix entries, which we will discuss in more detail in Section~\ref{sec:gsynchcond}.

\subsection{Applications of the sub-Gaussian method}
\label{sec:subg-ex}

In this section we use Theorem~\ref{thm:synch-subg} to prove contiguity for some specific synchronization problems.

First we consider $U(1)$ with a single frequency. This is a complex-valued spiked Gaussian Wigner matrix, where the spike is complex-valued with each entry having unit norm. \cite{sdp-phase} predicted that the statistical threshold for this problem should be the spectral threshold $\lambda = 1$; we now confirm this.
\begin{theorem}[$U(1)$ with one frequency]\label{thm:u1-1}
Consider the group $U(1)$ of unit-norm complex numbers under multiplication. Identify each element $e^{i\theta}$ of $U(1)$ with its angle $\theta$. Let $\Psi_1$ be the list containing the single frequency $\rho: \theta \mapsto e^{i\theta}$. For any $\lambda < 1$, $\GSynch(\lambda,U(1),\Psi_1)$ is contiguous to $\GSynch(0,U(1),\Psi_1)$.
\end{theorem}
\begin{proof}
We have $Z^{(U(1),\Psi_1)} = \sqrt{2}\,(\cos\theta, \sin\theta)$ where $\theta$ is drawn uniformly from $[0,2 \pi]$. Towards showing sub-Gaussianity we have, for any $v \in \mathbb{R}^2$,
$$\EE \exp\left(\left\langle Z^{(U(1),\Psi_1)},v \right\rangle\right) = \EE_\theta \exp\left(\sqrt{2}\, v_1 \cos \theta + \sqrt{2}\, v_2 \sin \theta\right) = \EE_\theta \exp\left( \sqrt{2} \,\|v\| \cos\theta \right).$$
Letting $w = \|v\|$, it is sufficient to show for all $w \ge 0$,
$$\EE_\theta \exp\left(\sqrt{2} \, w \cos \theta\right) \le \exp\left(\frac{1}{2} \, w^2\right).$$
This can be verified numerically but we also provide a rigorous proof. Using the Taylor expansion of $\exp$ and the identity $$\EE_\theta\left[\cos^k \theta\right] = \left\{\begin{array}{cc}\frac{(k-1)!!}{k!!} & k \text{ even} \\ 0 & k \text{ odd} \end{array} \right.$$
we have
\begin{align*}
\EE_\theta \exp\left(\sqrt{2} \, w \cos \theta\right)
&= \EE_\theta \sum_{k \ge 0} \frac{2^{k/2} w^k \cos^k \theta}{k!} 
= \sum_{k \ge 0} \frac{2^{k} w^{2k} \EE_\theta \cos^{2k} \theta}{(2k)!} \\
&= \sum_{k \ge 0} \frac{2^{k} w^{2k} (2k-1)!!}{(2k)!(2k)!!} 
= \sum_{k \ge 0} \frac{2^{k} w^{2k}}{(2k)!!(2k)!!} \\
&\le \sum_{k \ge 0} \frac{w^{2k}}{(2k)!!}
= \sum_{k \ge 0} \frac{w^{2k}}{2^k k!}
= \exp\left(\frac{1}{2} w^2\right).
\end{align*}
The exchange of expectation and infinite sum is justified by the Fubini--Tonelli theorem, provided we can show absolute convergence:
$$\sum_{k \ge 0} \EE_\theta \left|\frac{2^{k/2} w^k \cos^k \theta}{k!}\right| \le \sum_{k \ge 0} \left|\frac{2^{k/2} w^k}{k!}\right|$$
which converges by the ratio test.
\end{proof}

We now add a second frequency.
\begin{example}[$U(1)$ with two frequencies]\label{ex:u1-2}
Consider again $U(1)$ but now let $\Psi_2$ be the list of two frequencies: $\rho_1: \theta \mapsto e^{i\theta}$ and $\rho_2: \theta \mapsto e^{2 i \theta}$. For any $\lambda < \lambda^* \approx 0.9371$ (numerically computed), $\GSynch(\lambda,U(1),\Psi_2)$ is contiguous to $\GSynch(0,U(1),\Psi_2)$.
\end{example}
\noindent (We use ``example'' rather than ``theorem'' to indicate results that rely on numerical computations.) Although we are unable to show that the spectral threshold is optimal, note that this rules out the possibility that the threshold for two frequencies drops to $1/\sqrt{2}$ (which is what we would have if one could perfectly synthesize the frequencies). We expect that the true statistical threshold for this problem is $\lambda = 1$ and that our results are not tight here. We now move on to the case of $\ZZ/L$.
\begin{proof}[Details]
We have $Z^{(U(1),\Psi_2)} = \sqrt{2}\,(\cos\theta,\sin\theta,\cos(2\theta),\sin(2\theta))$. Our threshold is $\lambda^* = 1/\sigma^*$ where
$$(\sigma^*)^2 = \sup_v \frac{2}{\|v\|} \log \EE\left(\langle Z^{(U(1),\Psi_2)},v \rangle \right)
= \sup_v \frac{2}{\|v\|} \log \EE_\theta\left(\sqrt{2}(v_1 \cos \theta + v_2 \sin \theta + v_3 \cos(2\theta) + v_4 \sin(2\theta))\right).$$
By the change of variables $\theta \mapsto \theta - \theta_0$ (for some $\theta_0$) we can rotate $(v_1,v_2)$ arbitrarily, and so we can take $v_2 = 0$ and $v_1 \ge 0$ without loss of generality. By grid search over $v_1,v_3,v_4$, we see numerically that the maximizer is $v^* = (0.720,0,0.559,0)$ which yields contiguity for all $\lambda < \lambda^* \approx 0.937$.
\end{proof}

\begin{example}[$\ZZ/L$ with one frequency]\label{ex:ZL-1}
Now consider $\mathbb{Z}/L = \{0,1,\ldots,L-1\} \; (\mathrm{mod}\;L)$ with $L \ge 2$ and $\Psi_1$ the list of one frequency: $j \mapsto \exp(2 \pi i j / L)$. For $L = 3$, we have contiguity $\GSynch(\lambda,\ZZ/L,\Psi_1)\contig\GSynch(0,\ZZ/L,\Psi_1)$ for all $\lambda < \lambda^*_3 \approx 0.961$. For $L = 2$ and all $L \ge 4$, we have contiguity for all $\lambda < 1$.
\end{example}
\begin{proof}[Details]
This is shown numerically in a manner similar to the examples above. Of course we cannot test this for all values of $L$, but we conjecture that the $\lambda^* = 1$ trend continues indefinitely.
\end{proof}
\noindent We have that the spectral threshold is optimal for all $L$ except 3. It is surprising that $L = 3$ is an exception here, be we expect that this is a weakness of our techniques and that the true threshold for $L = 3$ is also $\lambda = 1$.

Finally we give a coarse but general result for any group with any number of frequencies.
\begin{theorem}[any group, any frequencies]\label{thm:any-hoeff}
Let $G$ be any group and let $\Psi$ be any list of frequencies, with $D = \sum_{\rho \in \Psi} \beta_\rho d_\rho^2$. If $\lambda < 1/\sqrt{D}$ then $\GSynch(\lambda,G,\Psi)$ is contiguous to $\GSynch(0,G,\Psi)$.
\end{theorem}
\begin{proof}
Since our representations $\rho$ are unitary, we have $\|\rho(g)\|_F^2 = d_\rho$ for any $g \in G$, and so $\|Z^{(G,\Psi)}\|^2 = D$. This means for any vector $v$ we have $|\langle Z^{(G,\Psi)},v\rangle| \le \|Z^{(G,\Psi)}\| \|v\| = \sqrt{D} \|v\|$. By Hoeffding's Lemma (see Section~\ref{sec:z2}) this implies the sub-Gaussian condition $\EE \exp(\langle Z^{(G,\Psi)},v\rangle) \le \exp(\frac{1}{2} D \|v\|^2)$.
\end{proof}

\subsection{The conditioning method for finite groups}\label{sec:gsynchcond}

Here we give an alternative method to show contiguity for finite groups, similar to the conditioning method of Section~\ref{sec:cond-method} based on \citet{bmnn}. Let $G$ be a finite group with $|G| = L$. Again take all the $\lambda$'s to be equal: $\lambda_\rho = \lambda$ for all $\rho$. For $a,b \in G$, let $N_{ab} = |\{u \,|\, g_u = a, g'_u = b\}|$. Rewrite the second moment in terms of $N_{ab}$:

\begin{align*}
\EE_{X,X'} \prod_\rho \exp\left(\frac{\beta_\rho \lambda_\rho^2 d_\rho}{2n} \left\|X_\rho^*{X'}_\rho\right\|_F^2 \right)
&= \EE_{X,X'} \exp\left(\frac{\lambda^2}{2n} \sum_\rho \beta_\rho d_\rho \sum_c (X_\rho^* {X'}_\rho)_c^2 \right) \\
\intertext{where $c$ ranges over all (real-valued) coordinates of entries of $\rho(g)$ (e.g.\ imaginary part of top right entry)}
&= \EE_{g,g'} \exp\left(\frac{\lambda^2}{2n} \sum_\rho \beta_\rho d_\rho \sum_c \left(\sum_u \rho(g_u^{-1} g'_u)_c\right)^2 \right) \\
&= \EE_N \exp\left(\frac{\lambda^2}{2n} \sum_\rho \beta_\rho d_\rho \sum_c \left(\sum_{ab} N_{ab}\, \rho(a^{-1}b)_c \right)^2 \right) \\
&= \EE_N \exp\left(\frac{1}{n} N^\top A N \right) \\
&= \EE_N \exp\left(Y^\top A Y \right)
\end{align*}
where $Y = \frac{\vec N - n \bar\alpha}{\sqrt n}$, $\bar\alpha = \frac{1}{L^2} \one_{L^2}$, and $A$ is the $L^2 \times L^2$ matrix
$$A_{ab,a'b'} = \frac{\lambda^2}{2} \sum_\rho \beta_\rho d_\rho \sum_c \rho(a^{-1}b)_c \,\rho({a'}^{-1}b')_c.$$
To justify the last step, note that $\bar\alpha$ is in the kernel of $A$ because all row- and column-sums of $A$ are zero. This follows from the Peter--Weyl theorem on orthogonality of matrix coefficients, which we will discuss in more detail shortly. By Proposition~5 in \citet{bmnn} (Proposition~\ref{prop:nn} in this paper) we have contiguity provided that
$$\sup_\alpha \frac{\alpha^\top A \alpha}{D(\alpha,\bar\alpha)} < 1$$
where $\alpha$ ranges over (vectorized) $L \times L$ matrices with all row- and column-sums equal to $\frac{1}{L}$.

\begin{theorem}[conditioning method]
\label{thm:synch-cond}
Let $G$ be a finite group of order $L$ and let $\Psi$ be a list of frequencies. Let $\tilde A$ be the $L^2 \times L^2$ matrix $\tilde A_{ab,a'b'} = \frac{1}{2} \sum_{\rho \in \Psi} \beta_\rho d_\rho \sum_c \rho(a^{-1}b)_c \rho(a'^{-1}b')_c$ where $a,b,a',b' \in G$ and $c$ ranges over (real) coordinates of matrix entries. Let $D(u,v)$ denote the KL divergence between two vectors: $D(u,v) = \sum_i u_i \log(u_i/v_i)$. If
$$\lambda < \left[\sup_{\alpha} \frac{\alpha^\top \tilde A \alpha}{D(\alpha,\bar\alpha)}\right]^{-1/2}$$
then $\GSynch(\lambda,G,\Psi)$ is contiguous to $\GSynch(0,G,\Psi)$. Here $\alpha$ ranges over (vectorized) $L \times L$ matrices with all row- and column-sums equal to $\frac{1}{L}$.
\end{theorem}

A finite group has only a finite number of irreducible representations (over $\mathbb{C}$), so let us now specialize to the case where our list $\Psi$ contains all of them (excluding the trivial representation, and only taking one representation per conjugate pair). Expand the numerator:
$$\alpha^\top A \alpha = \frac{\lambda^2}{2} \sum_\rho \beta_\rho d_\rho \sum_c \left(\sum_{ab} \alpha_{ab}\, \rho(a^{-1}b)_c \right)^2 = \frac{\lambda^2}{2} \sum_\rho \beta_\rho d_\rho \sum_c \left(\sum_h \alpha_h \, \rho(h)_c \right)^2$$
where $\alpha_h = \sum_{(a,b) \in S_h} \alpha_{ab}$ and $S_h = \{(a,b)\,|\, a^{-1}b = h\}$. We now appeal to the Peter--Weyl theorem on the orthogonality of matrix coefficients: the basis functions $\chi_{\rho ij}(h) = \sqrt{d_\rho^\mathbb{C}} \rho(h)_{ij}$ (for all irreducible $\rho$ over $\mathbb{C}$ and matrix entries $i,j$) form an orthonormal basis for $\mathbb{C}^G$ under the Hermitian inner product $\langle f_1, f_2 \rangle \equiv \frac{1}{L} \sum_{h \in G} f_1(h) \bar{f_2(h)}$. Here $d_\rho^\mathbb{C}$ is the dimension as a complex representation, which is the same as $d_\rho$ for real- and complex-type but equal to $2 d_\rho$ for quaternionic-type. This means the above can be thought of as projecting the vector $\{\alpha_h\}_{h \in G}$ onto these basis elements and then computing the $\ell^2$ norm of the result. By the basis-invariance of the $\ell^2$ norm, we can rewrite the above as
$$\frac{\lambda^2 L^2}{2} \left[\frac{1}{L} \sum_h \alpha_h^2 - \left(\frac{1}{L} \sum_h \alpha_h \right)^2\right] = \frac{\lambda^2 L}{2} \left[ \sum_h \alpha_h^2 - \frac{1}{L}\right].$$
The second term here corrects for the fact that the trivial representation did not appear in our original expression. Note that the factor of $\beta = 2$ for complex representations corrects for the fact that we were only using one representation per conjugate pair. The factor of $\beta = 4$ for quaternionic representations corrects for the fact that we were thinking of these representations as being defined over $\mathbb{H}$ rather than $\mathbb{C}$; the corresponding complex representation has dimension twice as large and represents each quaternion value by the following $2 \times 2$ complex matrix:
$$a + bi + cj + dk \quad\leftrightarrow\quad \left(\begin{array}{cc} a+bi & c+di \\ -c+di & a-bi \end{array} \right).$$
(One factor of 2 comes from the fact that $d_\rho^\mathbb{C} = 2d_\rho$ and the other factor of 2 comes from the fact that the squared-Frobenius norm of this $2 \times 2$ matrix is twice the squared-norm of the associated quaternion.)

Note that we now have exactly the same optimization problem that we arrived at for the truth-or-Haar model of Section~\ref{sec:toh} with $\lambda$ in place of $\tilde p$, so we can apply Proposition~\ref{prop:matrix-opt} to immediately obtain the following.
\begin{theorem}\label{thm:synch-finite}
Let $G$ be a finite group of order $L \ge 2$ and let $\Psi_\mathrm{all}$ be the list of {\bf all} frequencies (excluding the trivial one and only taking one from each conjugate pair). If for all $\rho \in \Psi_\mathrm{all}$,
$$\lambda_\rho < \lambda^*_L \equiv \sqrt{\frac{2(L-1)\log(L-1)}{L(L-2)}}$$
then $\GSynch(\{\lambda_\rho\},G,\Psi_\mathrm{all})$ is contiguous to $\GSynch(0,G,\Psi_\mathrm{all})$. For $L = 2$, $\lambda^*_2 = 1$ (the limit value of the $0/0$ expression).
\end{theorem}
Here we have used the monotonicity of the second moment: if we show contiguity when all the $\lambda_\rho$'s are equal to some $\lambda$, and we then decrease some of the individual $\lambda_\rho$'s, we will still have contiguity.


Interestingly, our critical value $\lambda^*_L$ is the same as our critical value $\tilde p^*_L$ from the truth-or-Haar model. As discussed previously, this matches the spectral threshold $\lambda = 1$ only when $L = 2$. However, for small values of $L$, our $\lambda^*_L$ is quite close to 1 (see the table in Section~\ref{sec:toh}).

Also note that when $L = 3$, Theorem~\ref{thm:synch-finite} matches (and proves rigorously) the numerical value $\lambda^* \approx 0.961$ of Example~\ref{ex:ZL-1} (obtained via the sub-Gaussian method). Note that when $L = 3$, $\ZZ/L$ only has one frequency, so these two results apply to the same problem. However, we see that the conditioning method gains no advantage over the sub-Gaussian method in this case. This seems to be true in general for synchronization problems because there are no particularly `bad' values for the spike due to symmetry of the group.

\subsection{Detection in the truth-or-Haar model}\label{sec:gsynchup1}

In this subsection, we show that exhaustive search outperforms spectral methods in the Truth-or-Haar Model when $L$ is large enough. Specifically, we show: 

\begin{theorem}\label{thm:toh-upper}
Let $G$ be a finite group of order $L \ge 2$. If
$$\tilde p > \sqrt{\frac{4 \log L}{L-1}}$$
there is a computationally inefficient algorithm that can distinguish between the spiked and unspiked models.
\end{theorem}

\noindent For small $L$, this theorem is not very interesting because the right-hand side exceeds the spectral threshold of 1. However, for $L \ge 11$, the right-hand side drops below 1, indicating that it is information-theoretically possible to go below the spectral threshold. As $L \to \infty$, this upper bound differs from the lower bound of Theorem~\ref{thm:toh} by a factor of $\sqrt{2}$. (We expect that the upper bound is asymptotically tight here and that the lower bound can be improved by the ``noise conditioning'' method recently introduced by \citet{noise-cond}.)

\begin{proof}
We will use a non-efficient algorithm based on exhaustive search over all candidate solutions $g \in G^n$. Given an observed matrix $Y$ valued in $G$, let $T(g)$ be the number of edges satisfied by $g$, i.e.\ the number of unordered pairs $\{u,v\}$ (with $u \ne v$) such that $Y_{uv} = g_u g_v^{-1}$. The algorithm will distinguish between $P_n = \ToH(G,\tilde p)$ and $Q_n = \ToH(G,0)$ by thresholding $T = \max_{g \in G^n} T(g)$ (at some cutoff to be determined later).

Suppose $Y$ is drawn from $P_n$ and let $g^* \in G^n$ be the true spike. Then $T(g) \sim \Binom(N,p')$ where $N = \binom{n}{2}$ and $p' = p + \frac{1-p}{L}$. By Hoeffding's inequality,
$$P_n(T(g^*) \le Np' - k) \le \exp\left(-\frac{2k^2}{N}\right)$$
which in turn implies $P_n(T(g^*) \le Np' - n \log n) = o(1).$

Now suppose $Y$ is drawn from $Q_n$ and fix any $g \in G^n$. Then $T(g) \sim \Binom(N,1/L)$. By the Chernoff bound,
$$Q_n(T(g) \ge k) \le \exp\left(-N D(k/N \,\middle \| \, 1/L)\right)$$
where $D\left(a \,\middle\|\, b\right) = a \log(a/b) + (1-a) \log((1-a)/(1-b))$. By a union bound over all $L^n$ choices for $g$,
\begin{align*}
Q_n(T \ge Np' - n \log n)
&\le L^n \exp\left(-N D\left(\frac{Np' - n \log n}{N} \,\middle\|\, \frac{1}{L} \right)\right) \\
&= \exp\left(n \log L - N D\left(p + \frac{1-p}{L} - \mathcal{O}\left(\frac{\log n}{n}\right) \,\middle\|\, \frac{1}{L} \right)\right) \\
&= \exp\left(n \log L - N D\left(1/L + \Delta \,\middle\|\, 1/L \right)\right) \\
\intertext{where $\Delta = p (1 - 1/L) - \mathcal{O}\left(\frac{\log n}{n}\right) = \frac{\tilde p(L-1)}{\sqrt{n} L} - o\left(\frac{1}{\sqrt n}\right)$}
&= \exp\left[n \log L - N \left((1/L+\Delta) \log\left(\frac{1/L+\Delta}{1/L}\right) + (1-1/L-\Delta)\log\left(\frac{1-1/L-\Delta}{1-1/L}\right)\right)\right] \\
&= \exp\left[n \log L - N \left((1/L+\Delta) \log\left(1 + L \Delta\right) + (1-1/L-\Delta)\log\left(1 - \frac{L\Delta}{L-1}\right)\right)\right] \\
&= \exp\Bigg[n \log L - N \Big((1/L+\Delta) (L \Delta - \frac{1}{2} L^2 \Delta^2) \\
&\qquad\qquad + \left(\frac{L-1}{L}-\Delta\right) \left(-\frac{L\Delta}{L-1} - \frac{L^2\Delta^2}{2(L-1)^2}\right) + o(1/n)\Big)\Bigg] \\
&= \exp\left[n \log L - N \left(\Delta + L\Delta^2 - \frac{1}{2}L\Delta^2 - \Delta + \frac{L}{L-1}\Delta^2 - \frac{L\Delta^2}{2(L-1)} + o(1/n)\right)\right] \\
&= \exp\left[n \log L - \frac{n^2}{2} \Delta^2 \frac{1}{2}\left(L + \frac{L}{L-1}\right) + o(n)\right] \\
&= \exp\left[n \log L - \frac{n}{4} \tilde p^2 \left(\frac{L-1}{L}\right)^2 \left(\frac{L^2}{L-1}\right) + o(n)\right] \\
&= \exp\left[n \log L - \frac{n}{4} \tilde p^2 (L-1) + o(n)\right] \\
&= o(1)
\end{align*}
provided $\log L < \tilde p^2(L-1)/4$, i.e.\
$$\tilde p > \sqrt{\frac{4 \log L}{L-1}}.$$
Therefore, it is possible to reliably distinguish $P_n$ and $Q_n$ by thresholding $T$ at $Np' - n \log n$.
\end{proof}

\subsection{Detection in the Gaussian synchronization model}\label{sec:gsynchup2}

In this subsection, we analyze the performance of exhaustive search in the Gaussian Synchronization Model. Specifically, we show:

\begin{theorem}\label{thm:synch-upper}
Let $G$ be a finite group of order $L$ and let $\Psi$ be a list of frequencies. If
$$\sum_{\rho \in \Psi} \lambda_\rho^2 \beta_\rho d_\rho^2 > 4 \log L$$
there is a computationally inefficient algorithm that can distinguish between the spiked and unspiked models.
\end{theorem}
\noindent See Corollary~\ref{cor:synch-upper-all} below for a simplification in the case of all frequencies.

Let $P_n = \GSynch_n(\{\lambda_\rho\},G,\Psi)$ and let $Q_n = \GSynch_n(0,G,\Psi)$. By the Neyman--Pearson lemma, the most powerful test statistic for distinguishing $P_n$ from $Q_n$ is the likelihood ratio $\dd[P_n]{Q_n}$. Similarly to \citet{bmvx} we use the following modified likelihood ratio. For $g \in G^n$, let $V_\rho(g)$ be the $n d_\rho \times d_\rho$ matrix formed by stacking the matrices $\rho(g_u)$. Given $Y = \{Y_\rho\}$ drawn from either $P_n$ or $Q_n$, our test is to compute $T = \max_{g \in G^n} T(g)$ where
$$T(g) = \sum_{\rho \in \Psi} \lambda_\rho \beta_\rho d_\rho \mathrm{Tr}(V_\rho(g)^* Y_\rho V_\rho(g)).$$
If $T \ge \sum_\rho n \lambda_\rho^2 \beta_\rho d_\rho^2 - \sqrt{n \log n}$ then we answer `$P_n$'; otherwise, `$Q_n$.' The definition of $T(g)$ is motivated by the computation of $\dd[P_n]{Q_n}$ in Section~\ref{sec:gsynch}; in fact, $T(g)$ is equal (up to constants) to $\dd[P_n(Y | g)]{Q_n(Y)}$. Note that this test is not computationally-efficient because it involves testing all possible solutions $g \in G^n$. The best computationally-efficient test that we know of is PCA (or AMP), which succeeds if and only if at least one $\lambda_\rho$ exceeds 1.

The proof of Theorem~\ref{thm:synch-upper} will require the following computation.
\begin{lemma}
Let $V$ be a fixed $nd \times d$ matrix where each $d \times d$ block is unitary of some type ($\RR,\CC,\mathbb{H}$). Let $W$ be an $nd \times nd$ Hermitian Gaussian matrix of the corresponding type (GOE, GUE, GSE, respectively). Let $\beta$ be $1,2,4$ (respectively) depending on the type. Then $\mathrm{Tr}(V^* W V) \sim \cN(0,2n^2d/\beta)$.
\end{lemma}
\begin{proof}
Let $u,v$ index the $d \times d$ blocks, and let $a,b,c$ index the entries within each block.
\begin{align*}
\mathrm{Tr}(V^* W V)
&= \sum_{u,v} \mathrm{Tr}\left( V_u^* W_{uv} V_v \right) \\
&= \sum_{u<v} 2 \,\mathrm{Tr}\,\mathfrak{Re}\left( V_u^* W_{uv} V_v \right) + \sum_u \mathrm{Tr}\left(V_u^* W_{uu} V_u\right)\\
&= \sum_{u < v} \sum_{a,b,c} 2 \,\mathfrak{Re}\left[(V_u^*)_{ab} (W_{uv})_{bc} (V_v)_{ca}\right] + \sum_u \sum_{a,b,c} (V_u^*)_{ab} (W_{uu})_{bc} (V_u)_{ca}\\
&= \sum_{u < v} \sum_{a,b,c} 2 \,\mathfrak{Re}\left[(V_u^*)_{ab} (W_{uv})_{bc} (V_v)_{ca}\right] + \sum_u \sum_{a,\,b<c} 2 \,\mathfrak{Re}\left[(V_u^*)_{ab} (W_{uu})_{bc} (V_u)_{ca}\right]\\
&\qquad\qquad + \sum_u \sum_{a,b} (V_u^*)_{ab} (W_{uu})_{bb} (V_u)_{ba}\\
&= \sum_{u < v} \sum_{b,c} 2 \,\cN(0, |\sum_a(V_u^*)_{ab} (V_v)_{ca}|^2/\beta) + \sum_u \sum_{b<c} 2 \,\cN(0,|\sum_a(V_u^*)_{ab} (V_u)_{ca}|^2/\beta)\\
&\qquad\qquad + \sum_u \sum_{b} \cN(0,2|\sum_a(V_u^*)_{ab} (V_u)_{ba}|^2/\beta)\\
&=  \cN(0, 2 \sum_{u,v} \sum_{b,c} |\sum_a(V_u^*)_{ab} (V_v)_{ca}|^2/\beta) \\
&=  \cN(0, 2 \sum_{u,v} \sum_{b,c} \sum_{a,a'} (\bar{V_u})_{ba} (V_v)_{ca} (\bar{V_v})_{ca'} (V_u)_{ba'} /\beta) \\
&=  \cN(0, 2 \sum_{u,v} \sum_{a,a'} \delta_{aa'} /\beta) \\
&=  \cN(0, 2 n^2 d /\beta). \\
\end{align*}
\end{proof}

\begin{proof}[Proof of Theorem~\ref{thm:synch-upper}]

We will now prove Theorem~\ref{thm:synch-upper} by showing that (given the condition in the theorem) the test $T = \max_g T(g)$ (defined above) succeeds with probability $1-o(1)$.
If $Y_\rho$ is drawn from the unspiked model $Q_n: \frac{1}{\sqrt{nd_\rho}} W$ then for any $g \in G^n$ we have
$$T(g) = \sum_\rho \lambda_\rho \beta_\rho d_\rho \frac{1}{\sqrt{nd_\rho}} \mathrm{Tr}(V_\rho(g)^* W V_\rho(g)) = \sum_\rho \lambda_\rho \beta_\rho d_\rho \frac{1}{\sqrt{nd_\rho}} \cN\left(0,\frac{2n^2 d_\rho}{\beta_\rho}\right) = \cN\left(0,\sum_\rho 2n\lambda_\rho^2 \beta_\rho d_\rho^2\right).$$
If instead $Y_\rho$ is drawn from the spiked model $P_n: Y_\rho = \frac{\lambda_\rho}{n} X_\rho X_\rho^* + \frac{1}{\sqrt{ nd_\rho}} W$ and we take $g$ to be the ground truth $g^*$ (so that $V_\rho(g) = X_\rho$), we have
$$T(g^*) = \sum_\rho \lambda_\rho \beta_\rho d_\rho \mathrm{Tr}\left(\frac{\lambda_\rho}{n} X_\rho^* X_\rho X_\rho^* X_\rho + \frac{1}{\sqrt{nd\rho}} V_\rho^* W V_\rho\right) = \sum_\rho n \lambda_\rho^2 \beta_\rho d_\rho^2 + \cN\left(0,\sum_\rho 2n\lambda_\rho^2 \beta_\rho d_\rho^2\right).$$
Using the Gaussian tail bound $\prob{\cN(0,\sigma^2) \ge t} \le \exp\left(\frac{-t^2}{2\sigma^2}\right)$, we have that under the spiked model,
$$P_n\left[T \le \sum_\rho n \lambda_\rho^2 \beta_\rho d_\rho^2 - \sqrt{n\log n}\right] \le \exp\left(\frac{-n \log n}{2}\left(\sum_\rho 2n \lambda_\rho^2 \beta_\rho d_\rho^2\right)^{-1}\right) = o(1).$$
Taking a union bound over all $L^n$ choices for $g \in G^n$, we have that under the unspiked model,
\begin{align*}
Q_n\left[T \ge \sum_\rho n \lambda_\rho^2 \beta_\rho d_\rho^2 - \sqrt{n\log n}\right] &\le L^n \exp\left(-\frac{1}{2} \left(\sum_\rho n \lambda_\rho^2 \beta_\rho d_\rho^2 - \sqrt{n\log n}\right)^2 \left(\sum_\rho 2n \lambda_\rho^2 \beta_\rho d_\rho^2\right)^{-1} \right) \\
&= \exp\left(n \log L - \frac{1}{4}\sum_\rho n \lambda_\rho^2 \beta_\rho d_\rho^2 + \mathcal{O}(\sqrt{n \log n})\right)
\end{align*}
which is $o(1)$ provided $\sum_\rho \lambda_\rho^2 \beta_\rho d_\rho^2 > 4 \log L$.
\end{proof}


We can simplify the statement of the theorem in the case where all frequencies are present. We note that if $\Psi_\mathrm{all}$ is the list of all frequencies then
$$\sum_{\rho \in \Psi_\mathrm{all}} \beta_\rho d_\rho^2 = L-1.$$
This follows from the ``sum-of-squares'' formula from the representation theory of finite groups. (The extra 1 comes from the fact that we don't use the trivial representation in our list. The factor of $\beta = 2$ for complex-type representations accounts for the fact that we only use one representation per conjugate pair. The factor of $\beta = 4$ for quaternionic-type representations accounts for the fact that the complex dimension is twice the quaternionic dimension.) We therefore have the following corollary.

\begin{corollary}
\label{cor:synch-upper-all}
Let $G$ be a finite group of order $L \ge 2$ and let $\Psi_\mathrm{all}$ be the list of {\bf all} frequencies (excluding the trivial one and only taking one from each conjugate pair). If
$$\lambda > \sqrt{\frac{4 \log L}{L-1}}$$
then a non-efficient algorithm can distinguish the spiked and unspiked models, and so $\GSynch(\lambda,G,\Psi_\mathrm{all})$ is {\bf not} contiguous to $\GSynch(0,G,\Psi_\mathrm{all})$.
\end{corollary}

\noindent Note that for large $L$ this differs from the lower bound of Theorem~\ref{thm:synch-finite} by a factor of $\sqrt{2}$. (We expect that the upper bound is asymptotically tight here and that the lower bound can be improved by the ``noise conditioning'' method recently introduced by \citet{noise-cond}.) Also note that the right-hand side matches Theorem~\ref{thm:toh-upper} (upper bound for the truth-or-Harr model); interestingly, both our lower and upper bounds indicate that the all-frequencies Gaussian model behaves like the truth-or-Haar model with $\lambda$ in place of $\tilde p$. In particular, we again see that a non-efficient algorithm can beat the spectral threshold once $L \ge 11$.

\section*{Acknowledgements}
The authors are indebted to Philippe Rigollet for helpful discussions and for many comments on a draft, and to Amit Singer and his group for discussions about synchronization.

\bibliography{main}

\appendix

\section{Proof of Theorem~\ref{thm:sphere-prior}}\label{app:confluent}
\begin{reptheorem}{thm:sphere-prior}
Consider the spherical prior $\cXs$. If $\lambda < 1$ then $\GWig(\lambda,\cXs)$ is contiguous to $\GWig(0)$.
\end{reptheorem}
\begin{proof}
By symmetry, we reduce the second moment above as
$$ \Ex_{x,x'} \exp\left(\frac{n \lambda^2}{2} \langle x,x' \rangle^2\right) = \Ex_{x} \exp\left(\frac{n \lambda^2}{2} \langle x,e_1 \rangle^2\right) = \Ex_{x_1} \exp\left(\frac{n \lambda^2}{2} x_1^2\right), $$
where $e_1$ denotes the first standard basis vector. Note that the first coordinate $x_1$ of a point uniformly drawn from the unit sphere in $\RR^n$ is distributed proportionally to $(1-x_1^2)^{(n-3)/2}$, so that its square $y$ is distributed proportionally to $(1-y)^{(n-3)/2} y^{-1/2}$. Hence $y$ is distributed as $\mathrm{Beta}(\frac12,\frac{n-1}{2})$. The second moment is thus the moment generating function of $\mathrm{Beta}(\frac12,\frac{n-1}{2})$ evaluated at $n \lambda^2/2$, and as such, we have
\begin{equation}\label{eq:confluent}  \Ex_{{Q}_n}\left(\dd[{P}_n]{{Q}_n}\right)^2 = {}_1 F_1\left( \frac12 ; \frac{n}{2} ; \frac{\lambda^2 n}{2} \right),  \end{equation}
where ${}_1 F_1$ denotes the confluent hypergeometric function.

Suppose $\lambda < 1$. Equation~13.8.4 from \citet{dlmf} grants us that, as $n \to \infty$,
\begin{align*}
{}_1 F_1\left( \frac12 ; \frac{n}{2} ; \frac{\lambda^2 n}{2} \right) &= (1+o(1)) \left(\frac{n}{2}\right)^{1/4} e^{\zeta^2 n/8} \left( \lambda^2 \sqrt{\frac{\zeta}{1-\lambda^2}} U(0,\zeta \sqrt{n/2}) \right. \\
&\qquad\qquad \left. + \left( -\lambda^2 \sqrt{\frac{\zeta}{1-\lambda^2}} + \sqrt{\frac{\zeta}{1-\lambda^2}} \right) \frac{U(-1,\zeta\sqrt{n/2})}{\zeta\sqrt{n/2}} \right),
\intertext{where $\zeta = \sqrt{2(\lambda^2-1-2\log \lambda)}$ and $U$ is the parabolic cylinder function,}
&= (1+o(1)) \left(\frac{n}{2}\right)^{1/4} e^{\zeta^2 n/8} \left( \lambda^2 \sqrt{\frac{\zeta}{1-\lambda^2}} e^{-\zeta^2 n / 8} (\zeta \sqrt{n/2})^{-1/2} \right.\\
&\qquad\qquad \left. + \left( -\lambda^2 \sqrt{\frac{\zeta}{1-\lambda^2}} + \sqrt{\frac{\zeta}{1-\lambda^2}} \right) \frac{e^{-\zeta^2 n/8} (\zeta \sqrt{n/2})^{1/2}}{\zeta\sqrt{n/2}} \right),
\intertext{by Equation~12.9.1 from \citet{dlmf},}
&= (1+o(1)) (1-\lambda^2)^{-1/2},\numberthis\label{eq:sphere-moment}
\end{align*}
which is bounded as $n \to \infty$, for all $\lambda < 1$. The result follows from Lemma~\ref{lem:sec}.
\end{proof}

\section{Proof of Propositions~\ref{prop:nong-lower-spherical} and~\ref{prop:nong-lower-iid}}
\label{app:nong-prior-conditions}

In this section we verify that the conditions of Theorem~\ref{thm:nongauss-lower} are satisfied for spherical and \iid priors.

\begin{repproposition}{prop:nong-lower-spherical}
Consider the spherical prior $\cXs$. Then conditions (i) and (ii) in Assumption~\ref{as:nong-lower} are satisfied.
\end{repproposition}

\begin{proof}

For the spherical prior we have $\lambda^*_\cXs = 1$, as computed in Theorem~\ref{thm:sphere-prior}.
Note that one can sample $x \sim \cXs$ by first sampling $y \sim \N(0,1)^n$ and then taking $x = y/\|y\|_2$. By Chebyshev, $\left| \|y\|_2^2 - n\right| < n^{3/4}$ with probability $1-o(1)$.
\begin{enumerate}[(i)]
\item Supposing that $\|y\|_2^2 > n - n^{3/4}$, which occurs with probability $1-o(1)$, we have
$$ \Pr[|x_u| \geq n^{-1/3}] \leq \Pr[|y_u| \geq n^{1/6}\sqrt{1-n^{-1/4}}] \leq e^{-n^{1/3}(1-n^{-1/4})/2} = o(1/n), $$
so that with probability $1 - o(1)$, we have for all $u$, $|x_u| < n^{-1/3}$.
\item We have $\|x\|_2 = 1$. For $q \in \{4,6,8\}$, $\|y\|_q^q$ has expectation $n(q-1)!!$ and variance $$n[(2q-1)!!-((q-1)!!)^2].$$ Supposing that $\|y\|_2^2 > n - n^{3/4} > n/2$, which occurs with probability $1-o(1)$, we have for any $\alpha_q$ that
\begin{align*}
\Pr[\|x_q\| > \alpha_q n^{\frac1q-\frac12}] &= \Pr[\|x\|_q^q > \alpha_q^q n^{1-\frac{q}2}] \\
&= \Pr[\|y\|_q^q > \alpha_q^q n^{1-\frac{q}2} \|y\|_2^q ] \\
&\leq \Pr[ \|y\|_q^q > \alpha_q^q 2^{-q/2} n] \\
&\leq \frac{n((2q-1)!! - ((q-1)!!)^2)}{n^2(2^{-q}\alpha_q^{2q} - (q-1)!!)^2},
\end{align*}
by Chebyshev. This probability is $o(1)$ so long as we take $\alpha_q^{2q} > 2^q (q-1)!!$. \qedhere
\end{enumerate}
\end{proof}

\begin{repproposition}{prop:nong-lower-iid}
Consider an \iid prior $\cX = \IID(\pi)$ where $\pi$ is zero-mean and unit-variance with $\EE[\pi^{16}] < \infty$. Then conditions (i) and (ii) in Assumption~\ref{as:nong-lower} are satisfied.
\end{repproposition}
An immediate implication of this is that conditions (i) and (ii) are also satisfied for a `conditioned' prior which draws $x$ from $\IID(\pi)$ but then outputs zero if a 'bad' event occured.

\begin{proof}
We have $x_i = \frac{1}{\sqrt n} \pi_i$ where $\pi_i$ are independent copies of $\pi$.
To prove (i),
$$\prob{|x_i| \ge n^{-1/3}} = \prob{|\pi_i| \ge n^{1/6}} = \prob{\pi_i^8 \ge n^{4/3}} \le \frac{\EE[\pi^8]}{n^{4/3}} = \mathcal{O}(n^{-4/3})$$
using Markov's inequality and $\EE[\pi^8] \le 1 + \EE[\pi^{16}] < \infty$. The proof follows by a union bound over all $n$ coordinates.

To prove (ii), for $q \in \{2,4,6,8\}$,
\begin{align*}
\prob{\|x\|_q > \alpha_q n^{\frac{1}{q}-\frac{1}{2}}}
&= \prob{\|x\|_q^q > \alpha_q^q n^{1-\frac{q}{2}}} = \problr{\sum_i x_i^q > \alpha_q^q n^{1-\frac{q}{2}}} \\
&= \problr{\sum_i \pi_i^q > \alpha_q^q n} = \problr{\sum_i \pi_i^q - n \EE[\pi^q] > (\alpha_q^q - \EE[\pi^q]) n}. \\
\intertext{Choose $\alpha_q$ so that $C \equiv \alpha_q^q - \EE[\pi^q] > 0$, and apply Chebyshev's inequality:}
&\le \frac{\mathrm{Var}[\sum_i \pi_i^q]}{C^2 n^2} = \frac{n \mathrm{Var}[\pi^q]}{C^2 n^2} = \mathcal{O}(1/n).
\end{align*}
Here we needed $\EE[\pi^{2q}] < \infty$ so that $\mathrm{Var}[\pi^q] < \infty$.
\end{proof}

\section{Proof of Proposition~\ref{prop:matrix-opt}}
\label{app:matrix-opt}

In this section we prove Proposition~\ref{prop:matrix-opt}, which we restate here for convenience.
\begin{repproposition}{prop:matrix-opt}
For $L \ge 2$,
$$\sup_\alpha \frac{L}{2} \frac{\left(\sum_{h \in G} \alpha_h^2 - \frac{1}{L}\right)}{D(\alpha,\bar\alpha)} = \frac{LC}{2}$$
where
$$C = \frac{L-2}{(L-1)\log(L-1)}.$$
Here $\alpha$ ranges over (vectorized) nonnegative $L \times L$ matrices with row- and column-sums equal to $\frac{1}{L}$. When $L = 2$, we define $C = 1$ (the limit value).
\end{repproposition}
Recall $G$ is a finite group of order $L$, $\bar\alpha = \frac{1}{L^2} \one_{L^2}$ and $\alpha_h = \sum_{(a,b) \in S_h} \alpha_{ab}$ where $S_h = \{(a,b)\,|\,a^{-1}b = h\}$. $D$ denotes the KL divergence, which in this case is
$$D(\alpha,\bar\alpha) = \sum_{ab} \alpha_{ab} \log(L^2 \alpha_{ab}) = 2\log L + \sum_{ab} \alpha_{ab}\log(\alpha_{ab}).$$

Although $\alpha$ belongs to a compact domain, we write $\sup$ rather than $\max$ in the optimization above. This is because when $\alpha = \bar\alpha$, the numerator and denominator of are both zero, so we are really optimizing over $\alpha \ne \bar\alpha$.

A high-level sketch of the proof is as follows. First we observe that the optimal $\alpha$ value should be constant on each $S_h$, allowing us to reduce the problem to only the variables $\alpha_h$. By local optimality, we show further that the optimal $\alpha$ should take a particular form where $\alpha_h = x$ for $k$ out of the $L$ group elements $h$, and $\alpha_h = y$ for the remaining ones (where $y = \frac{1-kx}{L-k}$ so that $\sum_h \alpha_h = 1$ as required). This allows us to reduce the problem to only the variables $k$ and $x$. We then show that for a fixed $k$, the optimum value is $\frac{LC_k}{2}$ where
$$C_k = \frac{L-2k}{k(L-k) \log\left(\frac{L-k}{k}\right)}$$
(defined to equal its limit value $\frac{2}{L}$ when $k = L/2$). Finally, we show that $C_k$ is largest when $k = 1$, in which case we have $C_1 = C$ and the proof is complete.

Now we begin the proof in full detail. Note that the numerator of the optimization problem depends only on the sums $\alpha_h$ and not the individual entries $\alpha_{ab}$. Furthermore, once we have fixed the $\alpha_h$'s, the denominator is minimized by setting all the $\alpha_{ab}$ values equal within each $S_h$. (Think of the fact that the uniform distribution maximizes entropy.) Therefore we only need to consider matrices $\alpha$ that are constant on each $S_h$. Note that any such matrix has row- and column-sums equal to $1/L$ (since each row or column contains exactly one entry in each $S_h$), so we can drop this constraint. (Interestingly, the fact that this constraint doesn't help means that we do not actually benefit from conditioning away from `bad' events in this case.) The denominator becomes
$$D(\alpha,\bar\alpha) = 2\log L + \sum_h L \cdot \frac{\alpha_h}{L} \log\left(\frac{\alpha_h}{L}\right) = \log L + \sum_{h \in G} \alpha_h \log(\alpha_h)$$
and so we have a new equivalent optimization problem:
$$\sup_{\alpha\ne\bar\alpha} M(\alpha)$$
where
$$M(\alpha) = \frac{L}{2} \cdot \frac{\sum_{h \in G} \alpha_h^2 - \frac{1}{L}}{\log L + \sum_h \alpha_h \log(\alpha_h)}.$$
Now $\alpha$ is simply a vector of $\alpha_h$ values, with the constraints $\alpha_h \ge 0$ and $\sum_h \alpha_h = 1$. Accordingly, $\bar\alpha_h = \frac{1}{L}$ for all $h$.

We will show that the optimum value is $\frac{LC}{2}$. We first focus on showing one direction: $\sup_{\alpha \ne \bar\alpha} M(\alpha) \le \frac{LC}{2}$. By multiplying through by the denominator of $M(\alpha)$ (which is positive for all $\alpha \ne \bar\alpha$ since it is the divergence), this is equivalent to
$$\max_\alpha T(\alpha) \le 0$$
where
$$T(\alpha) = \sum_h \alpha_h^2 - \frac{1}{L} - C \left[\log L + \sum_h \alpha_h \log(\alpha_h) \right].$$
Note that $\alpha \ne \bar\alpha$ is no longer required (since $T(\bar\alpha) = 0$) and so we now have a maximization problem over a compact domain. We will restrict to values of $\alpha$ that are locally optimal for $T(\alpha)$. Compute partial derivatives:
$$\frac{\partial T}{\partial \alpha_h} = 2\alpha_h - C \left[\log(\alpha_h) + 1\right]$$
$$\frac{\partial^2 T}{\partial \alpha_h^2} = 2 - \frac{C}{\alpha_h}.$$
Note that $\frac{\partial T}{\partial \alpha_h} \to \infty$ as $\alpha_h \to 0^+$ (and this is the only place in the interval $[0,1]$ where the derivative blows up), and so a maximizer $\alpha$ for $T(\alpha)$ should have no coordinates set to zero. $\frac{\partial T}{\partial \alpha_h}$ is decreasing when $\alpha_h < \frac{C}{2}$, and increasing when $\alpha_h > \frac{C}{2}$. In particular, $\frac{\partial T}{\partial \alpha_h}(\alpha_h)$ is (at most) 2-to-1. If some coordinate of $\alpha$ is 1 then the rest would have to be 0, which we already ruled out. Therefore, all coordinates of a maximizer are strictly between 0 and 1, which means $\frac{\partial T}{\partial \alpha_h}$ must be equal for all coordinates. Since the derivative is 2-to-1, this means a maximizer can have at most two different $\alpha_h$ values.

We can therefore restrict to $\alpha$ for which $k$ out of the $L$ coordinates have the value $x$, and the remaining $L-k$ coordinates have the value $y = \frac{1-kx}{L-k}$ (since the sum of coordinates must be 1). Therefore it is sufficient to show
$$\min_{1 \le k \le L/2} \;\min_{0 \le x \le 1/k} T_k(x) \ge 0$$
where
$$T_k(x) = C \left[\log L + kx \log x + (1-kx)\log\left(\frac{1-kx}{L-k}\right) \right] - \left[kx^2 + \frac{(1-kx)^2}{L-k} - \frac{1}{L}\right].$$
Although it only makes sense for $k$ to take integer values, we will show that the above is still true when $k$ is allowed to be any real number in the interval $[0,L/2]$.

Define
$$t_k(x) = C_k \left[\log L + kx \log x + (1-kx)\log\left(\frac{1-kx}{L-k}\right) \right] - \left[kx^2 + \frac{(1-kx)^2}{L-k} - \frac{1}{L}\right].$$
Note that this is the same as $T_k(x)$ but with $C$ replaced by $C_k$ (defined above). In the following two lemmas we will show $\min_{k,x} t_k(x) \ge 0$ and $C_k \le C_1 = C$ for all $k$. It follows that $T_k(x) \ge t_k(x)$ (since the coefficient of $C$ in $T_k(x)$ is the KL divergence, which is nonnegative). This completes the proof of the upper bound $\sup_{\alpha \ne \bar\alpha} M(\alpha) \le \frac{LC}{2}$ because
$$\min_{k,x} T_k(x) \ge \min_{k,x} t_k(x) \ge 0.$$

\begin{lemma}
For any $k \in [1,L/2]$, we have
$$\min_{x \in [0,1/k]} t_k(x) \ge 0.$$
\end{lemma}
\begin{proof}
We relax $k$ to be a real number in the interval $(0,L/2)$. The $k = L/2$ case will follow by continuity.
Compute the fourth derivative:
$$\frac{d^4 t_k}{dx^4} = C_k \left[\frac{2k}{x^3} + \frac{2k^4}{(1-kx)^3}\right] > 0.$$
Since the fourth derivative is strictly positive, the second derivative is convex. It follows that the first derivative $\frac{dt_k}{dx}$ has at most three zeros. One can check explicitly that these zeros are $\frac{1}{L} < \frac{1}{2k} < \frac{L-k}{kL}$. Using concavity of the second derivative, the middle zero $\frac{1}{2k}$ is a local maximum of $t_k(x)$ and the global minimum of $t_k(x)$ is achieved at either $\frac{1}{L}$ or $\frac{L-k}{kL}$. Both of these attain the value $t_k(x) = 0$, completing the proof.
\end{proof}

\begin{lemma}
For all $k \in [1,L/2]$, $C_k \le C_1 = C$.
\end{lemma}
\begin{proof}
We will show that $C_k$ is monotone decreasing in $k$ on the interval $(0,L/2)$, by showing that its derivative is negative. It then follows that we should take the smallest allowable value for $k$, i.e.\ $k = 1$. Compute the derivative:
$$\frac{dC_k}{dk} = \frac{L(L-2k) - (k^2 + (L-k)^2)\log\left(\frac{L-k}{k}\right)}{k^2(k-L)^2 \log^2\left(\frac{L-k}{k}\right)}.$$
The denominator is positive, so it suffices to show that the numerator is negative. Applying the bound $\log(x) < 2\left(\frac{x-1}{x+1}\right)$, valid for all $x \ge 1$, we see that the numerator is at most $-\frac{(L-2k)^3}{L} < 0$.
\end{proof}

This completes the proof of the upper bound $\sup_{\alpha \ne \bar\alpha} M(\alpha) \le \frac{LC}{2}$. The matching lower bound is achieved by taking the $\alpha$ value corresponding to $k = 1$ and $x = \frac{L-1}{L}$. (For $L = 2$, this corresponds to the sigularity $\bar\alpha$, but the optimum is achieved in the limit $x \to \frac{L-1}{L} = \frac{1}{2}$.)

\section{Improved Wishart lower bound}
\label{app:wish-nc}

In this section we improve our lower bound for the spiked Wishart model. The proof is based on the ``noise conditioning'' technique recently introduced by \citet{noise-cond}. Our main result is the following strengthening of Theorem~\ref{thm:wishart-ld}(i).

\begin{theorem}
\label{thm:wish-nc}
Let $\cX$ be a spike prior supported on the unit sphere in $\RR^n$, with rate function $f_\cX$ which is finite on $(0,1)$. Let $\beta \ge -1$ and $\gamma > 0$. If $\beta^2/\gamma > (\lambda^*_\cX)^2$ and
\begin{equation}
\label{eq:nc-result}
\gamma f_\cX(t^2) > -\log(1+\beta) + \beta + \frac{1}{2} \log(c^*/t)
- \frac{1+\beta-t c^*}{1-t^2} + 1 \quad \forall t \in (0,1)
\end{equation}
where
\begin{equation}
\label{eq:c-star}
c^* = c^*(t) = \frac{1}{2t}\left(-(1-t^2) + \sqrt{(1-t^2)^2 + 4 t^2 (1+\beta)^2}\right),
\end{equation}
then $\Wish(\gamma,\beta,\cX) \contig \Wish(\gamma)$.
\end{theorem}

Note that for simplicity we are now assuming that the prior $\cX$ is supported exactly on the unit sphere. One consequence for the Rademacher prior is that PCA is optimal for all positive $\beta$:
\begin{corollary}
Let $\cX$ be the \iid Rademacher prior. If $0 \le \beta < \sqrt{\gamma}$ then $\Wish(\gamma,\beta,\cX) \contig \Wish(\gamma)$; in this setting, the conditions of Theorem~\ref{thm:wish-nc} can be shown to hold.
\end{corollary}

The rest of this section is devoted to proving Theorem~\ref{thm:wish-nc}. The main idea is to apply the conditioning method (see Section~\ref{sec:cond-method}) to a more intricate `good' event that depends jointly on the signal and noise (whereas previously we only conditioned on the signal).

Define $P_n$ and $Q_n$ as in Section~\ref{sec:wish}. For a vector $x \in \RR^n$ and an $n \times n$ matrix $Y$, define the `good' event $\Omega(x,Y)$ by $x^\top Y x \in [N(1+\beta)(1-\delta),N(1+\beta)(1+\delta)]$ where $\delta = \frac{\log n}{\sqrt n}$. Note that under $P$ (where $x$ is the spike and $Y$ is the Wishart matrix), $\Omega(x,Y)$ occurs with probability $1-o(1)$. Let $\tilde P_n$ be the conditional distribution of $P_n$ given $\Omega(x,Y)$.

Following equation (\ref{eq:wishart-dpdq}) in the second moment computation of Section~\ref{sec:wish}, we compute the noise-conditioned second moment:
\begin{align}
\label{eq:cond-2nd}
\Ex_{Q_n} \left(\dd[\tilde P_n]{Q_n}\right)^2 &= \Ex_{x,x'\sim \cX} \,\Ex_{Y \sim Q_n} (1+\beta)^{-N} \exp\left( \frac12  \,\frac{\beta}{1+\beta} N (x^\top Y x + x'^\top Y x') \right) \\
&= (1+o(1)) \Ex_{x,x'\sim \cX} \,\Ex_{Y \sim Q_n} (1+\beta)^{-N} \exp\left(\beta N \left(1+\frac{\Delta}{2} + \frac{\Delta'}{2}\right)\right) \one_{|\Delta| \le \delta} \one_{|\Delta'| \le \delta} \\
&\defeq (1+o(1)) \Ex_{x,x' \sim \cX} m(\langle x,x' \rangle) \nonumber
\end{align}
where $\Delta,\Delta'$ are defined by $x^\top Y x = N(1+\beta)(1+\Delta)$ and $x'^\top Y x' = N(1+\beta)(1+\Delta')$. We will see below that $m$ is indeed only a function of $\langle x,x' \rangle$.

\subsection{Interval $|\alpha| \in [\eps,1-\eps]$}

Let $\alpha = \langle x,x' \rangle$. Let $\eps > 0$ be a small constant (not depending on $n$), to be chosen later. First let us focus on the contribution from $|\alpha| \in [\eps,1-\eps]$, i.e.\ we want to bound
$$M_1 \defeq \Ex_\alpha \left[\one_{|\alpha| \in [\eps,1-\eps]} m(\alpha) \right].$$

For $Y \sim Q_n$ and with $x,x'$ fixed unit vectors, the matrix
$$\left(\begin{array}{cc} x^\top Y x & x^\top Y x' \\ x^\top Y x' & x'^\top Y x' \end{array}\right)$$
follows the $2 \times 2$ Wishart distribution with $N$ degrees of freedom and shape matrix $$\left(\begin{array}{cc} 1 & \alpha \\ \alpha & 1 \end{array}\right)$$ where $\alpha = \langle x,x' \rangle$ as above.

By integrating over $c = \frac{1}{N} x^\top Y x'$ and using the PDF of the Wishart distribution, we have
\begin{align*}
&m(\alpha) = \\
&\iiint (1+\beta)^2 \exp\Big\{N\Big[-\log(1+\beta) + \beta\left(1 + \frac{\Delta}{2} + \frac{\Delta'}{2}\right) + \left(\frac{1}{2} - \frac{3}{N}\right) \log((1+\beta)^2(1+\Delta)(1+\Delta')-c^2) \\
&- \frac{1}{1-\alpha^2}\left((1+\beta)\left(1+\frac{\Delta}{2}+\frac{\Delta'}{2}\right)-\alpha c\right) - \frac{1}{2} \log(1-\alpha^2) + \log(N/2) - \frac{1}{N}\log \Gamma_2(N/2)\Big]\Big\} \,\dee c\, \dee \Delta\, \dee \Delta'
\end{align*}
where the integration ranges over $|\Delta| \le \delta$, $|\Delta'| \le \delta$, and $|c| \le (1+\beta)\sqrt{(1+\Delta)(1+\Delta')}$.

Using $\delta = o(1)$ and applying Stirling's approximation to $\Gamma_2$, we have for $|\alpha| \in [\eps,1-\eps]$,
$$m(\alpha) \le \max_{|c| \le 1+\beta} (1+\beta)^2 \exp\Big\{N\Big[-\log(1+\beta) + \beta + \frac{1}{2} \log((1+\beta)^2-c^2)
- \frac{1+\beta-\alpha c}{1-\alpha^2} - \frac{1}{2} \log(1-\alpha^2) + 1 + o(1) \Big]\Big\}$$
where the $o(1)$ is uniform in $\alpha$.

We can solve explicitly for the optimal value $c^*$ for $c$:
\begin{equation}
\label{eq:c-star-id}
c^* (1-\alpha^2) = \alpha((1+\beta)^2 - c^2)
\end{equation}
and so
$$c^* = \frac{1}{2\alpha}\left(-(1-\alpha^2) + \sqrt{(1-\alpha^2)^2 + 4 \alpha^2 (1+\beta)^2}\right).$$
Using (\ref{eq:c-star-id}), the above becomes
$$m(\alpha) \le m_1(\alpha) \defeq (1+\beta)^2 \exp\Big\{N\Big[-\log(1+\beta) + \beta + \frac{1}{2} \log(c^*/\alpha)
- \frac{1+\beta-\alpha c^*}{1-\alpha^2} + 1 + o(1) \Big]\Big\}.$$

Due to the symmetry $c^*(-\alpha) = -c^*(\alpha)$ we have $m_1(-\alpha) = m_1(\alpha)$ and so it is sufficient to restrict to the positive $\alpha$ case. (Here we assume for convenience that the distribution of $\alpha$ is symmetric about zero, but the proof easily extends to the asymmetric case.) Also, $m_1(\alpha)$ is increasing on $[0,1]$. We have
\begin{align*}
\frac{1}{2} M_1 &\le \Ex_\alpha \left[\one_{\alpha \in [\eps,1-\eps]} m_1(\alpha) \right] \\
&= \int_0^\infty \problr{\one_{\alpha \in [\eps,1-\eps]} m_1(\alpha) \ge u} \dee u \\
&= \int_0^\infty \problr{\alpha \in [\eps,1-\eps] \text{ and } m_1(\alpha) \ge u} \dee u \\
&= m_1(\eps) \problr{\alpha \in [\eps,1-\eps]} + \int_{m_1(\eps)}^{m_1(1-\eps)} \problr{\alpha \in [\eps,1-\eps] \text{ and } m_1(\alpha) \ge u} \dee u.
\intertext{The change of variables $u = m_1(t)$ yields}
&= m_1(\eps) \problr{\alpha \in [\eps,1-\eps]} + \int_\eps^{1-\eps} \problr{\alpha \in [\eps,1-\eps] \text{ and } \alpha \ge t} m_1(t) \,O(N)\, \dee t \\
&\le m_1(\eps) \problr{\alpha \ge \eps} + O(N) \int_\eps^{1-\eps} \problr{\alpha \ge t} m_1(t) \dee t.
\end{align*}

\noindent Plugging in the rate function to bound $\prob{\alpha \ge t}$, we obtain $M_1 = o(1)$ provided that (\ref{eq:nc-result}) holds.

\subsection{Interval $|\alpha| \in [0,\eps)$}

This case needs special consideration because (\ref{eq:nc-result}) does not hold with strict inequality at $t=0$ and so the last step above requires $\alpha$ to be bounded away from 0. As in the proof of Theorem~\ref{thm:wishart-ld} the contribution $M_2 \defeq \Ex_\alpha \left[\one_{|\alpha| \in [0,\eps)} m(\alpha) \right]$ is bounded as $n\to \infty$ provided that $\beta^2/\gamma < (\lambda^*_\cX)^2$ and $\eps$ is small enough. This step does not need to use the conditioning on $\tilde P_n$ and simply reverts back to the basic second moment, which is only larger (for each value of $\alpha$).

\subsection{Interval $|\alpha| \in (1-\eps,1]$}

This case needs special consideration because in the calculations for the $[\eps,1-\eps]$ interval, certain terms in the exponent blow up at $|\alpha|=1$ which prevents us from replacing $\Delta,\Delta'$ by an error term that is $o(1)$ uniformly in $\alpha$. To deal with this case we will bound $m(\alpha)$ by its worst-case value $m(1)$.

To see that $m(1)$ is the worst case, notice from (\ref{eq:cond-2nd}) that up to an $\exp(o(N))$ factor (which will turn out to be negligible), $m(\alpha)$ is proportional to $\prob{|\Delta| \le \delta \text{ and } |\Delta'| \le \delta}$. Since $x^\top Y x$ and $x'^\top Y x'$ each follow at $\chi_N^2$ distribution (with correlation that increases with $|\alpha|$), this probability is maximized when they are perfectly correlated at $|\alpha| = 1$.

We now proceed to bound $m(1)$. Let $|\alpha| = 1$, let $Y \sim Q_n$, and let $x,x'$ be fixed unit vectors. We have that $x^\top Y x$ follows a $\chi_N^2$ distribution, with $x'^\top Y x' = x^\top Y x$. Similarly to the computation for $[\eps,1-\eps]$ we obtain
$$m(1) \le m_3 \defeq (1+\beta) \exp\left\{N\left[-\frac{1}{2} \log(1+\beta) - \frac{1}{2}(1-\beta) + \frac{1}{2} + o(1)\right]\right\}$$
and
$$M_3 \defeq \Ex_\alpha \left[\one_{|\alpha| \in (1-\eps,1]} m(\alpha) \right] \le \exp(o(N)) \prob{|\alpha| \ge 1-\eps} m_3.$$
Plugging in the rate function, $M_3$ is $o(1)$ provided that $\gamma f((1-\eps)^2) > -\frac{1}{2}\log(1+\beta) - \frac{1}{2}(1-\beta) + \frac{1}{2}$. This follows from (\ref{eq:nc-result}) (near $t=1$) provided $\eps$ is small enough.

\end{document}